\documentclass[11pt, reqno]{amsart}

\newtheorem{theorem}{Theorem}[section]
\newtheorem{lemma}[theorem]{Lemma}
\newtheorem{Corollary}[theorem]{Corollary}
\newtheorem{proposition}[theorem]{Proposition}

\newtheorem{Conjecture}[theorem]{Conjecture}
\theoremstyle{definition}
\newtheorem{definition}[theorem]{Definition}
\newtheorem{example}[theorem]{Example}

\newtheorem{Condition}[theorem]{Condition}

\usepackage{amsmath,amssymb,amsthm,mathscinet}
\usepackage{enumerate}
\usepackage{easybmat,graphics}
\usepackage{hyperref}
\hypersetup{colorlinks=true,citecolor=blue,linkcolor=black}
\usepackage{etex}
\usepackage[all]{xy}
\usepackage[mathscr]{euscript}

\newtheorem{Remark}[theorem]{Remark}
\def\depth{\operatorname{depth}}
\def\max{\operatorname{max}}
\def\rec{\operatorname{rec}}
\def\det{\operatorname{det}}
\def\Ker{\operatorname{Ker}}

\def\supp{\operatorname{supp}}
\def\Ad{\operatorname{ Ad}}
\def\WD{\operatorname{ WD}}

\def\Ind{\operatorname{ Ind\, }}
\def\Hom{\operatorname{ Hom }}
\def\ind{\operatorname{ ind\, }}
\def\Char{\operatorname{char}}
\def\cond{\operatorname{ cond }}
\def\imod{\operatorname{ - mod }}
\def\vol{\operatorname{ vol }}
\def\Unr{\operatorname{Unr}}
\def\Span{\operatorname{Span}}
\def\Gal{\mathop{\rm Gal}\nolimits}
\def\GSp{\mathop{\rm GSp}\nolimits}
\def\Sp{\mathop{\rm Sp}\nolimits}
\def\GSpin{\mathop{\rm GSpin}\nolimits}
\def\GSO{\mathop{\rm GSO}\nolimits}
\def\SO{\mathop{\rm SO}\nolimits}
\def\GL{\mathop{\rm GL}\nolimits}
\def\SL{\mathop{\rm SL}\nolimits}
\def \Del{\mathop{\rm Del}\nolimits}
\def \Tr{\mathop{\rm Tr}\nolimits}
\def \Kaz{\mathop{\rm Kaz}\nolimits}
\def \cl{\mathop{\rm cl}\nolimits}
\def \Spec{\mathop{\rm Spec}\nolimits}

\def \Sym{\mathop{\rm Sym}\nolimits}
\def \Sim{\mathop{\rm sim}\nolimits}
\def \Spin{\mathop{\rm Spin}\nolimits}
\def \Std{\mathop{\rm Std}\nolimits}
\def\LG{\mathop{\mathord{{}^LG}}}
\numberwithin{equation}{section}

\newcommand{\calo}{\mathcal{O}}
\newcommand{\calp}{\mathfrak{p}}
\newcommand{\cala}{\mathfrak{a}}
\newcommand{\calu}{{\mathcal{U}}}
\newcommand{\calv}{{\mathcal{V}}}
\newcommand{\calw}{{\mathcal{W}}}
\newcommand{\calb}{\mathcal{B}}

\newcommand{\calx}{\mathcal{X}}
\newcommand{\calA}{\mathcal{A}}
\newcommand{\calc}{\mathcal{C}}
\newcommand{\gder}{{\bf G}^{\text{der}}}
\newcommand{\G}{\mathcal{G}}
\newcommand{\fR}{\mathfrak{R}}
\newcommand{\fH}{\mathscr{H}}
\newcommand{\fn}{\mathfrak{n}}
\newcommand{\fT}{\mathscr{T}}

\newcommand{\bfm}{{\bf M}}
\newcommand{\bfb}{{\bf B}}
\newcommand{\bfg}{{\bf G}}
\newcommand{\bfj}{{\bf J}}
\newcommand{\bfp}{{\bf P}}
\newcommand{\bfn}{{\bf N}}
\newcommand{\bfa}{{\bf A}}
\newcommand{\bft}{{\bf T}}

\newcommand{\bfu}{{\bf U}}
\newcommand{\bfi}{{\bf I}}
\newcommand{\bfx}{{ \bf X}}
\newcommand{\bfy}{{\bf Y}}
\newcommand{\lu}{{\bf u}}
\newcommand{\lw}{{\bf w}}

\newcommand{\bfz}{{\bf Z}}
\newcommand{\Z}{\mathbb{Z}}
\newcommand{\C}{\mathbb{C}}
\newcommand{\bG}{\mathbb{G}}

\newcommand{\R}{\mathbb{R}}
\newcommand{\A}{\mathbb{A}}
\newcommand{\F}{\mathbb{F}}
\newcommand{\tw}{\tilde{w}}
\newcommand{\ts}{\tilde{s}}

\newcommand{\mtheta}{{M_\theta}}
\newcommand{\momega}{{M_\Omega}}

\let\oldtocsection=\tocsection

\let\oldtocsubsection=\tocsubsection

\let\oldtocsubsubsection=\tocsubsubsection

\renewcommand{\tocsection}[2]{\hspace{0em}\oldtocsection{#1}{#2}}
\renewcommand{\tocsubsection}[2]{\hspace{1em}\oldtocsubsection{#1}{#2}}
\renewcommand{\tocsubsubsection}[2]{\hspace{2em}\oldtocsubsubsection{#1}{#2}}

\newlength{\margins}
\usepackage[top=\margins,bottom=\margins,left=\margins,right=\margins]{geometry}

\begin{document}
\begingroup
\title{The local Langlands correspondence for $\GSp_4$ over local function fields}
\author{Radhika Ganapathy}
\address{Department of Mathematics, The University of British Columbia, 1984 Mathematics Road, Vancouver, B.C. V6T 1Z2.}
\email{rganapat@math.ubc.ca}

\begin{abstract} 
We prove the local Langlands correspondence for $\GSp_4(F)$, where $F$ is a non-archimedean local field of positive characteristic with residue characteristic $>2$. 
\end{abstract}
\maketitle
\vspace*{-0.5cm}
\tableofcontents
\let\clearpage\relax

\section{Introduction}
The central result of local class field theory establishes an isomorphism between the multiplicative group $F^{\times}$ of a local non-archimedean field $F$ and the abelianization of the Weil group $W_F$ of $F$.
  This immediately gives  a bijection between the continuous characters of $F^{\times}$ and those of $W_F$.
  The local Langlands correspondence (LLC) is a vast set of conjectures that generalizes this ``one-dimensional version'' of class field theory. 
  This correspondence is a conjectural relation between the set of  irreducible, admissible representations of an algebraic group ${\bf G}(F)$ and the set of homomorphisms $\WD_F \rightarrow \mathord{{}^LG}$, where $\WD_F$ is the Weil-Deligne group of $F$ and $\mathord{{}^LG}$ is the Langlands dual group of ${\bf G}(F)$. 
 For the group $\GL_n(F)$, this correspondence is a bijection, whose existence has been established for non-archimedean local fields of characteristic 0 in \cite{HT01, Hen00} and more recently in \cite{Sch13}, and for local fields of positive characteristic in \cite{LRS93}. 
Recently Gan and Takeda \cite{GT11} proved the LLC for the group $\GSp_4(F)$ assuming that the characteristic of $F$ is 0. The purpose of this article is to study the relationship between the Deligne-Kazhdan correspondence and the LLC and, in particular, use it to establish the LLC for $\GSp_4(F')$, where $F'$ is a non-archimedean local field of positive characteristic with residue characteristic $>2$. 
The Deligne-Kazhdan correspondence can be summarized as follows.\\[4pt]
(a)  Given a local field $F'$ of characteristic $p$ and an integer $m \geq 1$, there exists a local field $F$ of characteristic 0 such that $F'$ is $m$-close to $F$, i.e., $\calo_F/\calp_F^m \cong \calo_{F'}/\calp_{F'}^m$.    \\[4pt]
(b) In \cite{Del84},  Deligne proved that if  $F$ and $F'$ are $m$-close, then
\[ \Gal(\bar{F}/F)/I_F^m \cong \Gal(\bar{F'}/F')/I_{F'}^m, \] where $\bar{F}$ is a separable algebraic closure of $F$, $I_F$ is the inertia subgroup and $I_F^m$ denotes the $m$-th higher ramification subgroup of $I_F$ with upper numbering. This gives a bijection
\begin{align*}
& \text{\{Cont., complex, f.d. representations of $\Gal(\bar{F}/F)$ trivial on $I_F^m$\}}\\
& \longleftrightarrow  \text{\{Cont., complex, f.d. representations of $\Gal(\bar{F'}/F')$ trivial on $I_{F'}^m$\}}. 
\end{align*}
Moreover, all of the above holds when $\Gal(\bar{F}/F)$ is replaced by $W_F$, the Weil group of $F$.\\[4pt]
(c) Let ${\bf G}$ be a split, connected reductive group defined over $\mathbb{Z}$ and let $G := {\bf G}(F)$. From now on, for an object $X$ associated to the field $F$, we will use the notation $X'$ to denote the corresponding object over $F'$. 
 In \cite{kaz86}, Kazhdan proved that  given $m \geq 1$, there exists $l \geq m$ such that if $F$ and $F'$ are $l$-close, then there is an algebra isomorphism $\Kaz_m:\mathscr{H}(G, K_m) \rightarrow \mathscr{H}(G', K_m')$, where $K_m $ is the $m$-th usual congruence subgroup of ${\bf G}(\calo_F)$. 
Hence, when the fields $F$ and $F'$ are sufficiently close, we have a bijection
 \begin{align*}
 &\text{\{Irreducible, admissible representations $(\sigma, V)$ of $G$ such that $\sigma^{K_m} \neq 0$\}} \\
  &\longleftrightarrow\text{\{Irreducible, admissible representations  $(\sigma', V')$ of $G'$ such that $\sigma'^{K_m'} \neq 0$\}}. 
\end{align*}
These results suggest that, if one understands the representation theory of $\Gal(\bar{F}/F)$ for all  local fields $F$ of characteristic 0, then one can use it to understand the representation theory of $\Gal(\bar{F}'/F')$ for a local field $F'$ of characteristic $p$, and similarly, with an understanding of the representation theory of $\bfg(F)$ for all local fields $F$ of characteristic 0, one can study the representation theory of $\bfg(F')$ in characteristic $p$. 

 A nice example based on this philosophy is the generalized Jacquet-Langlands correspondence.  This correspondence establishes a bijection between
the isomorphism classes of irreducible, admissible square integrable representations of $\GL_{rd}(F)$ and the isomorphism classes of irreducible, admissible square integrable representations of $\GL_r(D)$, where $D$ is a division algebra over $F$ of dimension $d^2$, and is uniquely characterized by certain properties. This theorem was first proved assuming that $\Char(F) = 0$ in \cite{Rog83} and \cite{DKV84}. Badulescu \cite{Bad02} then proved this in characteristic $p$. He first generalized the Kazhdan isomorphism to hold for the inner forms of $\GL_n$. He then studied square integrable representations and the characterizing properties of the Jacquet-Langlands correspondence over close local fields and combined it with the theorem in characteristic 0 to deduce the theorem in characteristic $p$.

 In the same manner, it is natural to expect that the Deligne-Kazhdan correspondence is compatible with the LLC, that is, loosely speaking, the following diagram 
 \begin{equation}\label{firstcommDiagram}
    \xymatrix{
         \{\text {Reps. $\sigma$ of } G \text{ with }\depth(\sigma) \leq m\} \ar[r]^-{\mbox{LLC}} \ar[d]_{\mbox{Kazhdan}}  &\{\phi: \WD_F \rightarrow \mathord{{}^LG}\text { with }\depth(\phi) \leq m \} \ar[d]^{\mbox{Deligne}} \\
       \{\text {Reps. $\sigma'$ of } G' \text{ with }\depth(\sigma') \leq m\} \ar[r]^-{\mbox{LLC}}      & \{\phi': \WD_{F'} \rightarrow \mathord{{}^LG}\text { with }\depth(\phi')\leq m \}}
\end{equation}
is commutative when the fields $F$ and $F'$ are sufficiently close. Here, the notion of depth on the left is defined via the Moy-Prasad filtration subgroups (\cite{MP94, MP96}) and the depth on the right is defined using the upper numbering filtration of the inertia subgroup (cf.  Section \ref{LLCGLNCLF} for precise definitions). 
As explained above, one way to prove that the above diagram is commutative is to study each of the characterizing properties of the LLC over close local fields and prove that these properties are compatible with the Deligne-Kazhdan correspondence. 
One of the key features of the LLC is that it matches the analytic $L$- and $\gamma$-factors of representations of $G \times \GL_r(F)$ with the Artin factors of the corresponding Langlands parameters. For $\GL_n$ the correspondence is in fact uniquely characterized by this property (cf. \cite{Hen93}).   On the analytic side, a theory of $L$- and $\gamma$-factors is available for generic representations via the Langlands-Shahidi method in great generality for reductive groups over local fields of characteristic 0 (\cite{Sha90}).
 In positive characteristic, such a theory is available for split classical groups in \cite{Lom09} and \cite{HL11}. All these works use the theory of local coefficients (available independent of characteristic) to inductively define the $\gamma$-factors. 
 These local coefficients arise from a  study of intertwining operators between certain  parabolically induced representations and the uniqueness of the Whittaker models of these induced representations (cf. \cite{Sha81}). In \cite{Del84}, Deligne proved that the Artin $L$- and $\epsilon$-factors remain the same for representations that correspond via the Deligne isomorphism. 
 In this paper, we prove an analogous result on the analytic side for representations of split reductive groups and use it to establish the LLC  for $\GSp_4(F')$ where $F'$ is a local field of odd positive characteristic. Let us summarize the results in this paper. 

 In \cite{How85}, Howe wrote down a presentation of the Hecke algebra $\fH(\GL_n(F), I_m)$, where $I_m$ denotes the Iwahori filtration subgroup of $\GL_n(F)$. Lemaire \cite{Lem01} used this presentation to prove that  if $F$ and $F'$ are $m$-close, then $\fH(\GL_n(F), I_m) \cong \fH(\GL_n(F'), I_m')$. He then used this variant of the Kazhdan isomorphism to transfer generic representations of $\GL_n$ over close local fields (\cite{Lem01}). This isomorphism is a useful variant of the Kazhdan isomorphism since, in order to study the action of the Hecke algebra on a representation, we only need to study the action of a simple list of explicitly described generators, which makes the question significantly easier to tackle. With the aid of structure theory of Chevalley groups, we first generalize the work of \cite{How85}. More precisely, let $\bfg$ be any split, connected reductive group defined over $\Z$. In Section \ref{HeckeAlgebraIso}, we write down a presentation for the Hecke algebra $\fH(G, I_m)$, where $I_m$ is the $m$-th Iwahori filtration subgroup of $G$. We then establish that $\fH(G, I_m) \overset {\zeta_m}\cong \fH(G', I_m')$ when $F$ and $F'$ are $m$-close. This Hecke algebra isomorphism enables us to prove a conjecture of Kazhdan. More precisely, in \cite{kaz86}, Kazhdan was able to construct a $\C$-isomorphism $\Kaz_m$ between these Hecke algebras $\fH(G, K_m)$ and $\fH(G', K_m')$ when the fields $F$ and $F'$ were just $m$-close. But in order to prove that $\Kaz_m$ is an algebra isomorphism, he needed the fields to be a few levels closer. He conjectured that $\Kaz_m$ should be an algebra isomorphism when the fields $F$ and $F'$ are just $m$-close. Using the fact that $I_m \subset K_m$, we see that $\fH(G, K_m)$ is a subalgebra of $\fH(G, I_m)$ and furthermore we prove that $\zeta_m|_{\fH(G, K_m)} = \Kaz_m$ as $\C$-maps. Now, using that $\zeta_m$ is an algebra isomorphism when the fields are just $m$-close, we obtain that $\Kaz_m$ is also an algebra isomorphism when the fields are just $m$-close, validating Kazhdan's conjecture. 

The map $\zeta_m$ gives a bijection between representations $(\sigma, V)$ of $G$ with $V^{I_m} \neq 0$,  and  representations $(\sigma', V')$ of $G'$ with $V'^{I_m'} \neq 0$. In Section \ref{PropertiesrepsCLF}, we study various properties of representations over close local fields.  In Section \ref{GR}, we follow the ideas of \cite{Lem01} and show that if $F$ and $F'$ are $(m+1)$-close and if $(\sigma, V)$ is a generic representation of $G$ with $V^{I_m} \neq 0$, then $(\sigma',V')$ is also generic where $(\sigma', V')$ corresponds to $(\sigma, V)$ via $\zeta_{m+1}$.  In Section \ref{IndRep}, we study induced representations from parabolic subgroups over close local fields. To state it more precisely, let $\bfb = \bft\bfu$ be a Borel subgroup of $\bfg$, $\Phi$ the set of roots of $\bft$, $\Phi^+$ the set of positive roots of $\bft$ in $\bfb$ and $\Delta \subset \Phi^+$ the set of simple roots. Fix a Chevalley basis $\{\lu_{\alpha}\,|\,\alpha \in \Phi\}$ for $\bfg$. Let $\theta \subset \Delta$ and let $\bfp_\theta= \bfm_\theta\bfn_\theta$  be the standard parabolic subgroup of $\bfg$ determined by $\theta$. Let $I_m$ be the $m$-th Iwahori filtration subgroup of $G$. Let $\sigma$ be an irreducible, admissible representation of $M_\theta$ with $\sigma^{I_m \cap M_\theta} \neq 0$. Let $l = m+3$. We show that if $F'$ is any local field that is $l$-close to $F$, then
\[(\Ind_{P_\theta}^G \sigma)^{I_m}\cong (\Ind_{P'_\theta}^{G'} \sigma')^{I_m'},\]
 where $\sigma'$ corresponds to $\sigma$ via the Hecke algebra isomorphism $\zeta_{l, M_\theta}: \fH(M_\theta, M_\theta \cap I_l) \overset{\cong}\rightarrow \fH(M_\theta', M'_\theta \cap I_l')$. 

In Section \ref{LocalCoefficients}, we study the Langlands-Shahidi local coefficients over close local felds. Assume $\bfm_\theta$ is maximal. Let $w_{l,\Delta}$ be the longest element of the Weyl group of $\bfg$ and let $w_{l, \theta}$ be the longest element of the Weyl group of $\bfm_\theta$. Set $w_0 = w_{l,\Delta}w_{l,\theta}$. Let $\chi =\displaystyle{\prod_{\alpha \in \Delta} \chi_\alpha\circ \lu_\alpha^{-1}}$ be a generic character of $U$ that is compatible with $\tw_0$ (cf. Section \ref{WF}). Here $\tw_0$ is a representative of $w_0$ in $G$ chosen using the fixed Chevalley basis for $\bfg$ (cf.  Section \ref{reps}). Let $(\sigma, V)$ be a $\chi|_{U\cap M_\theta}$-generic representation of $M_\theta$ with Whittaker model $\mathcal{W}(\sigma, \chi|_{U\cap M_\theta})$.   For $\nu \in X^*(\bfm_\theta) \otimes_\Z \R$, let $C_\chi(\nu, \sigma, \tw_0)$  be the Langlands-Shahidi local coefficient as in Theorem 3.1 of \cite{Sha81}  (see \cite{Lom09} for local fields of characteristic $p$).  Let $m$ be large enough so that $\cond(\chi_{\alpha}) \leq m$ for all $\alpha \in \Delta$, and additionally there exists $v \neq 0$ in $\sigma^{I_m \cap M_\theta}$ with the property that the Whittaker function $W_v \in\mathcal{W}(\sigma, \chi|_{U\cap M_\theta})$ satisfies $W_v(e) \neq 0$.  Set $l=m+4$. In Section \ref{LocalCoefficients} we prove that if $F'$ is a local field that is $l$-close to $F$, then 
\[C_\chi(\nu, \sigma, \tw_0) = C_{\chi'}(\nu, \sigma', \tw_0'),\]
where $\chi'$ is chosen ``compatible" with $\chi$ (cf. Section \ref{GR} for details), $(\sigma', V')$ is the $\chi'|_{U' \cap M_\theta'}$-generic representation of $M'_\theta$ that corresponds to $(\sigma, V)$ via $\zeta_{l, M_\theta}$, and $\tw_0'$ is a representative of $w_0$ in $G'$ chosen using the same Chevalley basis. We point out an application of this result to the Langlands-Shahidi method. In characteristic 0, Shahidi \cite{Sha90} used the theory of local coefficients to inductively define the $\gamma$-factors, and then, used a local-global argument  and combined it  with the theory of local factors at the archimedean places to prove the expected properties of the $\gamma$-factor and its uniqueness. This method is now being extended to include the case of local function fields, and difficulties arise in a local-global argument that is used to prove the various local properties of the $\gamma$-factor (including its local functional equation) and its uniqueness.  The uniqueness of $\gamma$-factors for classical groups over local function fields has been established in \cite{Lom09, HL11} using certain stability results and by proving the compatibility of the symmetric and exterior square $\gamma$-factors with the LLC for $\GL_n$. For split reductive groups $\bfg$, our result above can be useful in checking the various local properties of the $\gamma$-factors in positive characteristic (cf. Section \ref{LSMETHODGSP4} for some illustrations). For example, to prove the local functional equation of the $\gamma$-factors in positive characteristic, we observe that the $\gamma$-factors defined using local coefficients agree over close local fields and then combine it with the fact that the corresponding $\gamma$-factor in characteristic 0 satisfies the local functional equation.  

 The Plancherel measure can be seen as a coarser invariant than the $\gamma$-factor, and is defined for any irreducible, admissible representation of $M_\theta$. In fact, when the representation is additionally generic, the Plancherel measure can be expressed as a certain product of $\gamma$-factors.  Gan and Takeda \cite{GT11} use this to characterize the Langlands correspondence for non-generic supercuspidal representations  of $\GSp_4(F)$ in characteristic 0 (since one still does not have a theory of $\gamma$-factors for these representations). By using techniques similar to Section \ref{LocalCoefficients}, we prove in Section \ref{Plancherel} that, when the fields $F$ and $F'$ are sufficiently close,  the Plancherel measure \[\mu(\nu, \sigma, w_0) = \mu(\nu, \sigma', w_0)\] where $\sigma, \sigma', \nu, w_0$ are all as above. 

The final part of the paper deals with the applications of these results to the LLC. In section \ref{LLCGLNCLF}, we prove that Diagram \ref{firstcommDiagram} is commutative for the group $\GL_n$. It is known that the LLC for $\GL_n$ preserves the depth of the representation (Theorem 2.3.6.4 of  \cite{Yu09}). It is also easy to see that the Kazhdan and Deligne isomorphisms preserve their respective notions of depth, making Diagram \ref{firstcommDiagram} well-defined. The fact that it is commutative would follow by first re-characterizing the LLC for $\GL_n$ using  a stability theorem of \cite{DH81} and then applying the  main theorem of Section \ref{LocalCoefficients} and  the results of \cite{Del84}. 

Our final application of these results is for the group $\GSp_4(F')\cong \GSpin_5(F')$ where $F'$ is a local field of odd positive characteristic. We give a definition of local factors for generic representations of $\GSpin_5(F') \times \GL_t(F'), t \leq 2$, following \cite{Lom09}, and prove their compatibility over close local fields using the results of Section \ref{LocalCoefficients}. We prove that the LLC for $\GSp_4(F)$ in characteristic 0 preserves depth when the residue characteristic is odd. Then we define the Langlands parameter  for a representation $\sigma$ of $\GSp_4(F')$ of depth $\leq m$ in positive characteristic by choosing a sufficiently close local field of characteristic 0 (depending only on $m$) and forcing Diagram \ref{firstcommDiagram} to be commutative. We then prove that it is independent of the choices made, has the required properties, and is uniquely characterized by those properties. This is the content of Sections \ref{LSMETHODGSP4}  - \ref{LLCGSP4POSITIVE}.\vspace{6pt}
\begin{center}
\sc{Acknowledgements}\\[5pt]
\end{center}
 I am deeply  indebted to my thesis advisor Jiu-Kang Yu for suggesting this project and for the many stimulating conversations.  I am particularly grateful to him for providing the details of Section \ref{IGS} and for explaining the proof of Theorem \ref{stability} of this article. I am also indebted to Wee Teck Gan for explaining some of the ideas in Section \ref{DepthPreservation} of this article and for the many insightful questions and comments on this work. I thank Sungmun Cho, James Cogdell, Julia Gordon,  Luis Lomel\'i, Freydoon Shahidi, Lior Silberman, Shuichiro Takeda, Richard Taylor and Sandeep Varma for the mathematical conversations and for their advice/feedback on some aspects of this manuscript. I thank the referee for several helpful comments and suggestions.  A significant portion of this article was written during my stay at the Institute for Advanced Study, Princeton, in 2012 - 2013. I thank the Institute for their warm hospitality. 
 
 A version of the results in Sections \ref{DeligneKazhdanCorrespondenceA} - \ref{PropertiesrepsCLF} and Section \ref{LLCGLNCLF} of this article appeared as part of my Ph.D thesis \cite{Gan12}. Many of the results in Section \ref{LLCGLNCLF} have also been independently obtained in \cite{ABPS14}. The proof of Theorem \ref{RSGF} in the present article is given using the results in Section \ref{LocalCoefficients} on comparing the Langlands-Shahidi method over close local fields. This proof is different from the one provided in Theorem 2.3.10 of my thesis and Theorem 5.3 of \cite{ABPS14} (both these proofs  compare the Rankin-Selberg integrals for pairs over close local fields).  

\section{The Deligne-Kazhdan correspondence}\label{DeligneKazhdanCorrespondenceA}
Let $F$ be a non-archimedean local field with $\calo$ as its ring of integers, $\calp$ as its maximal ideal, $\pi$ as its uniformizer and $\mathfrak{f} = \calo/\calp$. Let $F'$ be another non-archimedean local field with $\calo'$, $\calp'$, $\pi'$, and $\mathfrak{f}'$ defined accordingly.
\begin{definition} Let $m \geq 1$. We say  that the fields $F$ and $F'$ are \textit{$m$-close} if there is a ring isomorphism $ \calo/\calp^m \rightarrow \calo'/\calp'^m$.
\end{definition}

A non-archimedean local field of characteristic $p$ can be viewed as a limit of non-archimedean local fields of characteristic 0.  More precisely, given a local field $F'$ of characteristic $p$ and an integer $m \geq 1$, we can always find a local field $F$ of characteristic 0 such that $F'$ is $m$-close to $F$. We just have to choose the field $F$ to be ramified enough.
\begin{example} The fields $\mathbb{F}_p((t))$ and $\mathbb{Q}_p\left(p^{1/m}\right)$ are $m$-close.
\end{example}

\subsection{Deligne's theory}\label{Delignetheory}

Let $m \geq 1$. Let $\bar{F}$ be a separable closure of $F$. Let $I_F$ be the inertia group of $F$ and  $I_F^m$ be its $m$-th  higher ramification subgroup with upper numbering (cf. Chapter IV of \cite{Ser79}). Let us summarize the results of Deligne \cite{Del84} that will be used later in this work.  Deligne considered the triplet $\Tr_m(F) = (\calo/\calp^m, \calp/\calp^{m+1}, \epsilon)$, where $\epsilon$ = natural\pagebreak[2] projection of $\calp/\calp^{m+1}$ on $\calp/\calp^m$, and proved that
\[\Gal(\bar{F}/F)/I_F^m,\]
together with its upper numbering filtration, is canonically determined by $\Tr_m(F)$. Hence an isomorphism of triplets $\Tr_m(F) \rightarrow \Tr_m(F')$ gives rise to an isomorphism
\begin{equation}\label{Deliso}
\Gal(\bar{F}/F)/I_F^m \xrightarrow{\Del_m} \Gal(\bar{F}'/F')/I_{F'}^m
\end{equation}
that is unique up to inner automorphisms (See Equation 3.5.1 of \cite{Del84}).  Here is a partial description of the map $\Del_m$ (Section 1.3 of \cite{Del84}).  Let $L$ be a finite totally ramified Galois extension of $F$ satisfying $I(L/F)^m = 1$ (here $I(L/F)$ is the inertia group of $L/F$). Then  $L = F(\alpha)$ where $\alpha$ is a root of an Eisenstein polynomial \[P(x) = x^n + \pi \sum a_ix^i\] for $a_i \in \calo$.  Let $a_i' \in \calo'$ be such that $a_i \mod \calp^m \rightarrow a_i' \mod \calp'^m$. So $a_i'$ is well-defined mod $\calp'^m$.
Then the corresponding extension $L'/F'$ can be obtained as $L' = F'(\alpha')$ where $\alpha'$ is a root of the polynomial \[P'(x) = x^n + \pi' \sum a_i'x^i\] where $\pi \mod \calp^m \rightarrow \pi' \mod \calp'^m$.  The assumption that $I(L/F)^m =1$ ensures that the extension $L'$ does not depend on the choice of $a_i'$ (Remark A.6.3 and A.6.4 of \cite{Del84}).

  Deligne proved some very interesting properties of the map $\Del_m$, which we list below. 
\begin{enumerate}[(i)]
\item From an isomorphism of triplets $\Tr_m(F) \cong \Tr_m(F')$ , we also obtain an isomorphism $F^{\times}/(1+\calp^m) \xrightarrow{\cl_m} F'^{\times}/(1+\calp'^m)$. Also, the map $\Del_m$ naturally induces an isomorphism between the abelianizations of the corresponding Galois groups.  These isomorphisms commute with local class field theory (LCFT), that is, the diagram
\begin{equation}\label{gl1Deligne}
    \xymatrix{
      (\Gal (\bar{F}/F)/I_F^m)^{ab}   \ar[r]^-{\Del_m} \ar[d]_-{\mbox{LCFT}} & (\Gal(\bar{F}'/F')/I_{F'}^m)^{ab} \ar[d]^-{\mbox{LCFT}} \\
   ( F^{\times}/(1+\calp^m))\widehat{   \,\,}\ar[r]_-{\cl_m}   & (F'^{\times}/(1+\calp'^m))\widehat{   \,\,} }
\end{equation}
is commutative, where $\,\widehat{   \,\,}\,$ denotes profinite completion  (Proposition 3.6.1 of \cite{Del84}).
\item The above properties hold when $\Gal(\bar{F}/F)$ is replaced by $W_F$, the Weil group of $F$, or more generally the Weil-Deligne group of $F$ (see Section 3.7 of \cite{Del84}).
\item Note that the isomorphism $\Del_m$ induces a bijection
\begin{align}\label{Delignebijection}
&\{\text{Isomorphism classes of representations of } \Gal(\bar{F}/F) \text{ trivial on } I_F^m  \}\nonumber\\
&\longleftrightarrow \{\text{Isomorphism classes of representations of } \Gal(\bar{F}'/F') \text{ trivial on } I_{F'}^m  \}.
\end{align} Let $\psi$ be a non-trivial additive character of $F$ and $k = \cond(\psi)$. Let $\psi'$ be a character of $F'$ that satisfies the following conditions:
\begin{itemize}
\item $\cond(\psi') = k$,
\item $\psi'|_{\calp'^{k-m}/\calp'^{k}} = \psi|_{\calp^{k-m}/\calp^{k}}$.
\end{itemize}
Let $(\phi,V)$ be a representation of $ \Gal(\bar{F}/F)$ trivial on $I_F^m$ and let $(\phi',V')$ be the representation of $ \Gal(\bar{F}'/F')$ obtained using Equation \eqref{Delignebijection}. Then their Artin $L$- and $\epsilon$-factors remain the same, that is,
\begin{align}\label{ArtinFactors}
 L(s, \phi)& = L(s, \phi'),\nonumber\\
\epsilon(s, \phi, \psi) &= \epsilon(s, \phi', \psi').
\end{align}
This is Proposition 3.7.1 of \cite{Del84}. If $(\phi_0,V_0)$ is an irreducible representation of $W_F$ then there is an unramified character $\chi$ of $W_F$ and a representation $(\phi, V)$ of $\Gal(\bar{F}/F)$ such that $\chi \otimes \phi_0 = \phi \circ i_F$, where $i_F: W_F \hookrightarrow \Gal(\bar{F}/F)$. Writing $\chi(Fr_F) = q^{-s(\chi)}$, with $Fr_F$ the Frobenius element, we have $L(s, \phi_0) = L(s - s(\chi), \phi)$ and  $\epsilon(s, \phi_0, \psi) = \epsilon(s-s(\chi), \phi, \psi)$.  Hence Equation \eqref{ArtinFactors} also holds for irreducible representations of $W_F$. Now it is easy to see that this property holds for semisimple representations of $\WD_F$ using Sections 2.3 - 2.4 of \cite{Hen02}. 
\end{enumerate}
Before recalling the work of Kazhdan \cite{kaz86}, let us fix some notation for the remainder of the paper. 
\subsection{Notation}\label{stdnotations}
Let {\bf G} be a split, connected reductive group defined over $\mathbb{Z}$ with $\gder$ its derived subgroup.  
 Let  {\bf B = TU} be a Borel subgroup of $\bf G$ with maximal torus $\bft$ and unipotent radical $\bfu$. 
Let $X^*({\bf T})$ (resp. $X_*({\bf T})$) be the character lattice (resp. cocharacter lattice),  $\Phi \subset X^*({\bf T})$ the set of roots of $\bft$ in $\bfg$,  $\Phi^+$ the set of positive roots of $\bft$ in $\bfb$ and $\Delta$ the set of simple roots. Let $\bfz$ denote the center of $\bfg$ and $N_\bfg(\bft)$ the normalizer of $\bft$ in $\bfg$. 

 There is a one-to-one correspondence between the parabolic subgroups of ${\bf G}$ containing ${\bf B}$ and subsets $\theta \subset \Delta$ written as follows $\theta \leftrightarrow \bfp_\theta = \bfm_\theta\bfn_\theta.$
 Let $\bfa_\theta$ be the connected component of $ \displaystyle{\cap_{\alpha \in \theta}} \Ker \alpha$. Note that $\bfm_\theta = Z_\bfg(\bfa_\theta)$ and $\bfa_\theta$ is the maximal split torus in the center of $\bfm_\theta$. The maximal torus in $\bfm_\theta$ can be taken to be $\bft = \bfa_\emptyset.$ 
 
 For a real vector space $V$, let $V^*$ denote its dual and $V_{\mathbb{C}}$ its complexification.
 Set $\cala_\theta^* = X^*( \bfm_\theta) \otimes_\mathbb{Z} \mathbb{R}$. The restriction map maps $X^*(\bfm_\theta)$ to a subgroup of $X^*(\bfa_\theta)$ of finite index, and therefore induces an isomorphism
 \[\cala_\theta^* \cong X^*(\bfa_\theta) \otimes_\Z \R.\]

 For $\theta \subset \Delta$, let $\Phi_\theta$ be the set of roots in the linear span of $\theta$, $\Phi_\theta^+ = \Phi^+ \cap \Phi_\theta$ and let $W_\theta$ be the Weyl group of $\bfm_\theta$ with respect to $\bft$. We write $W$ for $ W_\Delta$. 
 There is a natural inclusion $W_\theta \hookrightarrow W$. For $\alpha \in \Phi^+$, let $s_\alpha$ denote the reflection with respect to the root $\alpha$. Then $W= \langle s_\alpha|\alpha \in \Delta\rangle$ and  $W_\theta = \langle s_\alpha|\alpha \in \theta\rangle$.

We fix a Chevalley basis $\{\lu_\alpha\;|\; \alpha \in \Phi\}$ where  $\lu_\alpha: \mathbb{G}_a \rightarrow \bfu_\alpha$ (here $\bfu_\alpha$ denotes the root subgroup) is an isomorphism satisfying:

\begin{enumerate}[(1)]
\item For each $\alpha \in \Phi$ there is a $\Z$-homomorphism $\phi_\alpha: \SL_2 \rightarrow \bfg$ such that 
$\phi_\alpha\left(\begin{array}{cc}
1 & t \\
0 & 1 \\
\end{array}\right) = \lu_\alpha(t)$ and $\phi_\alpha\left(\begin{array}{cc}
1 & 0 \\
t & 1 \\
\end{array}\right) = \lu_{-\alpha}(t)$, and $\phi_\alpha\left(\begin{array}{cc}
t & 0 \\
0 & t^{-1} \\
\end{array}\right) = \alpha^{\vee}(t)$.
\item There exist universal structure constants $c^{\Z}_{\alpha, \beta, i, j} \in \Z$ ($\alpha, \beta \in \Phi, \; \alpha \neq -\beta, i\alpha +j \beta \in \Phi$ with $i,j >0$), such that for $s,t \in F$, the commutator of $\lu_\alpha(s)$ and $\lu_\beta(t)$ is given by:

\[ [\lu_\alpha(s), \lu_\beta(t)] = \displaystyle{\prod_{\substack{i\alpha +j\beta \in \Phi \\ i,j >0}}} \lu_{i\alpha+j\beta}(c_{\alpha, \beta, i,j} s^it^j)\]
where $c_{\alpha, \beta, i,j}$ is the corresponding element of $\calo$. 
\item Let $\lw_\alpha(t) =\phi_\alpha\left(\begin{array}{cc}
0 & t \\
-t^{-1} & 0 \\
\end{array}\right)$ for each $\alpha \in \Phi^+$.  Then \[ \lw_{\alpha}(t) = \lu_\alpha(t) \lu_{-\alpha}(-t^{-1})\lu_\alpha(t),\]  $\lw_\alpha(1)$ is a representative of the reflection $s_\alpha$ in $N_G(T)$, and there exist universal signs $\epsilon^\Z_{\alpha, \beta} = \pm 1$ depending only $\alpha$ and $\beta$ such that, for $t \in F$, 
\[\lw_\alpha(1) \lu_{\beta}(t) \lw_{\alpha}(1)^{-1} =  \lu_{s_\alpha(\beta)}(\epsilon_{\alpha, \beta}t)\]
 where $\epsilon_{\alpha, \beta}$ is the corresponding element of $\calo$. 

\end{enumerate}
Let $\bfg_\alpha: = \langle \bfu_\alpha, \bfu_{-\alpha} \rangle$. For an algebraic group  {\bf H} over $F$, let $H$ =  {\bf H}($F$). We then have $G$, $B$, $T$, $U$.  Let $U_{\alpha, \calo}:= \bfu_\alpha(\calo)$ and $U_{\alpha, \calp^m} = \Ker(\bfu_\alpha(\calo) \rightarrow \bfu_\alpha(\calo/{\calp^m}))$. Similarly, let $T_{\calp^m} = \Ker(\bft(\calo) \rightarrow \bft(\calo/{\calp^m}))$. 

Let $F'$ be another non-archimedean local field with $\calo', \calp', \pi'$ defined accordingly. Recall that for an object $X$ associated to the field $F$, we write $X'$ for the corresponding object over $F'$. Now the meaning of $G', B', T', U', T_{\calp'^m}', c_{\alpha, \beta, i,j}'$  and so on should be clear. 

 If the fields $F$ and $F'$ are $m$-close with ring isomorphism $\Lambda: \calo/\calp^m \rightarrow \calo'/\calp'^m$, then it is clear that $\Lambda(c_{\alpha, \beta, i,j} \mod \calp^m) = c_{\alpha, \beta, i, j}' \mod \calp'^m$. Similarly,  $\Lambda(\epsilon_{\alpha, \beta} \mod \calp^m) = \epsilon_{\alpha, \beta}' \mod \calp'^m$.  We will freely use this observation throughout this work.
\subsection{Kazhdan's theory}\label{KazhdanIsomorphism}
Let us recall the results of \cite{kaz86}. Let $\bfg$ be as above. Let $K_m = \Ker({\bf G}(\calo) \rightarrow {\bf G}(\calo/\calp^{m}))$ be the $m$-th usual congruence subgroup of $G$.
Fix a Haar measure $dg$ on $G$. Let \[t_x =\vol(K_m; dg)^{-1} \Char({K_m}xK_m),\] where $\Char(K_mxK_m)$ denotes the characteristic function of the coset $K_mxK_m$.  The set $\{ t_x| x \in G\}$ spans the $\C$-linear space $\mathscr{H}(G, K_m)$.   Let \[X_*({\bf T})_-=\{\lambda \in X_*({\bf T}) \,| \,\langle \alpha, \lambda\rangle \,\leq 0 \; \forall\;\alpha \in \Phi^+\}.\] Let $\pi_\lambda = \lambda(\pi)$ for $\lambda \in X_*({\bf T})_-$. Consider the Cartan decomposition of $G$:
\[ G = \displaystyle{\coprod_{\lambda \in X_*({\bf T})_-} {\bf G }(\calo)\pi_\lambda {\bf G}(\calo)}.\]
The set $ {\bf G }(\calo)\pi_\lambda{\bf G}(\calo)$ is a homogeneous space of the group $ {\bf G }(\calo)\times {\bf G}(\calo)$ under the action $(a,b).g = agb^{-1}$.
 The set $\{ {K_m}xK_m| x \in {\bf G }(\calo)\pi_\lambda {\bf G}(\calo) \}$ is then a homogeneous space of the finite group ${\bf G}(\calo/\calp^{m}) \times {\bf G}(\calo/\calp^{m})$. Let $\Gamma_\lambda \subset  {\bf G}(\calo/\calp^{m}) \times {\bf G}(\calo/\calp^{m})$ be the stabilizer of the double coset $K_m \pi_\lambda K_m$.
  Kazhdan  observed that an isomorphism $  {\bf G}(\calo/\calp^{m}) \times {\bf G}(\calo/\calp^{m}) \rightarrow  {\bf G}(\calo'/\calp'^{m}) \times {\bf G}(\calo'/\calp'^{m})$ (such an isomorphism would exist if the fields are $m$-close) maps $\Gamma_\lambda \rightarrow \Gamma_\lambda'$, where $\Gamma_\lambda'$ is the corresponding object for $F'$.
  Let $T_\lambda \subset {\bf G }(\calo)\times {\bf G}(\calo) $ be a set of representatives of $\left({\bf G}(\calo/\calp^{m})\times {\bf G}(\calo/\calp^{m})\right) /\Gamma_\lambda$. Then we have a bijection $T_\lambda \rightarrow T_\lambda'$. Kazhdan constructed a bijection of $\mathbb{C}$-vector spaces
\begin{align*} \mathscr{H}(G, K_m) \xrightarrow{\Kaz_m}& \mathscr{H}(G', K_m')\\
t_{a_i\pi_\lambda a_j^{-1}} \rightarrow& t_{a_i'\pi'_\lambda a_j'^{-1}}
 \end{align*}
where $\lambda \in X_*({\bf T})_-$, $(a_i,a_j) \in T_\lambda$ and $(a_i',a_j')$ is the corresponding element of $T_\lambda'$, and proved the following theorem.
\begin{theorem}[Theorem A of \cite{kaz86}] \label{Kaziso} Given $m \geq 1$, there exists $ \l \geq m$ such that if $F$ and $F'$ are $l$-close, the map $\Kaz_m$ constructed above is an algebra isomorphism.
\end{theorem}
Note that $\Kaz_m$ is an isomorphism of $\mathbb{C}$-vector spaces when the fields are just $m$-close. But in order to establish that this isomorphism is compatible with the Hecke algebra structures, he needed the fields to be a few levels closer. Kazhdan conjectured the following.

\begin{Conjecture}\label{Kazconj} In Theorem \ref{Kaziso}, we can take $l=m$.

\end{Conjecture}

An irreducible, admissible representation $(\sigma, V)$ of $G$ such that $\sigma^{K_m} \neq 0$ naturally becomes an $\mathscr{H}(G, K_m)$-module. Hence, if the fields $F$ and $F'$ are sufficiently close,  $\Kaz_m$ gives a  bijection
 \begin{align}\label{Kazreptrans}
 &\text{\{Iso. classes of irr., ad. representations $(\sigma, V)$ of $G$ with $\sigma^{K_m} \neq 0$\}}\nonumber \\
  & \longleftrightarrow\text{\{Iso. classes of irr., ad. representations  $(\sigma', V')$ of $G'$ with $\sigma'^{K_m'} \neq 0$\}}.
\end{align}

\subsubsection{A variant of the Kazhdan isomorphism for $\GL_n(F)$}\label{Kazvariant}

In \cite{How85}, Howe wrote down a presentation of  the Hecke algebra $\mathscr{H}(\GL_n(F), I_{m,n})$, where $I_{m,n}$ is the $m$-th filtration subgroup of the standard Iwahori subgroup $I_{0,n}$ of $\GL_n(F)$. Lemaire \cite{Lem01} used this presentation to prove the following.
\begin{enumerate}[(a)]
\item If $F$ and $F'$ are $m$-close, we have a Hecke algebra isomorphism \[\mathscr{H}(\GL_n(F), I_{m,n}) \cong \mathscr{H}(\GL_n(F'), I_{m,n}').\] This gives a bijection
 \begin{align*}\label{Howreptrans}
 &\text{\{Iso. classes of irr., ad. representations $(\sigma, V)$ of $\GL_n(F)$ with $\sigma^{I_{m,n}} \neq 0$\}}\\
  & \longleftrightarrow\\
  & \text{\{Iso. classes of irr., ad. representations  $(\sigma', V')$ of $\GL_n(F')$ with $\sigma'^{I_{m,n}'} \neq 0$\}}.
\end{align*}
\item Assume $F$ and $F'$ are $(m+1) $-close. Let $(\sigma, V)$ correspond to $(\sigma',V')$ as in (a). If $(\sigma,V)$ is $\psi$-generic, then $(\sigma',V')$ is $\psi'$-generic and $\cond(\sigma) = \cond(\sigma')$, where $\cond(\psi) = m$ and $\psi'$ is obtained from $\psi$ as in Section \ref{Delignetheory}.
\end{enumerate}
Note that (a) provides a very interesting variant of the Kazhdan isomorphism for $\GL_n(F)$. To prove (b),  the problem of understanding the action of Hecke algebra $\mathscr{H}(\GL_n(F), I_{m,n})$  on the  representation was reduced to understanding the action of the explicit list of generators described by Howe. This was a crucial advantage over the Kazhdan isomorphism in proving (b).

\section{The Hecke algebra $\mathscr{H}(G, I_m)$}\label{HeckeAlgebraIso}
We retain the notation of Section \ref{stdnotations}. Let $I$ be the standard Iwahori subgroup of $G$, defined as the inverse image under ${\bf G}(\calo) \rightarrow { \bf G} (\mathfrak{f})$ of ${\bf B}(\mathfrak{f})$. By Chapter 3 of \cite{Tit77}, there is a smooth affine group scheme $\bfi$ defined over $\calo$ with generic fiber $\bfg \times _\Z F$ such that $\bfi(\calo)  =I.$ Define 
$I_m : = \Ker(\bfi(\calo) \rightarrow \bfi(\calo/\calp^m)).$ Explicitly, $I = \left\langle U_{\alpha, \calo}, \bft(\calo), U_{-\alpha,\calp}\,|\, \alpha \in \Phi^+  \right\rangle$
 and $I_m = \left\langle U_{\alpha, \calp^m}, T_{\calp^m}, U_{-\alpha,\calp^{m+1}}\,|\, \alpha \in \Phi^+  \right\rangle$.
In this section we give a presentation for the Hecke algebra $\mathscr{H}(G, I_m)$. 
    Our presentation here is a generalization of Howe's presentation for $\GL_n$  in Chapter 3, \S{2} of \cite{How85}.
Let us first recall some results about the structure theory for reductive groups from \cite{Iwa65}. 

A \textit{generalized Tits system} is a triple $(G,\mathcal{I},N)$ where $\mathcal{I}$ and $N$ are subgroups of $G$ satisfying the  following properties:
\begin{enumerate}[(a)]
\item $T_0 = \mathcal{I}\cap N$ is a normal subgroup of $N$.
\item There exists a Coxeter group $W_S$ generated by a set of simple reflections $S$, a group $\Omega$, and an isomorphism $N/T_0 \cong W_S \rtimes \Omega$.
\item The following properties hold for elements of $S$:
\begin{itemize}
\item For any $w \in W_S \rtimes \Omega$ and $s \in S$, we have $w\mathcal{I}s \subset \mathcal{I}ws\mathcal{I} \cup \mathcal{I}w\mathcal{I}$.
\item For all $s \in S$, we have $s\mathcal{I}s^{-1} \neq \mathcal{I}$.
\end{itemize}
\item for $\rho \in \Omega$, we have $\rho S \rho^{-1} = S$, and $\rho \mathcal{I} \rho^{-1} = \mathcal{I}$.
\item $G$ is generated by $\mathcal{I}$ and $N$.
\end{enumerate}
The group $W_a = N/T_0 \cong W_S \rtimes \Omega$ is called the extended affine Weyl group of $(G,\mathcal{I},N)$. 

The group $G$ always admits a generalized Tits system. By the main theorem of \cite{Iwa65}, the triple $(G, \mathcal{I}, N)$ with $\mathcal{I} = I$ and $N = N_G(T)$,  satisfies all the conditions above. Note that $I \cap N = {\bf T}(\calo)$.  

There is another natural isomorphism associated to the group $W_a = N_G(T)/{\bf T}(\calo)$. Via the isomorphisms $W \cong N_{{\bf G}(\calo)}(T)/{\bf T}(\calo)$ and $X_*({\bf T}) \cong T/{\bf T}(\calo)$ we can realize these groups inside $W_a$ and in fact  $W_a \cong X_*({\bf T} ) \rtimes W$  where $W$ acts on $X_*({\bf T})$ in the obvious way. The length function on $W_a$ can be defined as follows. For $(\lambda, x) \in W_a$, 
\[l(\lambda, x) = \displaystyle{ \sum_{\alpha \in \Phi_1} |\langle\alpha, \lambda\rangle|} + \displaystyle{\sum_{\alpha \in \Phi_2}{ |\langle\alpha, \lambda\rangle + 1|}},\] where
$\Phi_1 = \{\alpha \in \Phi^+ \,|\, x^{-1}\alpha >0\}$ and $\Phi_2 = \{\alpha \in \Phi^+ \,|\, x^{-1}\alpha < 0\}$.
The groups $W_S$ and $\Omega$ have the following description.
We first write $\Phi = \Phi^{(1)} \cup \Phi^{(2)} \ldots \cup \Phi^{(p)}$ where each $\Phi^{(i)}$ is irreducible, and $\Delta = \Delta^{(1)} \cup \ldots \cup \Delta^{(p)}$ where $\Delta^{(i)}$ denotes the set of simple roots of $\Phi^{(i)}$. Let $\alpha_0^{(i)}$ denote the highest root of $\Phi^{(i)}$.  Let $s_\alpha$ denote the reflection with respect to the root $\alpha$.  Let $ s_0^{(i)} = ({-\alpha_0^{(i)}}^{\vee}, s_{\alpha_0^{(i)}}) \in X_*(\bft) \rtimes W$, and $S = S_1 \cup S_2$, where \[S_1 = \{s_\alpha| \alpha \in \Delta\}\; \text{ and }\; S_2= \{s_0^{(i)}| i = 1, 2\ldots p\}.\]
Let $\Delta_0 = \Delta_0^{(1)} \cup \ldots \cup\Delta_0^{(p)}$, where $\Delta_0^{(i)} = \Delta^{(i)} \cup \{ \alpha_0^{(i)}\}$. 
 Then $W_S$ is generated by $S$ and $W_S = Q^{\vee} \rtimes W $. Here $Q^{\vee}$ is the lattice generated by $\Phi^{\vee}$. The group $\Omega \cong X_*({\bf T})/Q^{\vee}$. With the length function as above, each element of $S$ has length 1 and the group $\Omega$ described above is equal to the set of elements of length $0$ in $W_a$.  Note that $\Omega$ is free when $\gder$ is simply connected.

Let $A$ be a set of representatives for $W_a$ in $N_G(T)$. Then, by Theorem 2.16 of \cite{IwaMat65} we know that $G = I A I$. Hence $G =\displaystyle{ \bigcup_ {w \in A,\\ x,y \in I}} I_m xwyI_m$. Fix a Haar measure $dg$ on $G$ such that $\vol(I_m; dg) = 1$. For $g\in G$, let $f_g$ denote the characteristic function of the double coset $I_mgI_m$. Then using the above decomposition we see that the set $ \{f_{xwy} \,|\,w \in A \text{ and } x,y \in I\}$ spans the $\C$-linear space $\mathscr{H}(G, I_m)$. We will make a suitable choice for the representatives of elements of $W_a$ before writing down the generators and relations for the Hecke algebra.

\subsection{Representatives}\label{reps}
Recall that we have fixed a  Chevalley basis $ \{\lu_{\alpha} | \alpha \in \Phi\}$. Also, $\lw_\alpha(1) = \lu_\alpha(1)\lu_{-\alpha}(-1)\lu_\alpha(1)$ and $\alpha^\vee(t) = \phi_\alpha\left(\begin{array}{cc}
t & 0 \\
0 & t^{-1} \\
\end{array}\right).$ Note that $\alpha^\vee(t) = \lw_\alpha(t)\lw_\alpha(-1)$.  Put $\tilde{s}_\alpha = \lw_{\alpha}(1)$ for $\alpha\in \Delta$ and $\tilde{s}_0^{(i)} = \lw_{\alpha_0^{(i)}}(\pi^{-1})$. Then $\tilde{s}_\alpha$ is a representative in $ N_G(T)$ of the reflection $s_\alpha$ for $\alpha \in \Delta$ and $\ts_0^{(i)}$ is a representative of $s_0^{(i)}$ in $N_G(T)$.
We now choose a set of representatives for the elements of $W_a$ as follows. 
\\
(a)\textbf{ Representatives for $W$}: Recall that $W = \langle S_1 \rangle$.  Let $ w= s_{1}\ldots s_{r}$ be a minimal decomposition of $w$ where $s_i \in S_1$. Then $\tilde{w} =  \tilde{s}_{1}\ldots \tilde{s}_{r}$ is a representative of $w$ in $N_G(T)$. By  Lemma 56 of \cite{Ste68} we see that this representative is independent of the choice of minimal decomposition of $w$. This way, we get a set of representatives for the elements of the Weyl group. \\
(b)\textbf{ Representatives for $\Omega$}: Note that $\Omega$ is a finitely generated abelian group. We fix an isomorphism
\[\Omega \cong \Z^l \times \Z/a_1\Z \times \Z/a_2\Z \ldots \Z/a_k\Z\]
with $a_1|a_2...|a_k$. Fix an ordered basis for $\Z^l$ and let $\rho_1, \ldots\rho_l$ be the corresponding elements of $\Omega$. Let $\mu_i$ be the element of $\Omega$ whose image under the above isomorphism  is  the generator of the cyclic group $\Z/a_i\Z$.  Then
\[\Omega= \{ \rho_1^{t_1}\rho_2^{t_2}\ldots \rho_l^{t_l} \mu_1^{r_1} \mu_2^{r_2} \ldots \mu_k^{r_k}| t_i \in \Z \;\forall\; 1 \leq i \leq l;\; 0 \leq r_j < a_j \;\forall\;1 \leq j \leq k\}\]
Let $\rho \in \{\rho_1, \rho_2, \ldots \rho_l, \mu_1, \mu_2, \ldots \mu_k\}$. Then $\rho$ 
 can be written as $(\lambda, x)$ for $\lambda \in  X_*(\bf{T} )$ and $x\in W$. 
As before, for $\lambda \in  X_*(\bf{T} )$ let $\pi_{\lambda} = \lambda(\pi)$. Let $\tilde{x}$ be the representative in $N_G(T)$ as described in (a). 
Then, $\tilde{\rho} = \pi_{\lambda}\tilde{x}$ is a representative of $\rho$ in $N_G(T)$. For $1 \leq i \leq l$, let $\tilde{\rho}_i^{-1} = \tilde{x}_i^{-1}.\pi_{\lambda_i}^{-1} = \pi_{-x_i^{-1}.\lambda_i}.\tilde{x}_i^{-1}$, the inverse of $\tilde{\rho}_i$ in $N_G(T)$, be the representative of $\rho_i^{-1}$. Finally, for an element $\rho \in \Omega$, write $\rho = \rho_1^{t_1}\rho_2^{t_2}\ldots \rho_l^{t_l} \mu_1^{r_1} \mu_2^{r_2} \ldots \mu_k^{r_k}$ and set
\[\tilde{\rho} : =\tilde{\rho}_1^{t_1}\ldots\tilde{\rho}_l^{t_l} \tilde{\mu}_1^{r_1} \ldots \tilde{\mu}_k^{r_k}\]
This gives us a set of representatives in $N_G(T)$ for the elements of $\Omega$.\\
(c)\textbf{ Representatives for $W_S$}: Let $\alpha, \beta \in \Delta_0$, $\alpha \neq \beta$. First, if $\alpha, \beta \in \Delta$, define $\theta_{\alpha, \beta}$ to be the angle between the simple roots $\alpha, \beta$. Note that if $\alpha \in \Delta_0^{(i)}$ and $\beta \in \Delta_0^{(j)}$ with $i \neq j$, then the angle between them is $\pi/2$. Suppose $\alpha = \alpha_0^{(i)}$ and $\beta$ is simple, define $\theta_{\alpha, \beta}$ to be the angle between $-\alpha_0^{(i)}$ and $\beta$. Then, it is well known that for $\alpha \neq \beta$, $ \theta_{\alpha, \beta} = \left(1-1/v\right)\pi, v= 2,3,4,6$. Let $m_{\alpha, \beta}$ denote the order of the element $s_\alpha s_\beta$. We simplify our notation a bit and write $m_{ij}$ as the order of the element $s_is_j$ for $s_i, s_j \in S$.  Then $m_{ij} = 2,3,4,6$ or infinite. Note that the case when $m_{ij}$ is infinite happens only when one of the irreducible components $\Phi^{(i)}$ is of rank 1 (in which case the highest root is parallel to the unique simple root - cf. Proposition 3 of  Chapter V, \S 3, No 4 of \cite{Bou02} and also Proposition 1.15 of \cite{IwaMat65}). Consequently, the only relations in $W_S$ are
 \begin{align}\label{A}
s_is_j &=  s_js_i  & if & & m_{ij} = 2 \nonumber\\
s_is_js_i  & =  s_js_is_j & if  & &m_{ij} =3\\
(s_is_j)^2 &= ( s_js_i)^2 & if &&  m_{ij} =4 \nonumber\\
(s_is_j)^3& = (s_js_i)^3 & if  && m_{ij} =6 \nonumber
\end{align} 
for $s_i, s_j \in S$. Moreover, if $s_{i_1}s_{i_2}\ldots.s_{i_r} = s_{j_1}s_{j_2}\ldots.s_{j_r}$ are two reduced expressions in $W_S$, then we can transform one to the other just using the relations above. We now prove the following lemma about their representatives. 
   
\begin{lemma}\label{si} Let $s_i, s_j \in S$. Then
\begin{align}\label{B}
 \tilde{s}_i\tilde{s}_j &= \tilde{s}_j\tilde{s}_i &if && m_{ij} =2 \nonumber\\
\tilde{s}_i\tilde{s}_j\tilde{s}_i &= \tilde{s}_j\tilde{s}_i\tilde{s}_j  & if &&   m_{ij} =3  \\
(\tilde{s}_i\tilde{s}_j)^2 &=( \tilde{s}_j\tilde{s}_i)^2 & if & & m_{ij}=4 \nonumber\\
(\tilde{s}_i\tilde{s}_j)^3 &=  ( \tilde{s}_j\tilde{s}_i)^3   & if &&  m_{ij}=6 \nonumber
\end{align} 

\end{lemma}
\begin{proof} By Corollary 5 (To Theorem 4') of \cite{Ste68}, we can assume that $\gder$ is simply connected. Furthermore, we can assume that $\gder$ is simple, that is, $\Phi$ is irreducible (Hence $S_2 = \{s_0\}$ and we drop all the superscripts in this proof). When $s_i, s_j \in S_1$, this is just Lemma 56 of \cite{Ste68}. So we only need to deal with the case when say $s_i \in S_1$ and $s_j \in S_2$.
  We will deal with the case $m_{ij} = 3$. The other cases follow similarly. So we are dealing with the case $s_i  = s_{\alpha_i}\in \Delta$ and $s_j = s_0$ with $s_is_0s_i  =  s_0s_is_0 $. 
  Let $ x = \tilde{s}_0\tilde{s}_i\tilde{s}_0\tilde{s}_i^{-1}\tilde{s}_0^{-1}\tilde{s}_i^{-1}$. We want to show $ x =1$. 
  Recall $\tilde{s}_0 = \lw_{\alpha_0}(\pi^{-1})  \text{and } \tilde{s}_i = \lw_{\alpha_i}(1)$. Then   $\tilde{s}_0^{-1} = \lw_{\alpha_0}(- \pi^{-1})  \text{ and } \tilde{s}_i^{-1} = \lw_{\alpha_i}(-1)$.
   Now, \[ \lw_{\alpha_i}(1)\lw_{\alpha_0}(\pi^{-1}) \lw_{\alpha_i}(-1) = \lw_{s_i(\alpha_0)}(c \pi^{-1})\] by Lemma 19(a) of \cite{Ste68}. Then, \[ \lw_{\alpha_0}(\pi^{-1}) \lw_{s_i(\alpha_0)}(c \pi^{-1}) \lw_{\alpha_0}(- \pi^{-1}) = \lw_{s_{\alpha_0}s_i(\alpha_0)}(c_1( \pi)^{<\alpha_0, s_i(\alpha_0)>} c\pi^{-1})\] by applying Lemma 19(a) of \cite{Ste68} again. 
   Since $s_{\alpha_0}s_i(\alpha_0) = \pm \alpha_i$, we see that the first 5 terms of $x$ lie in $G_{\alpha_i}$.
    Therefore $ x \in  G_{\alpha_i}$. 
    Considering the last 5 terms of $x$ and repeating the same argument, we see that $x \in G_{\alpha_0}$. Hence we conclude that $ x \in G_{\alpha_i} \cap G_{\alpha_0} \cap T$. This implies that $ x = \alpha_0^\vee(t) = \alpha_i^\vee(t')$ for some $ t, t' \in F$. Since $\alpha_0$ is the highest root, $ \alpha_0 = \displaystyle{\sum_{j=1}^{n}}m_j \alpha_j$, where the $m_j$'s are uniquely determined integers.
     By Lemma 19(c) of \cite{Ste68}, $\alpha_0^\vee(t) = \displaystyle{\prod_{\alpha_j \in \Delta}} \alpha_j^\vee(t^{m_j})$. Hence we have $ \displaystyle{\prod_{\alpha_j \in \Delta}} \alpha_j^\vee(t^{m_j}) = \alpha_i^\vee(t')$. Since $\gder$ is simply connected, we have $t^{m_i} = t'$ and $ t^{m_j} = 1$ for all $ j \neq i$ by Lemma 28 of \cite{Ste68}.  
     Using the explicit determination of $m_j$'s in Chapter 1, Section 2 of \cite{Mac03}, we see that, except when $\Phi$ is of type $B_n \text{ or } C_n$, there exist  $m_{j_1}$ and $m_{j_2}$, with $j_1, j_2 \neq i$, such that $ \operatorname{gcd}(m_{j_1}, m_{j_2}) = 1$. So $t = 1$ and hence $x = 1$ in all these cases. 
     We will deal with $B_n$ and $C_n$ explicitly. We first observe the following general fact: Let $\bft^{\text{der}}$ be the maximal torus in $\gder$, and $\alpha$ and $\beta$ be roots. Then $X^*(\bft^{\text{der}}) \overset{\alpha^{\vee}} \twoheadrightarrow X^*({\bf{G}}_\alpha \cap \bft^{\text{der}}) \cong \mathbb{Z}$. Moreover,
\begin{equation}\label{characters} X^*({\bf{G}}_\alpha \cap {\bf{G}}_\beta \cap \bft^{\text{der}}) \cong X^*({\bft^{\text{der}})}/(\Ker(\alpha^{\vee}) + \Ker(\beta^{\vee}))
\end{equation}
$ Case \; C_n$:  For type $C_n$, $\alpha_0 = 2\epsilon_1$. So we only need to deal with the case when $\alpha_i = 2\epsilon_n$ where $\epsilon_i$ are as in Pages 7-12 of \cite{Mac03}. Then $\alpha_0^{\vee} = \epsilon_1$ and $\alpha_i^{\vee} = \epsilon_n$. Here $X^*(\bft^{\text{der}}) =\displaystyle{ \oplus_{1 \leq i \leq n}} \mathbb{Z} \epsilon_i$. Clearly, $\Ker(\alpha_0^{\vee}) + \Ker(\alpha_i^{\vee}) =  X^*(\bft^{\text{der}})$. So we see that $G_{\alpha_0} \cap G_{\alpha_i} \cap T$ is trivial by \ref{characters}.\\
$Case \; B_n$:  For this type, we have $\alpha_0 = \epsilon_1 + \epsilon_2$. So we only need to deal with the case when $\alpha_i = \epsilon_1 - \epsilon_2$. Then $\alpha_0^{\vee} = \epsilon_1 + \epsilon_2, \alpha_i^{\vee} = \epsilon_1 - \epsilon_2$. Recall that \[X^*(\bft^{\text{der}}) = \left\{ \frac{1}{2}\sum a_i \epsilon_i | \text{ all } a_i \text{ have same parity}\right\}.\] Then, \[\Ker(\alpha_0^{\vee}) =  \left\{ \frac{1}{2}\sum a_i \epsilon_i | \text{ all } a_i \text{ have same parity and } a_1 = -a_2\right\},\] and \[\Ker(\alpha_i^{\vee}) =  \left\{ \frac{1}{2}\sum a_i \epsilon_i | \text{ all } a_i \text{ have same parity and } a_1 = a_2 \right\}.\] Now it is clear that $\Ker(\alpha_0^{\vee}) + \Ker(\alpha_i^{\vee}) =  X^*(\bft^{\text{der}}).$ So we are done by Equation \eqref{characters}.
\end{proof} 
By the lemma we see that if $w=  s_{i_1}\ldots s_{i_c}$ is a reduced expression for $w$ in $W_S$, then $\tilde{w} = \tilde{s}_{i_1}\ldots  \tilde{s}_{i_c}$, is a representative in $N_G(T)$ that is independent of the reduced expression of $w$.\\
(d)\textbf{ Representatives of $W_a$}: Now, for any element $w$ of $W_a$, let $w= s_{i_1} s_{i_2}\ldots.s_{i_c}\rho_1^{t_1}\ldots \rho_l^{t_l}\mu_1^{r_1}\ldots \mu_k^{r_k}$ be an expression for $w$ with $s_{i_1}\ldots.s_{i_c}$ reduced,  $t_i \in \Z$ and $0 \leq r_j < a_j$. Then, the length of $w$ is $c$. 
Define \[\tilde{w} =  \tilde{s}_{i_1}\ldots  \tilde{s}_{i_c}\tilde{\rho}_1^{t_1}\ldots \tilde{\rho}_l^{t_l}\tilde{\mu}_1^{r_1}\ldots \tilde{\mu}_k^{r_k}.\]
 Suppose $w =  s_{j_1}\ldots.s_{j_c}\rho_1^{e_1}\ldots \rho_l^{e_l}\mu_1^{f_1}\mu_2^{f_2} \ldots \mu_k^{f_k}$ is another reduced expression for $w$ with $0 \leq f_j <a_j$ and $e_i \in \Z$.  Then, since $W_S \cap \Omega ={1}$, we see that $ s_{i_1}\ldots.s_{i_c}=  s_{j_1}\ldots.s_{j_c}$ and $\rho_1^{k_1}\ldots \rho_l^{k_l}\mu_1^{r_1}\ldots \mu_k^{r_k}= \rho_1^{e_1}\ldots \rho_l^{e_l}\mu_1^{f_1}\mu_2^{f_2} \ldots \mu_k^{f_k}$.
  By Lemma \ref{si} we have $\tilde{s}_{i_1}\ldots  \tilde{s}_{i_c} =\tilde{s}_{j_1}\ldots  \tilde{s}_{j_c}$. Also, $t_i = e_i \;\forall\; 1 \leq i \leq l$. Moreover, since $0 \leq r_j, f_j <a_j$, we also have $r_j =f_j = 1 \;\forall\; 1 \leq j \leq k$. 
  Therefore $ \tilde{w} = \tilde{s}_{i_1}\ldots  \tilde{s}_{i_c}  \tilde{ \rho}_1^{t_1}\ldots \tilde{\rho}_l^{t_l}\tilde{\mu}_1^{r_1}\ldots \tilde{\mu}_k^{r_k}$ is well-defined.  
  Let $\tilde{W}_a$ denote this set of representatives of elements of  $W_a$.

\subsection{Generators and relations}
Recall the following proposition.
\begin{proposition}\label{vol}
Let G be a locally compact unimodular group and $ H \subset G$ an open compact subgroup. Normalize the Haar measure $dg$ on $G$ so that $H$ has measure 1. Suppose $g_1, g_2 \in G$, and suppose \[\vol(Hg_1H;dg)\vol(Hg_2H;dg) = \vol(Hg_1g_2H;dg),\] where $\vol(X;dg)$ indicates the Haar measure of the set $X$. Then $f_{g_1} * f_{g_2} = f_{g_1g_2}$.
\end{proposition}
\begin{proof} This is proved in Chapter 3,\S 2 of \cite{How85}.
\end{proof}
\noindent We will prove a few lemmas that will help us determine the generators and relations for $\mathscr{H}(G, I_m)$.
\begin{lemma} \label{volW}
For $ \tilde{w} \in \tilde{W}_a$, we have $\vol(I_m\tilde{w}I_m; dg) = q^{l(w)} .$
\end{lemma}
\begin{proof}
Note that $\vol(I_m\tilde{w}I_m; dg) = \#\left(I_m /(I_m \cap \tilde{w}I_m\tilde{w}^{-1})\right). \vol(I_m; dg)$. Hence it suffices to show that  \[ \#\left(I_m /(I_m \cap \tilde{w}I_m\tilde{w}^{-1})\right) = q^{l(w)}.\]  Let $w = ( \lambda, x)$ where $\lambda \in X_*({\bf T})$ and $x \in W$. The Iwahori factorization of $I_m$ gives 
\[I_m = \displaystyle{ \prod_{\alpha \in \Phi^+}U_{\alpha, \calp^m} T_{\calp^m}   \prod_{\alpha \in \Phi^-}U_{\alpha, \calp^{m+1}}} \]
as sets, for any ordering of the right hand side above.  
For each $\alpha \in \Phi$, let $a_\alpha: W_a \rightarrow \mathbb{Z}$ be the function $a_\alpha(w) = \langle x.\alpha, \lambda \rangle$. Then 
\begin{align*}
\tilde{w}I_m\tilde{w}^{-1}  =&\displaystyle{ \prod_{\alpha \in \Phi^+}U_{x. \alpha, \calp^{m + a_\alpha(w)}}  T_{\calp^m}   \prod_{\alpha \in \Phi^+}U_{-x. \alpha, \calp^{m+1 - a_{\alpha}(w)}}} \\
 =& \displaystyle{ \prod_{\alpha \in \Phi^+, x.\alpha \in \Phi^+}U_{x. \alpha, \calp^{m + a_\alpha(w)}}} \displaystyle{ \prod_{\alpha \in \Phi^+, x.\alpha \in \Phi^-}U_{x. \alpha, \calp^{m + a_\alpha(w)}}}  T_{\calp^m} \\
&\displaystyle{  \prod_{\alpha \in \Phi^+, x.\alpha \in \Phi^+}U_{-x. \alpha, \calp^{m+1 - a_{\alpha}(w)}}}  \displaystyle{  \prod_{\alpha \in \Phi^+, x.\alpha \in \Phi^-}U_{-x. \alpha, \calp^{m+1 - a_{\alpha}(w)}}}.
\end{align*}
Setting $x.\alpha = \beta$ in the previous expression, we get
\begin{align*}
\tilde{w}I_m\tilde{w}^{-1} =&  \displaystyle{ \prod_{\beta \in \Phi^+, x^{-1}.\beta \in \Phi^+}U_{\beta, \calp^{m + \langle\beta, \lambda\rangle}}}  \displaystyle{ \prod_{\beta \in \Phi^-, x^{-1}.\beta \in \Phi^+}U_{\beta, \calp^{m + \langle\beta, \lambda\rangle}}}  T_{\calp^m}\\
&  \displaystyle{ \prod_{\beta \in \Phi^+, x^{-1}.\beta \in \Phi^+}U_{-\beta, \calp^{m + 1 -\langle\beta, \lambda\rangle}}}   \displaystyle{ \prod_{\beta \in \Phi^-, x^{-1}.\beta \in \Phi^+}U_{-\beta, \calp^{m + 1 - \langle\beta, \lambda\rangle}}}\\
=&  \displaystyle{ \prod_{\beta \in \Phi^+, x^{-1}.\beta \in \Phi^+}U_{\beta, \calp^{m + \langle\beta, \lambda\rangle}}}  \displaystyle{ \prod_{\beta \in \Phi^+, x^{-1}.\beta \in \Phi^-}U_{-\beta, \calp^{m - \langle\beta, \lambda\rangle}}}  T_{\calp^m}\\
&  \displaystyle{ \prod_{\beta \in \Phi^+, x^{-1}.\beta \in \Phi^+}U_{-\beta, \calp^{m + 1 -\langle\beta, \lambda\rangle}}}   \displaystyle{ \prod_{\beta \in \Phi^+, x^{-1}.\beta \in \Phi^-}U_{\beta, \calp^{m + 1 + \langle\beta, \lambda\rangle}}}\\
 =&  \displaystyle{ \prod_{\beta \in \Phi_1}U_{\beta, \calp^{m + \langle\beta, \lambda\rangle}}}  \displaystyle{ \prod_{\beta \in \Phi_1}U_{-\beta, \calp^{m + 1 -\langle\beta, \lambda\rangle}}}  T_{\calp^m}\\
&   \displaystyle{ \prod_{\beta \in \Phi_2}U_{\beta, \calp^{m + (1 + \langle\beta, \lambda\rangle)}}}  \displaystyle{ \prod_{\beta \in \Phi_2}U_{-\beta, \calp^{m +1 -(1 + \langle\beta, \lambda\rangle)}}} ,
\end{align*}
where $\Phi_1 = \{\alpha \in \Phi^+ \,|\, x^{-1}\alpha >0\}$ and $\Phi_2 = \{\alpha \in \Phi^+ \,|\, x^{-1}\alpha < 0\}$.
Now it is clear that 
\[  \#\left(I_m /(I_m \cap \tilde{w}I_m\tilde{w}^{-1})\right) = \displaystyle{\prod_{\alpha \in \Phi_1}q^{|\langle\beta, \lambda\rangle|}} .  \displaystyle{\prod_{\alpha \in \Phi_2}q^{|1+\langle\beta, \lambda\rangle|}} = q^{l(w)}.\qedhere\]
\end{proof}
\noindent We know that for each $\rho \in \Omega, \rho S \rho^{-1} = S$. We need to understand how the representatives $\tilde{\rho} $ interact with the representatives of the elements of $S$.
\begin{lemma}\label{rhoS}

\begin{enumerate}[(i)]
\item Let $ \rho  \in \{\rho_1, \ldots \rho_l\}$ be a generator of $ \Omega $. Suppose $\rho s_i \rho^{-1} = s_j$.
\begin{enumerate}[(a)]
\item If $s_i$ and $ s_j \in S_1$, then $\tilde{\rho}\tilde{s}_i\tilde{\rho}^{-1} = \tilde{s}_j$.
\item Otherwise, $\tilde{\rho}\tilde{s}_i\tilde{\rho}^{-1} = t_{\rho, \alpha_i}\tilde{s}_j$, where $t_{\rho,\alpha_i} \in T$ is determined by $\rho$ and $\alpha_i$.

\end{enumerate}
\item Let $\mu_e \in \{\mu_1, \ldots \mu_k\}.$ If $\mu_es_i\mu_e^{-1} = s_j$, then
\[\tilde{\mu}_e \tilde{s}_i \tilde{\mu}_e^{a_e-1} =t_{\mu_e, \alpha_i} \tilde{s}_j,\]
where $t_{\mu_e, \alpha_i}$ is an element of order $\leq 2$ in $T$ determined by $\mu_e$ and $\alpha_i$.
\end{enumerate}
\end{lemma}
\begin{proof} Let us prove $(i)(a)$. Let $s_i = s_{\alpha_i}, s_j = s_{\alpha_j}, \alpha_i, \alpha_j \in \Delta$.  As explained in Section \ref{reps} (b), we write $\rho = (\lambda, x)$, and  choose $\tilde{\rho} = \pi_\lambda \tilde{x}$. Now, \[(\lambda, x)(0, s_i)(-x^{-1}\lambda, x^{-1}) = (0, s_j) \iff (\lambda - xs_ix^{-1}\lambda, xs_ix^{-1}) = (0, s_j).\] Hence $xs_ix^{-1} = s_j$ and $ \lambda = xs_ix^{-1}\lambda$.\\
$xs_ix^{-1} = s_j \implies x.\alpha_i = \pm \alpha_{j}$. Since $\rho \in \Omega, l(\rho) = 0$. This means that: \begin{align}\label{length}
 \left<\alpha_j, \lambda \right> = 0 &\text{ if }   x^{-1}\alpha_j >0, \nonumber \\
\left<\alpha_j, \lambda \right> = -1  &\text{ if }  x^{-1}\alpha_j <0. 
\end{align}
Since $ \left<\alpha_j, \lambda \right> = \left<\alpha_j, s_j\lambda \right> = \left<s_j \alpha_j, \lambda \right> = \left<- \alpha_j, \lambda \right>$, we have $\left<\alpha_j, \lambda \right> = 0$. Hence $x^{-1} \alpha_j = \alpha_i$, i.e $ x.\alpha_i = \alpha_j$. By 9.3.5 of \cite{Spr09}, we have $ \tilde{x}\lu_{\alpha_i}(t) \tilde{x}^{-1} = \lu_{\alpha_j}(t) \,\forall \,t \in F^\times$. This implies that $ \tilde{x}\lu_{- \alpha_i}(t) \tilde{x}^{-1} = \lu_{- \alpha_j}(t) \,\forall\, t \in F^\times$. Hence,\begin{align*}  \tilde{x}\tilde{s}_i\tilde{x}^{-1}  =  \tilde{x}\lw_{\alpha_i}(1)\tilde{x}^{-1} =  \lw _{\alpha_j}(1) =  \tilde{s}_j. 
\end{align*}
Therefore, \begin{align*} 
\tilde{\rho}\tilde{s}_i\tilde{\rho}^{-1}  & =  \pi_\lambda \tilde{x}\lw_{ \alpha_i}(1) \tilde{x}^{-1} \pi_{\lambda}^{-1} = \pi_{\lambda} \lw_{\alpha_j}(1) \pi_{\lambda}^{-1}  =   \lw_{\alpha_j}(\pi^{\left< \alpha_j, \lambda\right>}) =  \lw_{\alpha_j}(1) =  \tilde{s}_j.
\end{align*}
This finishes $(a)$.\\
To prove $(i) (b)$, we will just deal with the case $\rho s_i \rho^{-1} = s_0^{(j)}$ with $ i \neq 0$. All other cases follow by a similar argument. We drop the superscript and just write $s_0$ and $\alpha_0$ in place of $s_0^{(j)}$ and $\alpha_0^{(j)}$. \\As before,   $(\lambda, x)(0, s_i)(-x^{-1}\lambda, x^{-1}) = (-\alpha_0^{\vee}, s_{\alpha_0})$\begin{align*}
& \implies xs_ix^{-1} = s_{\alpha_0} \text{ and }  \lambda - xs_ix^{-1}. \lambda = -\alpha_0^{\vee} \\
& \implies x. \alpha_i = \pm \alpha_0 \text{ and }  \lambda = s_{\alpha_0}. \lambda - \alpha_0^{\vee}.
\end{align*}
Since $ \left< \alpha_0, \lambda \right> =  \left< \alpha_0, s_{\alpha_0}.\lambda \right> +  \left< \alpha_0, -\alpha_0^{\vee} \right> = - \left< \alpha_0, \lambda \right> - 2 $, we have $  \left< \alpha_0, \lambda \right> = -1$. 
So $ x.\alpha_i = -\alpha_0$ by Equation \eqref{length}. 
Then we see that  $ \tilde{x}\lu_{\alpha_i}(t) \tilde{x}^{-1} = \lu_{-\alpha_0}(\epsilon_{x,\alpha_i}t)$, where $\epsilon_{x,\alpha_i} = \pm 1$ depends only on $x$ and $\alpha_i$. Then $ \tilde{x}\lu_{-\alpha_i}(t) \tilde{x}^{-1} = \lu_{\alpha_0}(\epsilon_{x,\alpha_i}t)$. Hence,
\begin{align*}
\tilde{x}\lw_{\alpha_i}(1) \tilde{x}^{-1} =  \lw_{-\alpha_0}(\epsilon_{x,\alpha_i})=  \lw_{\alpha_0}(-\epsilon_{x,\alpha_i})= \alpha_0^\vee(-\epsilon_{x,\alpha_i})\lw_{\alpha_0}(1).
\end{align*}
Finally,\begin{align*}
\tilde{\rho}\tilde{s}_i\tilde{\rho}^{-1}  =\pi_\lambda \tilde{x}\lw_{\alpha_i}(1)\tilde{x}^{-1}\pi_\lambda^{-1}&= \pi_\lambda \alpha_0^\vee(-\epsilon_{x,\alpha_i})\lw_{\alpha_0}(1)\pi_\lambda^{-1}=  \alpha_0^\vee(-\epsilon_{x,\alpha_i}) \lw_{\alpha_0}(\pi^{\left<\alpha_0, \lambda \right>})\\
& =  \alpha_0^\vee(-\epsilon_{x,\alpha_i}) \lw_{\alpha_0}(\pi^{-1})=   \alpha_0^\vee(-\epsilon_{x,\alpha_i})\tilde{s}_0.
\end{align*}
This completes the proof of $(i)(b)$ with $t_{\rho, \alpha_i} =  \alpha_0^\vee(-\epsilon_{x,\alpha_i})$.

The proof of $(ii)$ is similar to $(i)$ and we omit the details. Note that  $\mu_e^{a_e - 1}  = \mu_e^{-1}$ and hence $\tilde{\mu}_e^{a_e-1}$ is a representative of $\mu_e^{-1}$, which explains the relation among the representatives described in $(ii)$.
\end{proof}

\begin{lemma}\label{rho} For each pair of elements $a ,b \in \{\tilde{\rho}_1, \ldots \tilde{\rho}_l, \tilde{\rho}_1^{-1}, \ldots \tilde{\rho}_l^{-1}, \tilde{\mu}_1, \ldots \tilde{\mu}_k\}$ there exists an element $t_{a,b}$ of order $\leq 2$ in $T$ determined by $ a$ and $b$, such that $a b = t_{a, b} b a$.
\end{lemma}
\begin{proof}
We will assume $a = \tilde{\rho}_i$ and $b= \tilde{\rho}_j$. The other cases follow similarly. Write $\rho_i = (\lambda_i, x_i)$ and $\rho_j = (\lambda_j, x_j)$. Since $\rho_i\rho_j = \rho_j \rho_i,$ we have $ (\lambda_i, x_i)(\lambda_j, x_j) = (\lambda_j, x_j)(\lambda_i, x_i) $ which implies $\lambda_i + x_i . \lambda_j = \lambda_j + x_j . \lambda_i, x_i x_j = x_j x_i$. Since   $ \pi_{\lambda_i}\tilde{x_i} \pi_{\lambda_j}\tilde{x_j}^{-1} =  \pi_{\lambda_i + x_i. \lambda_j}$, we have
\begin{align*}
\tilde{\rho}_i \tilde{\rho}_j &= \pi_{\lambda_i}\tilde{x}_i \pi_{\lambda_j}\tilde{x_j}= \pi_{\lambda_i}\tilde{x}_i \pi_{\lambda_j}\tilde{x}_i^{-1}\tilde{x}_i\tilde{x}_j= c(x_i, x_j)  \pi_{\lambda_i + x_i. \lambda_j }\widetilde{x_ix_j}\text{ (By 9.3.4 of \cite{Spr09})}\\
& = c(x_i,x_j)  \pi_{\lambda_j + x_j .\lambda_i}\widetilde{x_j x_i}= c(x_i, x_j)c(x_j,x_i)  \pi_{\lambda_j}\tilde{x_j}\pi_{\lambda_i}\tilde{x}_i = t_{\rho_i,\rho_j} \tilde{\rho}_j\tilde{\rho}_i.
 \end{align*} 
 where $ t_{\rho_i, \rho_j} = c(x_i, x_j)c(x_j,x_i)$ is an element of order $\leq 2$ in $T$ determined by $\rho_i$ and $\rho_j$.
\end{proof}
\begin{lemma}\label{conjI}
Let $w \in W_a$. 
\begin{enumerate}[(i)]
\item  Let $s_i = s_{\alpha_i}$ for some $\alpha_i \in \Delta$. Then $l(ws_i) = l(w) + 1 $ if and only if $\Ad\tilde{w}(\lu_{\alpha_i}(\calo)) \subset I$. 
\item Let $s_0^{(j)} \in S_2$. Then $l(ws_0^{(j)}) = l(w)+1 $ if and only if $\Ad\tilde{w}(\lu_{-\alpha_0^{(j)}}(\calp)) \subset I$.
\end{enumerate}
\end{lemma}
\begin{proof}
Let us prove $(i)$. Let $w = (\lambda, x)$ and recall that $a_{\alpha}(w) = \langle x.\alpha, \lambda \rangle$. With a simple manipulation, we see that
\[l(w) = \displaystyle{\sum_{\alpha \in \Phi^+, x.\alpha \in \Phi^+} |a_{\alpha}(w)| + \sum_{\alpha \in \Phi^-, x.\alpha \in \Phi^+} |a_{\alpha}(w) +1|}\]
Similarly,
\[l(ws_i) = \displaystyle{\sum_{\alpha \in \Phi^+, xs_i.\alpha \in \Phi^+} |a_{\alpha}(ws_i)| + \sum_{\alpha \in \Phi^-, xs_i.\alpha \in \Phi^+} |a_{\alpha}(ws_i) +1|}\]

We have two cases:
\begin{enumerate}[(a)]
\item Suppose $x. \alpha_i \in \Phi^+$. Since $\Ad\tilde{w}(\lu_{\alpha_i}(\calo)) = \lu_{x.\alpha_i}(\calp^{a_{\alpha_i}(w)})$, we need to show that $l(ws_i) = l(w)+1$ if and only if $a_{\alpha_i}(w) \geq 0$. Since $x.\alpha_i \in \Phi^+, \;\; xs_i.\alpha_i \in \Phi^-$. Also $s_i.(\Phi^+ \backslash \{\alpha_i\}) \subset \Phi^+$.
Then
\[l(w) = \displaystyle{\sum_{\alpha \in \Phi^+ \backslash \{\alpha_i\}, x.\alpha \in \Phi^+} |a_{\alpha}(w)| + \sum_{\alpha \in \Phi^- \backslash \{-\alpha_i\}, x.\alpha \in \Phi^+} |a_{\alpha}(w) +1| + |a_{\alpha_i}(w)|}\]
and
\begin{align*}
l(ws_i) &= \displaystyle{\sum_{ \alpha \in \Phi^+ \backslash \{\alpha_i\}, xs_i.\alpha \in \Phi^+} |a_{\alpha}(ws_i)|} + \displaystyle{\sum_{\alpha \in \Phi^- \backslash \{-\alpha_i\}, xs_i.\alpha \in \Phi^+} |a_{\alpha}(ws_i) +1|}  + |a_{-\alpha_i}(ws_i) + 1|\\
& = \displaystyle{\sum_{\beta \in \Phi^+ \backslash \{\alpha_i\}, x.\beta \in \Phi^+} |a_{\beta}(w)| + \sum_{\beta \in \Phi^- \backslash \{-\alpha_i\}, x.\beta \in \Phi^+} |a_{\beta}(w) +1|} + |a_{-\alpha_i}(ws_i) + 1|
\end{align*}
Hence we see that $l(ws_i) = l(w) +1$ if and only if $|a_{-\alpha_i}(ws_i) + 1| = |a_{\alpha_i}(w)| + 1$. Since $ a_{-\alpha_i}(ws_i) = a_{\alpha_i}(w)$, we see that the above holds if and only if $a_{\alpha_i}(w) \geq 0$. 
\item Suppose $x.\alpha_i \in \Phi^-$. We need to show that $l(ws_i) = l(w)+1$ if and only if $a_{\alpha_i}(w) \geq 1$. Then $xs_i.\alpha_i \in \Phi^+$. Proceeding as above, we see that $l(ws_i) = l(w) + 1$ if and only if $ |a_{\alpha_i}(ws_i)| - 1 = |a_{-\alpha_i}(w)  + 1|$. This holds if and only if $ a_{\alpha_i}(w) \geq 1$.
\end{enumerate}
For $(ii)$, the argument is quite similar to $(i)$, hence we will show only the important steps. We drop the superscript and just write $s_0$ in place of $s_0^{(j)}$ and $\alpha_0$ in place of $\alpha_0^{(j)}$.  Note that
\begin{align*}
l(ws_0) = \displaystyle{\sum_{\alpha \in \Phi^+, xs_{\alpha_0}.\alpha \in \Phi^+} }|a_{\alpha}(ws_{\alpha_0}) &+\langle\alpha, \alpha_0^{\vee}\rangle| +\displaystyle{\sum_{\alpha \in \Phi^- , xs_{\alpha_0}.\alpha\in \Phi^+} |a_{\alpha}(ws_{\alpha_0})) +\langle \alpha, \alpha_0^{\vee} \rangle +1|}
\end{align*}
Arguing as before we see that $x.\alpha_0 >0$ if and only if $ a_{\alpha_0}(w) \leq 0$ and   $x.\alpha_0 < 0$ if and only if $ a_{\alpha_0}(w) \leq 1$.  Hence $x.\alpha_0 >0$ if and only if $1- a_{\alpha_0}(w) \geq 1$ and $x.\alpha_0 <0$ if and only if $1- a_{\alpha_0}(w) \geq 0$. Since $\Ad \tilde{w}(\lu_{-\alpha_0}(\calp)) = \lu_{-x.\alpha_0}(\calp^{1-a_{\alpha_0}(w)})$, $(ii)$ follows. 
\end{proof}

\subsection{A presentation for $\fH(G,I_m)$}
We now write down a presentation for the Hecke algebra $\fH(G, I_m)$. The theorem below a generalization of the $\GL_n$ case in Chapter 3, \S{2} of \cite{How85}.
\begin{theorem}\label{presentation} The Hecke algebra $\mathscr{H}(G, I_m)$ is generated by the elements
\begin{enumerate}[(a)]
\item $f_{\tilde{s}_i}, \;s_i \in S$,
\item $f_{\tilde{\rho}_i}, f_{\tilde{\rho}_i^{-1}} , i = 1,2\ldots., l$,
\item $f_{\tilde{\mu}_j}, j = 1, 2 \ldots k$,
\item $f_{b}, b \in I$,
\end{enumerate}
subject to the following relations:
\begin{enumerate}[(A)]
\item For $s_i, s_j$ distinct elements of $S$,
\begin{enumerate}[(i)] 
\item $f_{\tilde{s}_i} * f_{\tilde{s}_j} = f_{\tilde{s}_j} * f_{\tilde{s}_i}  \text{  if  }  m_{ij}=2.$
\item $f_{\tilde{s}_i} * f_{\tilde{s}_j} * f_{\tilde{s}_i} = f_{\tilde{s}_j} *f_{\tilde{s}_i} * f_{\tilde{s}_j} \text{  if  }  m_{ij} = 3.$
\item $(f_{\tilde{s}_i} * f_{\tilde{s}_j})^{2} = (f_{\tilde{s}_j} * f_{\tilde{s}_i})^{2} \text{  if  } m_{ij} = 4.$
\item $(f_{\tilde{s}_i} * f_{\tilde{s}_j})^{3} =( f_{\tilde{s}_j} * f_{\tilde{s}_i})^{3} \text{  if  }  m_{ij} = 6 .$
\item $f_{\tilde{s}_i} * f_{\tilde{s}_i} * f_{\alpha_i^\vee(-1)} = q\left( \displaystyle{\sum_{x }} f_x\right), \text{  } x \in \Ad \tilde{s}_i(I_m). I_m/I_m. $
\end{enumerate}

\item  Let ${\rho}, {\rho}_0 \in \{{\rho}_1, \ldots {\rho}_l, {\rho}_1^{-1}, \ldots {\rho}_l^{-1}, {\mu}_1, \ldots {\mu}_k\}$.\begin{enumerate}[(i)]
\item Suppose $\rho = \mu_e, 1 \leq e \leq k$, then there exists an element $c(\rho)$ of order $\leq 2$ in T, uniquely determined by $\rho$  such that
\[\underbrace{f_{\tilde{\rho}}*f_{\tilde{\rho}}\ldots f_{\tilde{\rho}}}_{a_e \text{ times}} = f_{c(\rho)}.\]
\item
Suppose $\rho \in \{\rho_1, \rho_2, \ldots \rho_l, \rho_1^{-1}, \ldots \rho_l^{-1}\}$ and if $\rho s_i \rho^{-1} = s_j$. Then,
\begin{equation*} f_{\tilde{\rho}} * f_{\tilde{s}_i} * f_{{\tilde{\rho}}^{-1}} =
\begin{cases}
 f_{\tilde{s}_j} & \text{ if }  s_i,s_j \in S_1, \\
 f_{t_{\rho, \alpha_i}} *  f_{\tilde{s}_j} & \text{otherwise},\\
\end{cases}
\end{equation*}
where $t_{\rho, \alpha_i}$ are as in Lemma \ref{rhoS}$(i)$.
\item Suppose $\rho  = \mu_e, 1 \leq e \leq k$, and $\rho s_i \rho^{-1}   = s_j$, then 
\[ f_{\tilde{\rho}} *f_{\tilde{s}_i}* \underbrace{f_{\tilde{\rho}}*f_{\tilde{\rho}}* \ldots f_{\tilde{\rho}} }_{(a_e -1) \text{ times }}=  f_{t_{\rho, \alpha_i}} *  f_{\tilde{s}_j},\]
where $t_{\rho, \alpha_i}$ are as in Lemma \ref{rhoS}$(ii)$.
\item $ f_{\tilde{\rho}} * f_b * f_{{\tilde{\rho}}^{-1}} =  f_{\tilde{\rho} b\tilde{\rho}^{-1}}$ for $ b \in I.$
\item $ f_{\tilde{\rho}} * f_{\tilde{\rho}_0} = f_{t_{\rho,\rho_0}} * f_{\tilde{\rho}_0} * f_{\tilde{\rho}}$ where $t_{\rho, \rho_0}$ are as in Lemma \ref{rho}.
\end{enumerate}
\item \begin{enumerate}[(i)]
\item $f_1$ is the identity element of $\mathscr{H}(G, I_m).$
\item $f_b * f_{b'} = f_{bb'}$, for $ b, b' \in I.$
\item $ f_{\tilde{s}_i} *f_b = f_{\tilde{s}_ib\tilde{s}_i^{-1}} * f_{\tilde{s}_i} \text { for }  b \in I \cap \Ad \tilde{s}_i(I).$
\item For $ x \in \calo^\times, s_i \in S_1, \\f_{\tilde{s}_i} * f_{\lu_{\alpha_i}(x)} *f_{\tilde{s}_i} * f_{\alpha_i^\vee(-1)} =q( f_{\lu_{\alpha_i}(-x^{-1})} * f_{\tilde{s}_i} *  f_{\alpha_i^\vee(x)} *  f_{\lu_{\alpha_i}(-x^{-1})})$.
\item For $x \in \calo^\times$, $s_0  = s_0^{(j)}\in S_2,$ and $\alpha_0 = \alpha_0^{(j)}$,
\begin{align*}
f_{\tilde{s}_0} * f_{\lu_{-\alpha_0}(\pi x)} *f_{\tilde{s}_0} * f_{\alpha_0^\vee(-1)} =q( f_{\lu_{-\alpha_0}(-\pi x^{-1})} * f_{\tilde{s}_0} &*  f_{\alpha_0^\vee(-x^{-1})} *  f_{\lu_{-\alpha_0}(-\pi x^{-1})}).
\end{align*}
\end{enumerate}
\end{enumerate}
\end{theorem}
\begin{proof}
 As noted earlier, the set $\{f_{x\tilde{w}y} | x,y \in I, \tilde{w} \in \tilde{W}_a\}$ spans the $\mathbb{C}$-space $\mathscr{H}(G, I_m)$. Using Proposition \ref{vol}, we see that $ f_{x\tilde{w}y} = f_x * f_{\tilde{w}} * f_y$. Moreover, if $ w=  s_{i_1}\ldots.s_{i_c}\rho_1^{t_1}\ldots \rho_l^{t_l}\mu_1^{r_1}\ldots \mu_k^{r_k}$ is a reduced expression for $w$ with $t_i \in \Z, 1 \leq i \leq l$ and $0 \leq r_j < a_j, 1 \leq j \leq k$, then $ \tilde{w} = \tilde{s}_{i_1}\ldots  \tilde{s}_{i_c} \tilde{\rho}_1^{t_1}\ldots \tilde{\rho}_l^{t_l}\tilde{\mu}_1^{r_1}\ldots \tilde{\mu}_k^{r_k}$. Combining Proposition \ref{vol} and Lemma \ref{volW}, we see that $ f_{\tilde{w}} = f_{\tilde{s}_{i_1}}*  f_{\tilde{s}_{i_2}}\ldots.* f_{\tilde{s}_{i_c}}* f_{\tilde{\rho}_1}^{t_1} *  f_{\tilde{\rho}_2}^{t_2}\ldots *  f_{\tilde{\rho}_l}^{t_l}*f_{\tilde{\mu}_1}^{r_1}\ldots *  f_{\tilde{\mu}_k}^{r_k}$. Here, we have used the notation
\begin{equation}\label{negativepower} f_{\tilde{\rho}}^k  =
\begin{cases}
 f_{\tilde{\rho}}*f_{\tilde{\rho}}\ldots. *f_{\tilde{\rho}}  \text{ ($k$ times) } & \text{ if }  k \geq 0, \\
 f_{\tilde{\rho}^{-1}}* f_{\tilde{\rho}^{-1}}*\ldots * f_{\tilde{\rho}^{-1}} \text{ ($-k$ times) } & \text{ if } k <0.\\
\end{cases}
\end{equation}
Hence, the generators of the Hecke algebra are as stated in the theorem. 

By Lemma \ref{si}, Proposition \ref{vol}, and Lemma \ref{volW}, it is clear that \textit{(A)(i) - (A)(iv)} hold in $\mathscr{H}(G, I_m)$.
 For \textit{(A)(v)}, first observe that when $s_i \in S_1$, $I_m \Ad\tilde{s}_i(I_m) = \Ad\tilde{s}_i(I_m) I_m = U_{-\alpha_i, \calp^m}I_m$ and this contains $I_m$ as a normal subgroup. Now,  $f_{\tilde{s}_i} * f_{\tilde{s}_i} * f_{\alpha_i^\vee(-1)} = f_{\tilde{s}_i} * f_{\tilde{s}_i^{-1}}$  has its support in  $I_m \Ad\tilde{s}_i(I_m)$. Moreover, $ f_{\tilde{s}_i} * f_{\tilde{s}_i^{-1}}$ is invariant on left and right by $I_m \Ad\tilde{s}_i(I_m)$. Since the RHS of $(A)(v)$ is just the characteristic function of $I_m \Ad\tilde{s}_i(I_m)$, we see that $(A)(v)$ holds up to multiples. The volume of RHS of $(A)(v)$  is 
 \begin{align*} 
 q. \#( I_m \Ad\tilde{s}_i(I_m)/I_m)=q. \# ( U_{-\alpha_i, \calp^m}I_m/I_m) =q. \#( U_{-\alpha_i, \calp^m}/  U_{-\alpha_i, \calp^{m+1})}= q^2.
\end{align*}
 Since the $\vol(f_{\tilde{s}_i}; dg).\vol(f_{\tilde{s}_i^{-1}}; dg) = q^2$, we see that the relation $A(v)$ holds for $s_i \in S_1$.
 When $s_0 = s_0^{(j)}$, we have $I_m \Ad\tilde{s}_0(I_m) = \Ad\tilde{s}_0(I_m) I_m = U_{\alpha_0, \calp^{m-1}}I_m$, and we proceed as above to see that $A(v)$ holds for $s_0^{(j)} \in S_2$.\\
From Proposition \ref{vol}, Lemma \ref{volW} and the fact that $\mu_e$ is of order $a_e$, we see that $B(i)$ holds. By Lemmas \ref{volW}, \ref{rhoS} and \ref{rho}, we see that $(B)(ii), B(iii)$, and $(B)(v)$ hold. $(B)(iv)$ follows from the fact that $\tilde{\rho}, \tilde{\rho}^{-1}$ normalize $I$.\\
Relations $C(i)$-$C(iii)$ clearly hold in $\mathscr{H}(G, I_m)$. Let $ x \in \calo^\times$. Using the relation 
\begin{align*}
\left(\begin{array}{cc}
0 & 1 \\
-1 & 0 \\
\end{array}\right) \left(\begin{array}{cc}
1 & x \\
0 & 1 \\
\end{array}\right) &\left(\begin{array}{cc}
0 & -1 \\
1 & 0 \\
\end{array}\right)  = \left(\begin{array}{cc}
1 &0 \\
-x & 1 \\
\end{array}\right) \\  &= \left(\begin{array}{cc}
1 & -x^{-1} \\
0 & 1 \\
\end{array}\right) \left(\begin{array}{cc}
0 & 1 \\
-1 & 0 \\
\end{array}\right) \left(\begin{array}{cc}
x & 0 \\
0 & x^{-1} \\
\end{array}\right) \left(\begin{array}{cc}
1 & -x^{-1} \\
0 & 1 \\
\end{array}\right)
\end{align*} and applying $\phi_{\alpha_i}$  throughout, we see that
\begin{align}\label{relation}
 \tilde{s}_i \lu_{\alpha_i}(x)\tilde{s}_i^{-1} = \lu_{-\alpha_i}(-x) =  \lu_{\alpha_i}(-x^{-1}) \tilde{s}_i\alpha_i^\vee(x)\lu_{\alpha_i}(-x^{-1})
 \end{align} Now the LHS of $C(iv)$ has its support in $I_m\tilde{s}_iI_m \lu_{\alpha_i}(x)I_m \tilde{s}_i^{-1}I_m$. 
 Using the Iwahori factorization of $I_m$, and noting that $\Ad\tilde{s}_i(U^-\cap I_m) \subset I_m$, we see that\[I_m\tilde{s}_iI_m \lu_{\alpha_i}(x)I_m \tilde{s}_i^{-1}I_m = I_m\tilde{s}_i \lu_{\alpha_i}(x)U_{\alpha_i, \calp^{m}} \tilde{s}_i^{-1}I_m = I_m\lu_{-\alpha_i}(-x)U_{-\alpha_i, \calp^m}I_m.\]
 We need to show that $ I_m \lu_{-\alpha_i}(-x)U_{-\alpha_i, \calp^{m}}I_m = I_m \lu_{-\alpha_i}(-x)I_m$. Since there is only one root involved, we can perform this computation inside $\SL_2$ and then apply $\phi_{\alpha_i}$.  Hence we need to find $A, B$ in the $m$-th Iwahori filtration subgroup of $\SL_2(\calo)$ such that 
\[\left(\begin{array}{cc}
1 & 0 \\
-x+\pi^my & 1 \\
\end{array}\right) =  A \left(\begin{array}{cc}
1 & 0 \\
-x & 1 \\
\end{array}\right) B.\]
The above equality holds with
\[ A = \left(\begin{array}{cc}
1+\pi^my/x & \pi^my/x^2 \\
\pi^{2m}y^{2}/x& 1- \pi^my/x + \pi^{2m}y^2/x^2 \\
\end{array}\right) \text{ and } B= \left(\begin{array}{cc}
1 & -\pi^my/x^2 \\
0 & 1 \\
\end{array}\right)  \] Now, using Equation \eqref{relation} it is clear that $C(iv)$ holds up to multiples. A comparison of volumes shows that  $C(iv)$ holds.
 For $C(v)$, first notice that we have the relation
 \begin{align*}
 &
\left(\begin{array}{cc}
0 & \pi^{-1} \\
-\pi & 0 \\
\end{array}\right) \left(\begin{array}{cc}
1 & 0 \\
\pi x & 1 \\
\end{array}\right) \left(\begin{array}{cc}
0 & -\pi^{-1} \\
\pi & 0 \\
\end{array}\right) = \left(\begin{array}{cc}
1 & -\pi^{-1}x\\
0 & 1 \\
\end{array}\right)  \\& = \left(\begin{array}{cc}
1 & 0 \\
-\pi x^{-1} & 1 \\
\end{array}\right) \nonumber \left(\begin{array}{cc}
0 & \pi^{-1} \\
-\pi & 0 \\
\end{array}\right) \left(\begin{array}{cc}
-x^{-1} & 0 \\
0 & -x \\
\end{array}\right) \left(\begin{array}{cc}
1 & 0 \\
-\pi x^{-1} & 1 \\
\end{array}\right)
\end{align*}
Applying $\phi_{\alpha_0}$ throughout, we have
\begin{equation}\label{relation1} \tilde{s}_0 \lu_{-\alpha_0}(\pi x)\tilde{s}_0^{-1} = \lu_{\alpha_0}(-\pi^{-1}x) =  \lu_{-\alpha_0}(-\pi x^{-1}) \tilde{s}_0\alpha_0^\vee(-x^{-1})\lu_{-\alpha_0}(-\pi x^{-1})
\end{equation}
The LHS of $C(v)$ has its support in \[I_m \tilde{s}_0I_m\lu_{-\alpha_0}(\pi x)I_m \tilde{s}_0^{-1}I_m = I_m U_{\alpha_0, \calp^{m-1}}\lu_{\alpha_0}(-\pi^{-1} x) I_m.\]
 Now proceeding as in the case of $C(iv)$, we see that \[I_m U_{\alpha_0, \calp^{m-1}}\lu_{\alpha_0}(-\pi^{-1} x) I_m = I_m \lu_{\alpha_0}(-\pi^{-1} x) I_m,\] and consequently $C(v)$ holds.
 It remains to see that $(A), (B), (C)$ are the defining set of relations. Let $\mathscr{H}^*$ be the free algebra generated by the elements $(a), (b),(c),$ and  $(d)$, and satisfying the relations $(A), (B),$ and $(C)$ stated in the theorem. We need to show that $\mathscr{H}^* \rightarrow \mathscr{H}(G, I_m)$ is an isomorphism. 
 For $\tilde{w} \in \tilde{W}_a$, write $\tilde{w} =  \tilde{s}_{i_1}\ldots  \tilde{s}_{i_c} \tilde{ \rho}_1^{t_1}\ldots \tilde{\rho}_l^{t_l}\tilde{ \mu_1}^{r_1}\ldots \tilde{\mu}_k^{r_k}$ and define $\hat{f}_{\tilde{w}} = f_{\tilde{s}_{i_1}}* f_{\tilde{s}_{i_2}}*\ldots. f_{\tilde{s}_{i_c}} * f_{\tilde{\rho}_1}^{t_1}* f_{\tilde{\rho}_2}^{t_2}*\ldots f_{\tilde{\rho}_l}^{t_l}*f_{\tilde{\mu}_1}^{r_1}* f_{\tilde{\mu}_2}^{r_2}*\ldots f_{\tilde{\mu}_k}^{r_k}$ (with notation explained in Equation \eqref{negativepower}). 
 By relations $(A)(i)$ - $(A)(iv)$, this is independent of the choice of expression for $\tilde{w}$ (once we fix the Chevalley basis, the uniformizer $\pi$ and the representatives of $\Omega$ as  in Section \ref{reps}(b)  - Check for end of Section \ref{reps} for explanation). For $ g \in G$, write $g = x\tilde{w}y$ with $ x,y \in I,$ and $\tilde{w} \in \tilde{W}_a$. Define $\hat{f}_g =f_x * \hat{f}_{\tilde{w}}*f_y$. We need to show that $\hat{f}_g$ depends only on the coset $I_mgI_m$. Suppose $g = x_1\tilde{w}_1y_1$. Since $ G = \displaystyle{\coprod_{\tilde{w} \in\tilde{W}_a} I\tilde{w}I} $, we see that $ x\tilde{w}y =x_1\tilde{w}ty_1$ for some $t \in {\bf T}(\calo)$. We want 
 \begin{equation}\label{coset}
 f_x *\hat{f}_{\tilde{w}}*f_y =  f_{x_1} *\hat{f}_{\tilde{w}}*f_{ty_1}.
 \end{equation}  Since $yy_1^{-1}t^{-1} = \Ad\tilde{w}^{-1}(x^{-1}x_1)$,  $yy_1^{-1}t^{-1} \in \displaystyle{I \cap \Ad\tilde{w}^{-1}(I)}$. Therefore,
 \begin{align*}
\hat{f}_{\tilde{w}} * f_{yy_1^{-1}t^{-1}}  =  &f_{\tilde{s}_{i_1}}* f_{\tilde{s}_{i_2}}*\ldots. f_{\tilde{s}_{i_c}} * f_{\Ad\tilde{\rho}_1^{k_1}\ldots \tilde{\rho}_l^{k_l}({yy_1^{-1}t^{-1}})} * f_{\tilde{\rho}_1}^{t_1}\\
&* f_{\tilde{\rho}_2}^{t_2}*\ldots f_{\tilde{\rho}_l}^{t_l}*f_{\tilde{\mu}_1}^{r_1}* f_{\tilde{\mu}_2}^{r_2}*\ldots f_{\tilde{\mu}_k}^{r_k}
\end{align*}
using relations $B(ii)$ and $C(i)$. It is easy to check that
\begin{align*}a \in  \Ad (\tilde{s}_{i_1}\ldots  \tilde{s}_{i_c})^{-1} (I) \cap I \implies a \in \Ad \tilde{s}_{i_c}( I) \cap I \text{ and } \Ad\tilde{s}_{i_c}(a) \in I \cap \Ad (\tilde{s}_{i_1}\ldots  \tilde{s}_{i_{c-1}})^{-1} (I).  
\end{align*}
Using the above and relation $C(iii)$ repeatedly, we have 
\begin{align*}
\hat{f}_{\tilde{w}} * f_{yy_1^{-1}t^{-1}}  =& f_{\tilde{s}_{i_1}}* f_{\tilde{s}_{i_2}}*\ldots.  f_{\Ad\tilde{s}_{i_c}\tilde{ \rho}_1^{k_1}\ldots \tilde{\rho}_l^{k_l}\tilde{ \mu}_1^{r_1}\ldots \tilde{\mu}_k^{r_k}({yy_1^{-1}t^{-1}})}\\
& *  f_{\tilde{s}_{i_c}} * f_{\tilde{\rho}_1}^{t_1}* f_{\tilde{\rho}_2}^{t_2}*\ldots f_{\tilde{\rho}_l}^{t_l}*f_{\tilde{\mu}_1}^{r_1}* f_{\tilde{\mu}_2}^{r_2}*\ldots f_{\tilde{\mu}_k}^{r_k}\\
& = f_{\Ad\tilde{w}(yy_1^{-1}t^{-1})} * \hat{f}_{\tilde{w}}\\
& = f_{x^{-1}x_1} * \hat{f}_{\tilde{w}}
\end{align*}
Now it is clear that Equation \eqref{coset} holds using $C(i)$ and $ C(ii)$.
Hence we have a family of elements parametrized by $I_m \backslash G/ I_m$. To finish the proof, we need to show that the span, call it $\mathcal{J}^*$, of these elements is invariant under multiplication by generators. Then $\mathcal{J}^*$ will be a subalgebra of $\mathscr{H}^*$ containing all the generators, and hence $\mathcal{J}^*= \mathscr{H}^*$. \\
By relation $C(ii)$, it is clear that $\mathcal{J}^*$ is invariant under multiplication by $f_b, b \in I$. Using relations $B(i)$ - $B(v)$, $C(ii)$-$C(iii)$,  we see that $\mathcal{J}^*$ is invariant under multiplication by $f_{\tilde{\rho}_i},  f_{\tilde{\rho}_i^{-1}}, i = 1,2,\ldots l$, and $f_{\tilde{\mu}_j}, 1 \leq j \leq k$. 

For $s_i \in S_1$, we have show that $f_{\tilde{s}_i} * f_x * \hat{f}_{\tilde{w}}$ lies in $\mathcal{J}^*$. 
For $x \in I$, write $x = zy$ with $ z \in I \cap\Ad\tilde{s}_i(I), y \in U_{\alpha_i, \calo}$. Then $f_{\tilde{s}_i} * f_x * \hat{f}_{\tilde{w}} = f_{\Ad\tilde{s}_i(z)} * f_{\tilde{s}_i}* f_y *  \hat{f}_{\tilde{w}}$ using relations $C(ii), C(iii)$. \\
\textbf{Case $1_i$}: $l(s_iw) = l(w) +1$. Then $\Ad\tilde{w}^{-1}(y) \in I$ by Lemma $\ref{conjI}(i)$. Hence a repeated application of $C(iii)$ gives $f_y * \hat{f}_{\tilde{w}} = \hat{f}_{\tilde{w}} * f_{\Ad\tilde{w}^{-1}(y)}$ which is an element of $\mathcal{J}^*$. \\
\textbf{Case $2_i$}: $l(s_iw) = l(w) - 1 = c\; (say)$. Set $w_1 = s_iw$. Therefore $l(w) = l(w_1) + 1$. Moreover, we can write $w = s_is_{j_1}\ldots s_{j_c}\rho_1^{t_1}\ldots \rho_l^{t_l}\mu_1^{r_1}\ldots \mu_k^{r_k}$, where $ s_{j_1}\ldots s_{j_c}\rho_1^{t_1}\ldots \rho_l^{t_l}\mu_1^{r_1}\ldots \mu_k^{r_k}$ is the expression for $w_1$.
 Hence we have $\hat{f}_{\tilde{w}} = f_{\tilde{s}_i} * \hat{f}_{\tilde{w_1}}$. So $f_{\tilde{s}_i} * f_x * \hat{f}_{\tilde{w}} = f_{\Ad\tilde{s}_i(z)} * f_{\tilde{s}_i}* f_y * f_{\tilde{s}_i}* \hat{f}_{\tilde{w_1}} $.
 Now, if $y \in U_{\alpha_i, \calp}$, then we may use $C(iii)$ again to move $f_y$ to the left of $f_{\tilde{s}_i}$ and we are left with an expression of the form $f_{\tilde{s}_i} * f_{\tilde{s}_i} * \hat{f}_{\tilde{w_1}} $, and we are done using relation $A(v)$. If $ y \notin  U_{\alpha_i, \calp}$, then $ y = \lu_{\alpha_i}(t)$ with $t \in \calo^\times$.  Using $C(iv)$, we see that 
 \[  f_{\tilde{s}_i}* f_y * f_{\tilde{s}_i}* \hat{f}_{\tilde{w_1}} = q( f_{\lu_{\alpha_i}(-t^{-1})} * f_{\tilde{s}_i} *  f_{\alpha_i^\vee(t)} *  f_{\lu_{\alpha_i}(-t^{-1})}) * f_{\alpha_i^\vee(-1)} * \hat{f}_{\tilde{w_1}}\]
 Since $l(s_iw_1) = l(w_1) + 1$, we proceed as in \textit{Case $1_i$}.
 
Finally, for $s_0 = s_0^{(j)} \in S_2$ (and $\alpha_0 = \alpha_0^{(j)}$),  we have to show that $f_{\tilde{s}_0} * f_x * \hat{f}_{\tilde{w}}$ lies in $\mathcal{J}^*$. Again, for $x \in I$, we write $x = z.y$ for $z \in I \cap\Ad\tilde{s}_0(I)$ and $y \in U_{-\alpha_0, \calp}$. Hence we only have to prove that $f_{\tilde{s}_0} * f_y * \hat{f}_{\tilde{w}}$ lies in $\mathcal{J}^*$.\\
 \textbf{Case $1_0$}: $l(s_0w) = l(w) +1$. Then $\Ad\tilde{w}^{-1}(y) \in I$ by Lemma $\ref{conjI}(ii)$. Hence a repeated application of $C(iii)$ gives $f_y * \hat{f}_{\tilde{w}} = \hat{f}_{\tilde{w}} * f_{\Ad\tilde{w}^{-1}(y)}$ which is an element of $\mathcal{J}^*$. \\
\textbf{Case $2_0$}: $l(s_0w) = l(w) - 1$. Set $w_1 = s_0w$. Therefore $l(w) = l(w_1) + 1$. Moreover, we can find a reduced expression for $w$ as $w = s_0s_{j_1}\ldots s_{j_c}\rho_1^{t_1}\ldots \rho_l^{t_l}\mu_1^{r_1}\ldots \mu_k^{r_k}$, where $ s_{j_1}\ldots s_{j_c}\rho_1^{t_1}\ldots \rho_l^{t_l}\mu_1^{r_1}\ldots \mu_k^{r_k}$ is a reduced expression for $w_1$.
 Hence we have $\hat{f}_{\tilde{w}} = f_{\tilde{s}_0} * \hat{f}_{\tilde{w_1}}$. So $f_{\tilde{s}_0} * f_y * \hat{f}_{\tilde{w}} =  f_{\tilde{s}_0}* f_y * f_{\tilde{s}_0}* \hat{f}_{\tilde{w_1}} $.
 Now, if $y \in U_{-\alpha_0, \calp^2}$, then we may use $C(iii)$ again to move $f_y$ to the left of $f_{\tilde{s}_0}$ and we are left with an expression of the form $f_{\tilde{s}_0} * f_{\tilde{s}_0} * \hat{f}_{\tilde{w_1}} $, and we are done using relation $A(v)$. If $ y \notin  U_{-\alpha_0, \calp^2}$, then $ y = \lu_{-\alpha_0}(\pi t)$ with $t \in \calo^\times$.  Using $C(v)$, we see that 
 \[  f_{\tilde{s}_0}* f_y * f_{\tilde{s}_0}* \hat{f}_{\tilde{w_1}} = q( f_{\lu_{-\alpha_0}(-\pi t^{-1})} * f_{\tilde{s}_0} *  f_{\alpha_0^\vee(-t^{-1})} *  f_{\lu_{-\alpha_0}(-\pi t^{-1})}) * f_{\alpha_0^\vee(-1)} * \hat{f}_{\tilde{w_1}}\]
 Since $l(s_0w_1) = l(w_1) + 1$, we proceed as in \textit{Case $1_0$}.\\
 Hence we have shown that $\mathcal{J}^*$  is invariant under multiplication by all the generators of $\mathscr{H}^*$. Hence, $\mathcal{J}^* = \mathscr{H}^*$.
 \end{proof}

\subsection{A variant of the Kazhdan isomorphism}\label{CLF}
 Assume $F$ and $F'$ are $m$-close. Let $\Lambda : \calo/\calp^{m} \xrightarrow{\cong}\calo'/\calp'^{m}$. When $\Lambda(x \text{ mod } \calp^{m}) = (x' \text{ mod } \calp'^{m})$, we write $x \sim_{\Lambda} x'$.    Recall that for an algebraic group  {\bf H} over $F'$, we let $H'$ =  {\bf H}($F'$).  We then have $G'$, $B'$, $T'$, $U'$ and so on.  Let $I'$ be the Iwahori subgroup of $G'$, that is $I'$ is the inverse image of ${\bf B}(\mathfrak{f}')$ under the map ${\bf G}(\calo') \rightarrow {\bf G}(\mathfrak{f}')$.  Again there is a unique  smooth affine group scheme $\bfi'$ defined over $\calo'$ with generic fiber $\bfg \times _\Z F'$ such that $\bfi(\calo') = I'$. 
Define $I_m':=\Ker(\bfi'(\calo') \rightarrow \bfi'(\calo'/\calp'^m))$. More explicitly, $I_m' = \left\langle U_{\alpha, \calp'^m}, T_{\calp'^m}', U_{-\alpha,\calp'^{m+1}}\,|\, \alpha \in \Phi^+  \right\rangle$. Using the presentation written above, we would like to establish a variant of the Kazhdan isomorphism analogous to Section \ref{Kazvariant}. To proceed, we need to establish that $I/I_m\cong I'/I_m'$ when $F$ and $F'$ are $m$-close. Since the group schemes $\bfi$ and $\bfi'$ are defined over $\calo$ and $\calo'$ respectively (and not over $\Z$), a little bit of care is needed to compare these objects over close local fields.  The author thanks Professor Jiu-Kang Yu for providing the following note on Iwahori group schemes and comparing them over close local fields.
\subsubsection{A note on Iwahori group schemes}\label{IGS}
Let us recall the construction of the group scheme $\bfi$: it is the dilatation of $\bfb \times \mathfrak{f}$ on $\bfg \times \calo$ (see \cite{YuPreprint}).
\subsubsection*{Review of dilatation} Let $\bfx$ be an affine flat scheme of finite type over $\calo$, and $\bfy$ a closed subscheme of $\bfx \times_\calo \mathfrak{f} $. The dilatation of $\bfy$ on $\bfx$ is a unique affine flat $\calo$-scheme $\bfx_0$ representing the functor
\[C \rightarrow\{x \in \bfx(C)| x_\mathfrak{f}:\Spec(C\otimes \mathfrak{f}) \rightarrow \bfx_\mathfrak{f} \text{ factors through } \bfy \hookrightarrow \bfx_\mathfrak{f}\}\]
on the category of $\calo$-algebras. If $A$ is the affine ring of $\bfx$ and $(\pi, f_1, f_2\ldots, f_n)$ is the ideal of $A$ defining $\bfy$, then in fact the affine ring of $\bfx_0$ is simply
\[A[f_1/\pi, f_2/\pi, \ldots, f_n/\pi] \subset A \otimes_\calo F.\]
\subsubsection*{A variant of dilatation} Let $R$ be a Noetherian ring and $T$ an indeterminate. Let $\bfx$ be an affine $R$-scheme of finite type, and let $\bfy$ be a closed subscheme of $\bfx$. We will define a scheme $\bfx_0$ over $R[T]$ analogous to dilatation. More precisely, let $A$ be the affine ring of $\bfx$ and $(f_1, \ldots, f_n)$ be the ideal of $A$ defining $\bfy$. We define $\bfx_0$ to be the affine $R[T]$-scheme with affine ring
\[A[T, f_1/T,\ldots f_n/T] \subset A[T,T^{-1}].\]
Let us call an $R[T]$-algebra $C$ \textit{free of $T$-torsion }if the ideal $\{f \in C: Tf = 0\} = \{0\}$. It is clear that the affine ring of $\bfx_0$ is free of $T$-torsion, and it is uniquely characterized as the representing object of the functor
\[C \rightarrow \{ x \in \bfx(C)| \bar{x}: \Spec(C/TC) \rightarrow \bfx \text{ factors through } \bfy \hookrightarrow \bfx\}\]
on the category of $R[T]$-algebras free of $T$-torsion. We will call $\bfx_0$ the $T$-dilatation of $\bfy$ on $\bfx$. 

There are two important properties of $T$-dilatation for our application. The first is that the formation of $T$-dilatation commutes with products, similar to usual dilatation. 
\begin{proposition} For $i=1,2$, let $\bfx_i$ be affine $R$-schemes of finite type, and let $\bfy_i$ be a closed subscheme of $\bfx_i$. Let $(\bfx_i)_0$ be the $T$-dilatation of $\bfy_i$ on $\bfx_i$. Then the $T$-dilatation of $\bfy_1 \times_R \bfy_2$ on $\bfx_1 \times_R \bfx_2$ is canonically isomorphic to $(\bfx_1)_0 \times_{R[T]} (\bfx_2)_0$.
\end{proposition}
\begin{Corollary} If $\bfx$ is an affine $R$-group scheme of finite type, and $\bfy$ is a closed $R$-subgroup scheme of $\bfx$, then the $T$-dilatation $\bfx_0$ of $\bfy$ on $\bfx$ is naturally a group scheme over $R[T]$.
\end{Corollary}
\noindent The proof of the proposition and corollary are immediate from the definition of $T$-dilatation. 

The next property is the compatibility of $T$-dilatation with usual dilatation.
\begin{proposition}\label{Tdilusualdil}Let $R, \bfx, \bfy$ as above, and  let $\bfx_0$ (with $R[T]$-group scheme structure as in the preceding corollary) be the $T$-dilatation of $\bfy$ on $\bfx$ . Let $\calo$ be a complete discrete valuation ring that is an $R$-algebra, and let $\pi$ be a prime element of $\calo$. Let $\phi: R[T] \rightarrow \calo$ be the unique $R$-algebra morphism sending $T$ to $\pi$. Then $\bfx_0 \times_{R[T],\phi} \calo$ is canonically isomorphic to the dilatation of $\bfy \times_R \calo/\pi \calo$ on $\bfx \times_R \calo$ as group schemes.  
\end{proposition}
\noindent This follows from the explicit description of the affine rings of dilatation and $T$-dilatation.
\subsubsection*{Comparing Iwahori group schemes over close local fields}
Let $F$ and $F'$ be non-archimedean local fields that are $m$-close. Let $\bfi/\calo$ and $\bfi'/\calo'$ be the Iwahori group schemes associated to the standard Iwahori subgroups as before. 
\begin{theorem} The isomorphism $\Lambda:\calo/\calp^m \rightarrow \calo'/\calp'^m,\;\pi\mod \calp^m \rightarrow \pi' \mod \calp'^m$, induces a canonical isomorphism
\[ \bfi \times_\calo (\calo/\pi^m \calo) \times_\Lambda (\calo'/\pi'^m\calo') \rightarrow \bfi' \times_{\calo'} (\calo'/\pi'^m\calo')\]
as group schemes.
\end{theorem}
\begin{proof} Let $\bfj$ be the $\Z[T]$-scheme which is the $T$-dilatation of $\bfb$ on $\bfg$. Let $\phi$ be as in Proposition \ref{Tdilusualdil}, and define $\phi'$ analogously. Then, $\bfi$ is canonically isomorphic to $\bfj \times_\phi \calo$. Therefore the LHS is canonically isomorphic to $\bfj \times_\eta (\calo'/\pi'^m\calo')$ where $\eta$ is the composition
\[\Z[T] \xrightarrow{\phi}\calo \rightarrow \calo/\pi^m\calo \xrightarrow{\Lambda} \calo'/\pi'^m\calo'.\]
But clearly, $\eta$ is identical to the composition $\Z[T] \xrightarrow{\phi'} \calo' \rightarrow \calo'/\pi'^m\calo'$ and the theorem follows.
\end{proof}
\begin{Corollary} The isomorphism $\Lambda$ induces a canonical isomorphism of groups
\[\bfi(\calo/\pi^m\calo) \rightarrow \bfi'(\calo'/\pi'^m\calo').\]

\end{Corollary}
\subsubsection{Hecke algebra isomorphism over close local fields} We now generalize the isomorphism in Section \ref{Kazvariant} (a) to split reductive groups.  Let $F$ and $F'$ be $m$-close. Let  $\beta: I/I_m \xrightarrow{\cong} I'/I_m'$  obtained using the above corollary. 
For each $b \in I$, let $b' \in I' $ be such that $b' \mod I_m' = \beta (b \mod I_m)$, and we write $b \sim_{\beta}b'$. Fixing the Chevalley basis as in Section \ref{stdnotations}, note that if we fix an expression of $b \in I$ as \begin{align}\label{betacloselocalfields}
&b = \displaystyle{\left(\prod_{\alpha \in \Phi^+} \lu_{\alpha}(x_\alpha)\right)\;t\;\left(\prod_{\alpha \in \Phi^+} \lu_{-\alpha}(\pi x_{-\alpha})\right)}, \text{ then }\\\nonumber
&b' = \displaystyle{\left(\prod_{\alpha \in \Phi^+} \lu_{\alpha}(x_\alpha')\right)\;t'\;\left(\prod_{\alpha \in \Phi^+} \lu_{-\alpha}(\pi' x_{-\alpha}')\right)}
\end{align}
is an element of $I'$ with $x_\alpha \sim_{\Lambda} x_\alpha'$, $\pi x_{-\alpha} \sim_{\Lambda} \pi' x_{-\alpha}'$, and $t'$  is obtained from $t$ using the isomorphism ${\bf T}(\calo/\calp^m) \xrightarrow{\cong} {\bf T}(\calo'/\calp'^m)$. Fix a uniformizer $\pi'$ of $F'$ such that $\pi \sim_\Lambda \pi'$. 
 We choose representatives of $W_a$ to $N_{G'}(T')$ exactly as in Section \ref{reps} using the same Chevalley basis, the same isomorphism fixed in \ref{reps}(b), and the uniformizer $\pi'$. We denote this set by $\tilde{W}_a'$.  Hence we have elements 
\begin{itemize}
\item $\tilde{s}_i', s_i \in S$
\item $\tilde{\rho}_i', \tilde{\rho}_i'^{-1}, i = 1,2\ldots l$
\item $\tilde{\mu}_j', 1 \leq j \leq k$
\end{itemize}
Let $ a: \tilde{W}_a \rightarrow \tilde{W}_a'$ be this bijection. So $a(\tilde{s}_i) = \tilde{s}_i', a(\tilde{\rho}_i^{\pm 1} )= \tilde{\rho}_i'^{\pm1}$,  and so on. We assume that the measures $dg$ and $dg'$ on $G$ and $G'$ respectively satisfy $\vol(I_m, dg) = \vol(I_m', dg') = 1$. 
\begin{theorem}\label{closealgebra} There exists a unique algebra isomorphism $\zeta_m: \mathscr{H}(G, I_m) \rightarrow \mathscr{H}(G', I_m')$ satisfying \begin{enumerate}
\item $ \zeta_m(f_{\tilde{s}_i}) = f_{\tilde{s}_i'}\; (s_i \in S), $
\item $\zeta_m(f_{\tilde{\rho}_i^{\pm 1}} )= f_{\tilde{\rho}_i'^{\pm1}} \; (i = 1,2\ldots.l), $
\item $\zeta_m(f_{\tilde{\mu}_j} )= f_{\tilde{\mu}_j'} \; (j = 1,2\ldots.k),$
\item $ \zeta_m(f_b) = f_{b'}$ where $ b \in I$ and $b \sim_\beta b'$. 
\end{enumerate}
\end{theorem}
\begin{proof} By Theorem \ref{presentation}, we know that $\mathscr{H}(G, I_m)$ is generated by the elements $f_{\tilde{s}_i}, \; i = 0,1,\ldots n$, $f_{\tilde{\rho}_i^{\pm1}},\; i = 1,2\ldots l, f_{\tilde{\mu}_j}, \; 1 \leq j \leq k  \text{ and } f_b, \; b \in I$. Define $\zeta_m$ on these elements according to $(a)$ - $(d)$ above. Let $g \in G$.  Write $g = b_1\tilde{w}b_2$, for $b_1, b_2 \in I$ and $ \tilde{w} \in \tilde{W}_a$. 
Define $\zeta_m(f_g) = f'_{b_1'a(\tilde{w})b_2'}$ where $b_1 \sim_\beta b_1'$ and $b_2 \sim_\beta b_2'$. We need to prove that $\zeta_m$ is well-defined.        Let $\tilde{w} \in \tilde{W_a}$. The set $I\tilde{w}I$ is a homogeneous space of the group $I \times I$ under the action $(b_1,b_2).g = b_1gb_2^{-1}\; (b_1, b_2 \in I, \; g \in I\tilde{w}I)$. 
The set $\{I_mgI_m \,|\, g \in I \tilde{w}I\}$ is then a homogeneous space of the finite group $I/I_m \times I/I_m$.
 Let $\Gamma_{\tilde{w}}\subset I/I_m \times I/I_m$ be the stabilizer of the double coset $I_m \tilde{w}I_m$.  We will show that $(\beta \times \beta)(\Gamma_{\tilde{w}})  = \Gamma_{\tilde{w}'}$, where $\tilde{w}' = a(\tilde{w})$. It is easy to see that we have a group isomorphism 
 \begin{align*}
  (I \cap \Ad\tilde{w}(I) )/(I_m \cap \Ad \tilde{w}(I_m)) &\xrightarrow{\cong} \Gamma_{\tilde{w}},\\
 b &\rightarrow (b, \Ad\tilde{w}^{-1}(b)).
 \end{align*} The map
 \begin{align*} 
\beta_{\tilde{w}}: (I \cap\Ad\tilde{w}(I) )/(I_m \cap \Ad \tilde{w}(I_m)) &\rightarrow (I' \cap \Ad\tilde{w}'(I') )/(I'_m \cap \Ad \tilde{w}'(I'_m))\\
b \mod(I_m \cap \Ad \tilde{w}(I_m)) &\rightarrow b' \mod (I'_m \cap \Ad \tilde{w}'(I'_m)),
 \end{align*}
where $b \sim_\beta b'$ and $ \tilde{w}^{-1}b \tilde{w} \sim_\beta \tilde{w}'^{-1}b'\tilde{w}'$, is an isomorphism of groups and hence
\begin{equation}\label{stabilizer} (\beta \times \beta)(\Gamma_{\tilde{w}})  = \Gamma_{\tilde{w}'}.\end{equation}
Consequently, the set $\{I_mgI_m| g \in I\tw I\}$ is in bijection with the set $\{I_m'g'I_m'| g' \in I'\tw'I'\}$ and hence $\zeta_m$ is a well-defined bijection. Also, it is clear that the map $\zeta_m$ preserves the relations $(A)$ - $(C)$ of Theorem \ref{presentation}. Hence $\zeta_m$ is in fact an algebra isomorphism.
\end{proof}
\begin{Remark}\label{Remark2} Recall that the extended affine Weyl group $W_a = X_*({\bf T}) \rtimes W \cong W_S \rtimes \Omega$. 
 Using the second isomorphism, we wrote $w \in W_a$ as  \[w = s_{j_1}\ldots s_{j_c}\rho_1^{t_1}\ldots \rho_l^{t_l}\mu_1^{r_1}\ldots \mu_k^{r_k}\] with $s_{j_1}, \ldots.s_{j_c} \in W_S$ and $\rho_1^{t_1}, \ldots\rho_l^{t_l}, \mu_1^{r_1},\ldots \mu_k^{r_k} \in \Omega$. 
In Section \ref{reps}, we fixed a uniformizer $\pi$ of $F$ to choose representatives of $\tilde{w} = \tilde{s}_{i_1}\ldots.\tilde{s}_{i_c}.\tilde{\rho}_1^{t_1}\ldots\tilde{\rho}_l^{t_l}\tilde{\mu}_1^{r_1}\ldots\tilde{\mu}_k^{r_k}$. 
Note that $w$ can be written as $(\lambda, x)$ as an element of $ X_*({\bf T}) \rtimes W$. Hence $\pi_\lambda\tilde{x}$ is another lifting of $w$. A priori, these two liftings differ by an element of ${\bf T}(\calo)$. But since the same uniformizer $\pi$ is used to choose these representatives, it is clear that there is an element $c_w$ of order $\leq 2$ in $T$ such that $\tilde{w} = c_w \pi_{\lambda}\tilde{x}$. Moreover, $\tilde{w}' = a(\tilde{w}) = c_w'\pi'_{\lambda}\tilde{x}'$, where $c_w \sim_\beta c_w'$. Hence $\zeta_m(f_{\pi_\lambda \tilde{x}}) = \zeta_m(f_{c_w^{-1}})* \zeta_m(f_{\tilde{w}}) = f'_{{c'}_w^{-1}}* f'_{\tilde{w}'} = f'_{\pi'_{\lambda}\tilde{x}'}$. 

\end{Remark}
\begin{Corollary}\label{proofKazConj} Conjecture \ref{Kazconj} is valid, that is, $\Kaz_m$ is a Hecke algebra isomorphism when the fields $F$ and $F'$ are $m$-close. Also, $\zeta_m|_{\mathscr{H}(G, K_m)} = \Kaz_m$. 

\end{Corollary}
\begin{proof}Since $I_m \subset K_m$, $\mathscr{H}(G, K_m)$ is a subalgebra of $\mathscr{H}(G, I_m)$. Let \[e_{K_m} = \vol(K_m; dg)^{-1}. \Char(K_m)\]  be the idempotent element of $\mathscr{H}(G, K_m)$. Then $\mathscr{H}(G, K_m) = e_{K_m} * \mathscr{H}(G, I_m) * e_{K_m}$. 
Let  $F$ and $F'$ be $m$-close and let $\zeta_m:\mathscr{H}(G, I_m) \rightarrow \mathscr{H}(G', I_m')$ be the Hecke algebra isomorphism constructed above. 
Since $K_m/I_m \cong K_m'/I_m'$, we see that $\zeta_m(e_{K_m}) = e_{K_m'}$. Hence the restriction of $\zeta_m$ induces an isomorphism between the Hecke algebras  $\mathscr{H}(G, K_m) \cong \mathscr{H}(G', K_m')$ when the fields $F$ and $F'$ are $m$-close.
It remains to verify that $\zeta_m = \Kaz_m$ on $\mathscr{H}(G, K_m)$, as $\C$-maps. Let $t_g  =\vol(K_m; dg)^{-1}. \Char(K_mgK_m)$. We need to prove that $\zeta_m(t_{a_i\pi_\lambda a_j^{-1}}) = t_{a_i'\pi'_\lambda a_j'^{-1}}$ (with notation as in Section \ref{KazhdanIsomorphism}). By Lemma 2.1(b) of \cite{kaz86}, we see that $t_{a_i'\pi'_\lambda a_j'^{-1}} = t_{a_i'}*t_{\pi_{\lambda}'}*t_{a_j'^{-1}}$. Hence it suffices to show that
\begin{enumerate}[(i)]
\item $\zeta_m(t_{\pi_\lambda}) = \Kaz_m(t_{\pi_\lambda}) \; \forall \lambda \in X_*({\bf T})_-$,
\item $\zeta_m(t_{k}) =  \Kaz_m(t_{k}) \; \forall k \in {\bf G}(\calo)$.
\end{enumerate}
The set $\{ I_maI_m \,|\, a \in K_m \pi_\lambda K_m\}$ is a homogeneous space of the group $K_m/I_m \times K_m/I_m$. Let $\Gamma_{\lambda, m} \subset K_m/I_m \times K_m/I_m $ be the stabilizer of the coset $I_m \pi_\lambda I_m$. It is again easy to see that 
\begin{align*}
 \Gamma_{\lambda, m} \cong (K_m \cap \Ad\pi_\lambda(K_m))/(I_m \cap \Ad \pi_\lambda(I_m)) \cong (K_m' \cap \Ad\pi'_\lambda(K'_m))/(I'_m \cap \Ad \pi'_\lambda(I'_m)) \cong  \Gamma'_{\lambda, m} 
 \end{align*}
if $F$ and $F'$ are $m$-close. Let 
$T_{\lambda, m}$ be a set of representatives of \[\left(K_m/I_m \times K_m/I_m \right)/\Gamma_{\lambda,m}.\]
Then we immediately have
\[ t_{\pi_\lambda} = \vol(K_m; dg)^{-1} \displaystyle{\sum_{(k_1, k_2) \in T_{\lambda,m
}} \Char(I_mk_1\pi_\lambda k_2^{-1}I_m)}.\]
Applying the theorem, it is easy to see that (i) holds. (ii) follows by a similar argument. 
\end{proof}

Let $\fR(G)$ be the category of smooth complex representations of $G$. Let $\fR^m(G)$ be the subcategory of $\fR(G)$ of representations $(\sigma, V)$ of $G$ generated by its $I_m$-fixed vectors, that is $\sigma(G)(V^{I_m}) = V$.  Let $\fH(G, I_m)\imod$ be the category of $\fH(G, I_m)$-modules. 
\begin{proposition}\label{categories}
 The category $\fR^m(G)$ is closed under sub-quotients and the functor
\begin{align*}
 J_m:\fR^m(G) &\longrightarrow \fH(G, I_m) \imod,\\
(\sigma, V)&\longrightarrow V^{I_m},
\end{align*}

 is an equivalence of categories with left adjoint
\begin{align*}
 j_m: \fH(G, I_m)\imod &\longrightarrow \fR^m(G),\\
 V^{I_m} &\longrightarrow \fH(G)\otimes_{\fH(G, I_m)} V^{I_m}.
\end{align*}
 \end{proposition}
\begin{proof}
This would follow from Corollary 3.9 of \cite{Ber84} as soon as we verify that the compact open subgroup $I_m$ satisfies Conditions 3.7.1 and 3.7.2 of \cite{Ber84}. Let us recall these conditions.
\begin{enumerate}[(\text{3.7.}1)]
\item \label{Lemma1} Let $\bf P$ be a parabolic subgroup of $\bf G$ and let $(\sigma, V)$ be a representation of $G$.  Let $V(N) = \Span\langle \sigma(n)v - v | v \in V, n \in N \rangle$ and let $V_N = V/V(N)$. Then the canonical map $V^{M \cap I_m} \rightarrow V_N^{I_m}$ is surjective.
\item  Let $\bf M$ be a Levi subgroup of $\bf G$ and let $\bf P = MN$ be a parabolic subgroup of $\bf G$ with Levi $\bf M$. Let $K$ be a $G$-conjugate of $I_m$ and let $K_P = K \cap P/K\cap N$. For any parabolic subgroup $\bf Q$ of $\bf G$   with the same Levi subgroup $\bf M$ and any other $G$-conjugate $K_1$ of $I_m$, $(K_1)_Q$ is a conjugate of $K_P$ in $M$. 
\end{enumerate}
Since $I_m$ admits an Iwahori factorization with respect to $P$, the proof of (3.7.1) above follows from Proposition 3.5.2 of \cite{Ber84}. Condition (3.7.2) has been verified for $\GL_n$ in Lemma 1.3.3 of \cite{Lem01}, and the same proof works for any split reductive group $\bfg$. 
\end{proof}
\section{Properties of representations over close local fields}\label{PropertiesrepsCLF}
In this section we study various properties of representations over close local fields. We retain the notation of Section \ref{stdnotations}. 
\subsection{Genericity}\label{GR}
 Let $\psi:F \rightarrow \C^\times$ be a non-trivial additive character of $F$.  Since
\[{\bf U/[U,U] = \displaystyle{\prod_{\alpha \in \Delta} U_\alpha}},\]
a character of $\chi$ of $U$ can be written as
\[\chi =\displaystyle{\prod_{\alpha \in \Delta} \chi_\alpha\circ \lu_\alpha^{-1}}\]
where $\chi_\alpha$ is an additive character of $F$. Note that there exists $a_\alpha \in F$ such that
 \[\chi_\alpha(x) = \psi(a_\alpha x) \; \forall \; x \in F.\]
 Let $m_\alpha = \cond(\chi_\alpha)$. The character $\chi$ is \textit{generic} iff $\chi_\alpha$ is nontrivial for all $\alpha\in \Delta$, or equivalently,  $a_\alpha \in F^\times $ for all $\alpha \in \Delta$.
Let $(\tau, V)$ be an irreducible, admissible representation of $G$. We say that $(\tau, V)$ is $\chi$-\textit{generic} if \[\operatorname{Hom}_G(V, \,\Ind_U^G \mathbb{C}_\chi) \neq 0.\] In this case, this space is in fact one dimensional (\cite{Sha74}). We denote the image of $V$ in $\Ind_U^G \mathbb{C}_\chi$ as $\mathcal{W}(\tau, \chi)$ and call it the Whittaker model of $\tau$.

In this section, we prove that generic representations correspond over close local fields, generalizing Lemaire's work in  \cite{Lem01} for $\GL_n$. In fact, all the key ideas needed for this proof are already present in \cite{Lem01} for $\GL_n$.
We begin by recalling the following simple lemma.
\begin{lemma}\label{ind}
Let $H$ be a closed subgroup of $G$ and $K$ be a compact open subgroup of $G$. Fix a set of representatives $S$ for $H\backslash{G}/K$. Let $(\sigma, W)$ be an irreducible smooth representation of $H$. Then the following map\\
\begin{equation*}
\phi: (\Ind_{H}^{G} \sigma) ^{K} \longrightarrow \displaystyle \prod_{g \in S} W^{H \cap gKg^{-1}}
\end{equation*}
 defined by $ \phi(f) = (f(g))_{g\in S} $ is an isomorphism of $\mathbb{C}$-vector spaces.
\end{lemma}
\begin{proof}
We will first check that $\phi$ is well-defined,  that is, we need to show that $f(g) \in W^{H \cap gKg^{-1}}$. Let $h \in H \cap gKg^{-1}$. Then $hg = gk$ for some $k\in K$. Now, $\sigma(h)f(g) = f(hg) = f(gk) = f(g)$. It is clear that $\phi$ is a bijection.
\end{proof}

Let $m\geq 1$. Let $F'$ be another non-archimedean local field such that $F'$ is $(m+1)$-close to $F$. Let \[\Lambda_0: \calo/\calp^{m+1} \xrightarrow{\cong} \calo'/\calp'^{m+1}\] and fix a uniformizer $\pi'$ of $F'$ such that $\Lambda_0(\pi \text{ mod } \calp^{m+1}) = (\pi' \text{ mod } \calp'^{m+1})$.
 More generally, let
 \[\Lambda_{i-m-1}: \calp^{i-m-1}/\calp^i\rightarrow \calp'^{i-m-1}/\calp'^{i}\] be the isomorphism of additive groups satisfying \[\Lambda_{i-m-1}(\pi^{i-m-1}x \mod \calp^i) = \pi'^{i-m-1}\Lambda_0(x \mod \calp^{m+1})\] for $x \in \calo$.
  When $\Lambda_0(x \text{ mod } \calp^{m+1}) = (x' \text{ mod } \calp'^{m+1})$, we write $x \sim_{\Lambda_0} x'$ (and similarly for the maps $\Lambda_{i-m-1}$).
Let $\beta$ denote the isomorphism \[\beta: I/I_{m} \rightarrow I'/I_{m}'.\] We write $b \sim_\beta b'$ if $\beta(b\mod  I_{m}) = b' \mod I_{m}'$.

Let  $\zeta_m: \mathscr{H}(G, I_m) \rightarrow \mathscr{H}(G', I_m')$ be as in Theorem \ref{closealgebra}.
             In this section, we will be using the map $\zeta_m$ only at level $m$, hence we drop the subscript and just denote the above map as $\zeta$.  
Let ($\sigma, E(\chi))$ denote the induced representation $\Ind_U^{G} \mathbb{C}_\chi$ and $ E(\chi)^m = E(\chi)^{I_m}$.
So $E(\chi)^m = \{f: G \rightarrow C \,|\, f(ugb) = \chi(u)f(g) \text{ for } u \in U,\, g \in G,\, b \in I_m\}$.

Let $\chi'= \displaystyle \prod_{\alpha \in \Delta} \chi_{\alpha}'\circ \lu_\alpha^{-1}$, where $\chi_{\alpha}'$ satisfies the following conditions:
\begin{equation}\label{chialphacomp}
\cond(\chi_{\alpha}') = m_\alpha \hspace{6pt}\text{ and } \hspace{6pt}\chi_{\alpha}' \restriction{\calp'^{m_\alpha-m-1}/\calp'^{m_\alpha}}=  \chi_\alpha \restriction{\calp^{m_\alpha-m-1}/\calp^{m_\alpha}}.
\end{equation}

We then have the corresponding objects $E(\chi')$ and $E(\chi')^m$.
We will prove that there is an isomorphism of $\C$-vector spaces \[E(\chi)^m \overset{\cong}\rightarrow E(\chi')^m\] that is compatible with the Hecke Algebra isomorphism $\zeta$.

Let $E(\chi)^m_g \subset E(\chi)^m$ be the subspace of functions with support in $UgI_m$.
Since $\displaystyle{\coprod_{\tilde{w} \in \tilde{W_a}}} U\tilde{w}I = G$, we see that $E(\chi)^m$ is the sum of the subspaces $E(\chi)^m_{\tilde{w}b} \text{  },\,\tilde{w} \in \tilde{W_a}, b \in I$. Clearly, $\operatorname{dim} (E(\chi)^m_{\tilde{w}b}) \leq 1$. Notice that $E(\chi)^m_{\tilde{w}b} \neq 0$ iff $\chi|_{U \cap \tilde{w}I_m\tilde{w}^{-1}}  = 1$.
As before, for each $\alpha \in \Phi$, let $a_\alpha: W_a \rightarrow \mathbb{Z}$ be the function $a_\alpha(w) = \langle x.\alpha, \lambda \rangle$ where $w = (\lambda, x) $ with $\lambda \in X_*({\bf T}), x \in W$. Then, writing $I_m =\displaystyle{ \prod_{\alpha \in \Phi^+}U_{\alpha, \calp^m} \;\;T_{\calp^m}\;\;   \prod_{\alpha \in \Phi^-}U_{\alpha, \calp^{m+1}}}$, we see that
\[\tilde{w}I_m\tilde{w}^{-1} =\displaystyle{ \prod_{\alpha \in \Phi^+}U_{x. \alpha, \calp^{m + a_\alpha(w)}} \;\; T_{\calp^m}  \;\; \prod_{\alpha \in \Phi^-}U_{x. \alpha, \calp^{m+1 + a_{\alpha}(w)}}}, \] and
\[U \cap \tilde{w}I_m\tilde{w}^{-1} = \displaystyle{ \prod_{\alpha \in \Phi^+, x.\alpha \in \Phi^+}U_{x. \alpha, \calp^{m + a_\alpha(w)}} \prod_{\alpha \in \Phi^-, x.\alpha \in \Phi ^ +}U_{x. \alpha, \calp^{m+1 + a_{\alpha}(w)}}}.\]
Since $\cond(\chi_\alpha) = m_\alpha$, we see that
\begin{equation}\label{wIm1}
 \chi|_{U \cap \tilde{w}I_m\tilde{w}^{-1}} = 1 \iff \begin{cases} a_\alpha(w) \geq m_{x.\alpha}-m & \text{ if } \alpha \in \Phi^+, x.\alpha \in \Delta,\\
a_\alpha(w) \geq m_{x.\alpha}-m-1 & \text{ if } \alpha \in \Phi^-, x.\alpha \in  \Delta. \\
\end{cases}
\end{equation}
\noindent Set \[\tilde{W}_a^1 = \{\tilde{w} \in \tilde{W_a} | a_\alpha(w) \text{ satisfies Equation \eqref{wIm1} for those } \alpha \in \Phi \text { with } x.\alpha \in \Delta \}.\]
For $g \in G$, let $h_g \in E(\chi)_g^m$ be the function characterized by\\
\begin{equation*}
h_g(g)=
\begin{cases} 1 & \text{if $g \in U\tilde{W}_a^1I$,}
\\
0 &\text{otherwise.}
\end{cases}
\end{equation*}
For $g \in G$, the function $h_g$ satisfies
\begin{equation}\label{homo}
h_g = \chi(u) h_{ug} \text{ } (u \in U).
\end{equation}
In the following lemma, we will describe the action of the generators of the Hecke algebra $\mathscr{H}(G, I_m)$  on a typical element $h_g$ described above.
\begin{lemma}\label{heckeaction}
We have the following:
\begin{enumerate}[(a)]
\item For $ x \in \{ \tilde{\rho_1},\ldots, \tilde{\rho_l}, \tilde{\rho_1}^{-1}, \ldots, \tilde{\rho_l}^{-1}, \tilde{\mu}_1, \ldots, \tilde{\mu}_k\}, \sigma(f_x)(h_g)= h_{gx^{-1}}$.
\item Let $s_i = s_{\alpha_i} \in S_1$. Then $\sigma(f_{\tilde{s_i}})(h_g) = \displaystyle \sum_{\{t_\nu\}}h_{g\lu_{\alpha_i}(\pi^mt_\nu)\tilde{s_i}^{-1}} $, where $\{t_\nu\}$ is a set of representatives for $\calo/\calp$.
\item Let $s_0 = s_0^{(j)} \in S_2$. Then $ \sigma(f_{\tilde{s_0}})(h_g) = \displaystyle \sum_{\{t_\nu\}}h_{g\lu_{-\alpha_0}(\pi^{m+1}t_\nu)\tilde{s_0}^{-1}} $, where $\{t_\nu\}$ is a set of representatives for $\calo/\calp$ and $\alpha_0 = \alpha_0^{(j)}$ is the highest root of the irreducible root system $\Phi^{(j)}$. 
\end{enumerate}
\end{lemma}
\begin{proof}
Let $du$ be the Haar measure on $U$ such that $\vol(U \cap I_m; du) = 1$. Recall that $dg$ is a Haar measure on $G$ such that $\vol(I_m; dg) = 1$. For $x,y,g \in G$,
\begin{align*}
\sigma(f_x)(h_g)(y) & = \int_{G}f_x(t)h_g(yt)dt\\
& = \int_{UgI_m} f_x(y^{-1}t)h_g(t)dt\\
& =\vol(U \cap gI_mg^{-1};du)^{-1}\int_U f_x(y^{-1}ug)\chi(u) du \\
& = \vol(U \cap gI_mg^{-1};du)^{-1}\int_{U \cap yI_mxI_mg^{-1}} \chi(u)du.
\end{align*}
Now, if $ x \in \{ \tilde{\rho_1},\ldots \tilde{\rho_l}, \tilde{\rho_1}^{-1},\ldots \tilde{\rho_l}^{-1},\tilde{\mu}_1, \ldots, \tilde{\mu}_k \}$, then $I_mxI_m = I_mx$ and $ U \cap yI_mxg^{-1} \neq \emptyset \iff y \in Ugx^{-1}I_m$. So, say, $ y = ugx^{-1}b$. Then,
\begin{align*}
\sigma(f_x)(h_g)(y) & = \vol(U \cap gI_mg^{-1};du) ^{-1}\int_{U \cap yI_mxI_mg^{-1}} \chi(u)du \\
 &=\frac{ \vol(U \cap gI_mg^{-1};du)}{\vol(U \cap gI_mg^{-1};du) } \chi(u) = \chi(u)h_{ugx^{-1}}(ugx^{-1})  = h_{gx^{-1}}(ugx^{-1}b)= h_{gx^{-1}}(y).
\end{align*}
The second equality above follows from the fact that for $\tilde{w} \in \tilde{W}_a^1$, $\chi|_{U \cap \tilde{w}I_m\tilde{w}^{-1}} = 1$. This proves $(a)$.\\
For $(b)$, notice that $I_m \tilde{s_i}I_m = \displaystyle \coprod_{\{t_\nu\}}I_m \lu_{-\alpha_i}(\pi^mt_\nu)\tilde{s_i} = \coprod_{\{t_\nu\}}I_m \tilde{s_i}\lu_{\alpha_i}(-\pi^mt_\nu)$. With $x = \tilde{s_i}, $
\begin{align*}
 \vol(U \cap gI_mg^{-1}) ^{-1}&\int_{U \cap yI_mxI_mg^{-1}} \chi(u)du \\
&= \vol(U \cap gI_mg^{-1}) ^{-1}\sum_{\{t_\nu\}}\int_{U \cap y I_m \tilde{s_i}\lu_{\alpha_i}(-\pi^mt_\nu)g^{-1}} \chi(u)du.
\end{align*}
Now, $ U \cap y I_m \tilde{s_i}\lu_{\alpha_i}(-\pi^mt_\nu)g^{-1} \neq \emptyset \iff y \in Ug \lu_{\alpha_i}(\pi^mt_{\nu})\tilde{s_i}^{-1}I_m$ and we proceed as before to finish the proof of $(b)$.\\
For $(c)$, note that $I_m \tilde{s_0}I_m = \displaystyle \coprod_{\{t_\nu\}}I_m \lu_{\alpha_0}(\pi^{m-1}t_\nu)\tilde{s_0} = \coprod_{\{t_\nu\}}I_m \tilde{s_0}\lu_{-\alpha_0}(-\pi^{m+1}t_\nu)$. Now, proceed as in $(b)$ to complete the proof.
\end{proof}
Recall that the elements  of $\tilde{W}_a^1$ satisfy Equation \eqref{wIm1}. We have a similar description for $\tilde{W}_a^{'1}$, the corresponding object over $F'$. For $g \in G'$, let $h'_g \in E(\chi')_g^m$ be the function characterized by\\
\begin{equation*}
h'_g(g)=
\begin{cases} 1 & \text{if $g \in U'\tilde{W}_a'^1I'$,}
\\
0 &\text{otherwise.}
\end{cases}
\end{equation*}
For $ \tilde{w} \in \tilde{W_a}, b \in I$, define \[\kappa(h_{\tilde{w}b}) = h'_{\tilde{w}'b'},\] where $ b \sim_{\beta}b'$ (Check Equation \eqref{betacloselocalfields}).
\begin{lemma} The map $h_{\tilde{w}b} \rightarrow \kappa(h_{\tilde{w}b})$ extends  to an isomorphism of $\mathbb{C}$-vector spaces $\kappa: E(\chi)^m \rightarrow E(\chi')^m$.
\end{lemma}
\begin{proof}
Since the $h_g$'s satisfy Equation \eqref{homo}, $\kappa$ is a homomorphism of $\mathbb{C}$-vector spaces iff \[\chi(\tilde{w}u\tilde{w}^{-1})\kappa(h_{\tilde{w}ub}) = \kappa(h_{\tilde{w}b})\] for $ u \in I \cap \tilde{w}^{-1}U\tilde{w}$. Put $w = (\lambda, x)$. For $ u \in I \cap \tilde{w}^{-1}U\tilde{w}$, choose $ u' \in I' \cap \tilde{w}'^{-1}U'\tilde{w}'$ such that $ u \sim_\beta u'$.  Fix an Iwahori factorization of  $u$ as
\[ u =   \displaystyle{ \prod_{\beta \in \Phi^+, x.\beta \in \Phi^+}\lu_{\beta}(t_{\beta}) . \prod_{\beta \in \Phi^-, x.\beta \in \Phi^+} \lu_{\beta}(\pi t_{\beta})}.\]
Then 
\[ u' =   \displaystyle{ \prod_{\beta \in \Phi^+, x.\beta \in \Phi^+}\lu_{\beta}(t_{\beta}') . \prod_{\beta \in \Phi^-, x.\beta \in \Phi^+} \lu_{\beta}(\pi' t_{\beta}')}\]
where $t_{\beta} \sim_{\Lambda_0} t_\beta'$. Computing $\tw u\tw^{-1}$, we have 
\[\tilde{w}u\tilde{w}^{-1}=  \displaystyle{ \prod_{\beta \in \Phi^+, x.\beta \in \Phi^+}\lu_{x.\beta}(c_{x.\beta,w}t_{\beta}\pi^{a_{\beta}(w)}) .\prod_{\beta \in \Phi^-, x.\beta \in \Phi^+} \lu_{x.\beta}(c_{x.\beta,w}\pi t_{\beta}\pi^{a_{\beta}(w)})}\]
and
\[\tilde{w}'u'\tilde{w}'^{-1}=  \displaystyle{ \prod_{\beta \in \Phi^+, x.\beta \in \Phi^+}\lu_{x.\beta}(c_{x.\beta,w}'t_{\beta}'\pi'^{a_{\beta}(w)}) . \prod_{\beta \in \Phi^-, x.\beta \in \Phi^+} \lu_{x.\beta}(c_{x.\beta,w}'\pi' t_{\beta}'\pi'^{a_{\beta}(w)})}\]
where $c_{x.\beta, w}  = \pm 1$ and $c_{x.\beta, w} \sim_{\Lambda_0} c_{x.\beta, w}'$. 
Since $\tilde{w} \in \tilde{W}_a^{1}$, in view of Equations \eqref{chialphacomp} and  \eqref{wIm1}, we see that $\chi(\tilde{w}u\tilde{w}^{-1}) = \chi'(\tilde{w}'u'\tilde{w}'^{-1})$. Finally
\[\kappa(h_{\tilde{w}ub}) = h'_{\tilde{w}'u'b'} = \chi'(\tilde{w}'u'\tilde{w}'^{-1})h'_{\tilde{w}'b'} = \chi(\tilde{w}u\tilde{w}^{-1})\kappa(h_{\tilde{w}b}).\]
We need to check that $\kappa$ is an isomorphism. For each $ \tilde{w} \in \tilde{W}_a^{1}$, let $R(\tilde{w})$ be the system of representatives of $(I \cap \tilde{w}^{-1}U\tilde{w}) \backslash I / I^m$ in $I$. Then the functions $h_{\tilde{w}b},\; \tilde{w} \in \tilde{W}_a^{1}, b \in R(\tilde{w})$ completely determine $E(\chi)^m$. Since $\beta$ induces, by passage to quotient, a bijection between $ (I \cap \tilde{w}^{-1}U\tilde{w}) \backslash I / I^m$ and $(I' \cap \tilde{w}'^{-1}U'\tilde{w}') \backslash I' / I'^m$, we see that $\kappa(h_{\tilde{w}b}), \, \tilde{w} \in \tilde{W}_a^{1}, b \in R(\tilde{w})$ forms a basis for $E(\chi')^m$.
\end{proof}

\begin{theorem}\label{genericrepresentationshecke}
Assume $F$ and $F'$ are $(m+1)$-close and let $\chi$ correspond to $\chi'$ as before. Then for each $f \in \mathscr{H}(G, I_m)$, the following diagram commutes:
\begin{displaymath}
    \xymatrix{
        E(\chi)^m \ar[r]^{\sigma(f)} \ar[d]_{\kappa} & E(\chi)^m \ar[d]^{\kappa} \\
        E(\chi')^m \ar[r]_{\sigma'(\zeta(f))}       & E(\chi')^m }
\end{displaymath}
where $\zeta$ is given in Theorem \ref{closealgebra}.
\end{theorem}
\begin{proof}
It suffices to prove that $\kappa(\sigma(f)(h_{\tilde{w}b}) )= \sigma'(\zeta(f)) \kappa(h_{\tilde{w}b}),$
where $\tilde{w} \in \tilde{W}_a^{1}, b \in I$, and $f$ is a generator of $\mathscr{H}(G, I_m)$ given in $(a), (b)$, $(c)$, and $(d)$ of Theorem \ref{presentation}.
If $ x \in I$, let $x' \in I'$ such that $ x \sim_\beta x'$. Then, $bx^{-1} \sim_\beta b'x'^{-1}$. Hence,
\[\kappa(\sigma(f_x)(h_{\tilde{w}b}) ) = \kappa(h_{\tilde{w}bx^{-1}}) = h_{\tilde{w}'b'x'^{-1}} = \sigma'(f'_{x'})(h'_{\tilde{w}'b'}) = \sigma'(\zeta_m(f_x))\kappa(h_{\tilde{w}b}).\]
If $ x \in\{\tilde{\rho}_1, \tilde{\rho}_1^{-1}, \ldots, \tilde{\rho}_l, \tilde{\rho}_l^{-1}, \tilde{\mu}_1, \ldots, \tilde{\mu}_k\}$, then  \[\kappa(\sigma(f_x)(h_{\tilde{w}b}) ) = \kappa(h_{\tilde{w}bx^{-1}})\] by Lemma \ref{heckeaction}.
We will just deal with the case  $ x = \tilde{\rho}_i^{-1}$. Write \[\tilde{w} =  \ts_{i_1}\ldots \ts_{i_c} \tilde{ \rho}_1^{t_1}\ldots \tilde{\rho}_l^{t_l}\tilde{\mu}_1^{r_1} \ldots \tilde{\mu}_k^{r_k}.\] Note that
$\tilde{w}\tilde{\rho}_i = \widetilde{w\rho_i}c$
where $c$ is an element of order $\leq 2$ in $T$, uniquely determined by $w$ and $\rho_i$. This constant $c$ arises out of some product of structure constants appearing in Lemma \ref{rho} (Check Remark \ref{Remark2} for a more detailed explanation). 

Let $c'$ be such that $\tilde{w}'\tilde{\rho}_i' = \widetilde{w\rho_i}'c'$. Then it is clear that $c \sim_{\beta}c'$.
Since $ b \sim_\beta b' \implies\Ad\tilde{\rho_i}(b) \sim_{\beta}\Ad\tilde{\rho_i}'(b')$, we see that
\begin{align*}
\kappa(h_{\tilde{w}bx^{-1}}) & = \kappa(h_{\tilde{w}\tilde{\rho_i}\Ad\tilde{\rho_i}^{-1}(b)})
= \kappa(h_{\widetilde{w\rho_i}c\Ad\tilde{\rho_i}^{-1}(b)} ) \\ &= h'_{\widetilde{w\rho_i}'c'\Ad\tilde{\rho_i}'^{-1}(b')}
 = h'_{\tilde{w}'\tilde{\rho_i}'c'^{-1}c'\Ad\tilde{\rho_i}'^{-1}(b')}\\
& = h'_{\tilde{w}'b'\tilde{\rho_i}'}= \sigma'(\zeta(f_{\tilde{\rho_i}^{-1}})) \kappa(h_{\tilde{w}b}).
\end{align*}
We will now deal with the case $x = \tilde{s_i},$ where $s_i = s_{\alpha_i} \in S_1$. By Lemma \ref{heckeaction}, we have
\[ \kappa(\sigma(f_{\tilde{s_i}}) h_{\tilde{w}b}) = \kappa \left(\sum_{\{t_\nu\}}h_{\tilde{w}b\lu_{\alpha_i}(\pi^mt_\nu)\tilde{s_i}^{-1}}\right).\]
Write $b = \lu_{\alpha_i}(t)c$, where $c \in I \cap\Ad\tilde{s_i}(I)$ and $t \in \calo^\times$. Then\[\tilde{w}b\lu_{\alpha_i}(\pi^mt_\nu)\tilde{s_i}^{-1} = \tilde{w}\lu_{\alpha_i}(t) \tilde{s_i}^{-1}\Ad\tilde{s_i}(c\lu_{\alpha_i}(\pi^mt_\nu)).\]
Fix $ c' \in I' \cap\Ad\tilde{s_i}(I')$ such that $ c \sim_{\beta}c',\Ad\tilde{s_i}(c) \sim_{\beta} \Ad\tilde{s_i}'(c')$. We also have \[\Ad\tilde{s_i}(\lu_{\alpha_i}(\pi^mt_\nu)) \sim_\beta\Ad\tilde{s_i}(\lu_{\alpha_i}(\pi'^mt_\nu')),\]
where $t_\nu'$ is chosen so that $t_\nu' \sim_{\Lambda_0} t_\nu$. We distinguish two cases.\\
\textbf{Case $1_i$}: Suppose $\Ad\tilde{w}(\lu_{\alpha_i}(t)) \in U$. Then,
\[ \kappa(h_{\tilde{w}b\lu_{\alpha_i}(\pi^mt_\nu)\tilde{s_i}^{-1}}) = \chi(\Ad\tilde{w}(\lu_{\alpha_i}(t)))\kappa(h_{\tilde{w}\tilde{s_i}\alpha_i^\vee(-1)\Ad\tilde{s_i}(c\lu_{\alpha_i}(\pi^mt_\nu))}).\]
Now, arguing as in the previous case, using Lemma \ref{rhoS}, we see that $\tilde{w}\tilde{s_i} = \widetilde{ws_i}a$, where $a$ is an element of order $\leq 2$ in $T$. Similarly,  $\tilde{w}'\tilde{s_i}' = \widetilde{ws_i}'a'$, and $ a \sim_{\beta}a'$. Moreover,  we clearly have
\[ a\alpha_i^\vee(-1)\Ad\tilde{s_i}(c\lu_{\alpha_i}(\pi^mt_\nu)) \sim_{\beta} a'\alpha_i^\vee(-1)\Ad\tilde{s_i}'(c'\lu_{\alpha_i}(\pi'^mt_\nu')),\]
 and consequently,
\[ \kappa(h_{\tilde{w}\tilde{s_i}\alpha_i^\vee(-1)\Ad\tilde{s_i}(c\lu_{\alpha_i}(\pi^mt_\nu))}) = h'_{\tilde{w}'\tilde{s_i}'\alpha_i^\vee(-1)\Ad\tilde{s_i}'(c'\lu_{\alpha_i}(\pi'^mt_\nu'))}.\]
Hence it remains to see that $\chi(\Ad\tilde{w}(\lu_{\alpha_i}(t))) = \chi'(\Ad\tilde{w}'(\lu_{\alpha_i}(t'))) $.
Note that \[\Ad\tilde{w}(\lu_{\alpha_i}(t)) = \lu_{\alpha}(c_{\alpha_i,w}\pi^{a_{\alpha_i}(w)}t), \] where $\alpha  = x.\alpha_i$, with $w = (\lambda, x)$. Now if $\alpha$ is not simple, then \[\lu_{\alpha}(c_{\alpha_i,w}\pi^{a_{\alpha_i}(w)}t) \in \operatorname{Ker}(\chi) \text{  (respectively for }\chi').\] Otherwise $a_{\alpha_i}(w) \geq m_\alpha - m$ by Equation \eqref{wIm1}.  In the latter case,
\begin{align*}
 \chi(\Ad\tilde{w}(\lu_{\alpha_i}(t))) &= \chi_{\alpha}(c_{\alpha_i,w}\pi^{a_{\alpha_i}(w)}t) = \chi'_{\alpha}(c_{\alpha_i,w}'\pi'^{a_{\alpha_i}(w)}t')= \chi'(\Ad\tilde{w}'(\lu_{\alpha_i}(t'))).
 \end{align*}
For the second equality above, note that \[c_{\alpha_i,w}t \sim_{\Lambda_0} c_{\alpha_i,w}' t' \implies c_{\alpha_i,w}\pi^{a_{\alpha_i}(w)}t \sim_{\Lambda_{m_\alpha -m-1}} c_{\alpha_i,w}\pi'^{a_{\alpha_i}(w)}t',\] and
$\chi_\alpha\restriction\calp^{m_\alpha -m-1}/\calp^{m_\alpha}  = \chi'_\alpha \restriction \calp'^{m_\alpha-m-1}/\calp'^{m_\alpha}$. This finishes \textit{Case $1_i$}.\\
\textbf{Case $2_i$}: Suppose $\Ad\tilde{w}(\lu_{\alpha_i}(t)) \notin U$. Then $\Ad\tilde{w}\tilde{s_i}(\lu_{\alpha_i}(t)) \in U$. Using relation \eqref{relation},
\begin{align*}
\tilde{w}b\lu_{\alpha_i}(\pi^mt_\nu)\tilde{s_i}^{-1} &= \tilde{w}\lu_{\alpha_i}(t)\tilde{s_i}^{-1}\Ad\tilde{s_i}(c\lu_{\alpha_i}(\pi^mt_\nu))\\
& = \tilde{w}\tilde{s_i}^{-1}\Ad\tilde{s_i}(\lu_{\alpha_i}(t))\Ad\tilde{s_i}(c\lu_{\alpha_i}(\pi^mt_\nu))\\
& = \tilde{w}\tilde{s_i}^{-1} \lu_{\alpha_i}(-t^{-1}) \tilde{s_i}\alpha_i^\vee(t)\lu_{\alpha_i}(-t^{-1})\Ad\tilde{s_i}(c\lu_{\alpha_i}(\pi^mt_\nu))\\
& = \tilde{w}\tilde{s_i} \lu_{\alpha_i}(-t^{-1}) \tilde{s_i}^{-1}\alpha_i^\vee(t)\lu_{\alpha_i}(-t^{-1})\Ad\tilde{s_i}(c\lu_{\alpha_i}(\pi^mt_\nu))\\
& =\Ad\tilde{w}\tilde{s_i}( \lu_{\alpha_i}(-t^{-1}))\tilde{w}\alpha_i^\vee(t)\lu_{\alpha_i}(-t^{-1})\Ad\tilde{s_i}(c\lu_{\alpha_i}(\pi^mt_\nu)).
\end{align*}
Hence
\begin{align*}
\kappa(h_{\tilde{w}b\lu_{\alpha_i}(\pi^mt_\nu)\tilde{s_i}^{-1}}) &= \chi(\Ad\tilde{w}\tilde{s_i}( \lu_{\alpha_i}(-t^{-1})) \kappa(h_{\tilde{w}\alpha_i^\vee(t)\lu_{\alpha_i}(-t^{-1})\Ad\tilde{s_i}(c\lu_{\alpha_i}(\pi^mt_\nu))})\\
& = \chi(\Ad\tilde{w}\tilde{s_i}( \lu_{\alpha_i}(-t^{-1})) h'_{\tilde{w}'\alpha_i^\vee(t')\lu_{\alpha_i}(-t'^{-1})\Ad\tilde{s_i}'(c'\lu_{\alpha_i}(\pi'^mt'_\nu))}.
\end{align*}
Now we only need to prove that \[ \chi(\Ad\tilde{w}\tilde{s_i}( \lu_{\alpha_i}(-t^{-1})) = \chi'(\Ad\tilde{w}'\tilde{s_i}'( \lu_{\alpha_i}(-t'^{-1})).\]
If $l(ws_i) = l(w) - 1$, then $\Ad\tilde{w}\tilde{s_i}(U_{\alpha_i, \calo}) \subset I$ by Lemma \ref{conjI}. Hence
\[\Ad\tilde{w}\tilde{s_i}( \lu_{\alpha_i}(-t^{-1})) \in {\bf U}(\calo)\] and we argue as in $Case\;1_i$ to finish the proof.\\
If $l(ws_i) = l(w) + 1$, then $\Ad\tilde{w}(\lu_{\alpha_i}(\calo)) \subset I \cap {\bf U}^-$. More precisely, \[\Ad\tilde{w}(\lu_{\alpha_i}(\calo) )= \lu_{x.\alpha_i}(\calp^{a_{\alpha_i}(w)})\] with $a_{\alpha_i}(w) \geq 1$ by the proof of Lemma \ref{conjI}.
Then \[\Ad\tilde{w}\tilde{s_i}(\lu_{\alpha_i}(\calo)) = \lu_{-x.\alpha_i}(\calp^{-a_{\alpha_i}(w)})\]  (recall that $a_{-\alpha}(w) = -a_{\alpha}(w)$).
Now, if $\alpha = -x.\alpha_i$ is not simple, then \[\Ad\tilde{w}\tilde{s_i}( \lu_{\alpha_i}(-t^{-1}) \in \operatorname{Ker}(\chi) \;(\text{respectively for $\chi'$)}.\]
If $\alpha$ is simple then, since we are assuming that $\tw \in \tilde{W}_a^1$, we have $a_{-\alpha_i}(w) \geq m_\alpha-m-1$ by Equation \eqref{wIm1}. Hence we see that $m_\alpha - m-1 \leq a_{-\alpha_i}(w) \leq -1$. This is possible only if $m_\alpha \leq m$. (The case where $l(ws_i) = l(w)+1$, $-x.\alpha_i$ is simple, $m_\alpha>m$ and $w\in W_a^1$ is not possible because if all these are true then the above inequality $m_\alpha - m-1 \leq a_{-\alpha_i}(w) \leq -1$ has to be satisfied, which is not possible).   Finally,
\begin{align*}
\chi(\Ad\tilde{w}\tilde{s_i}( \lu_{\alpha_i}(-t^{-1})))& = \chi_{-x.\alpha_i}(-c_{\alpha_i,w}\pi^{-a_{\alpha_i}(w)}t^{-1})\\& = \chi'_{-x.\alpha_i}(-c_{\alpha_i,w}'\pi'^{-a_{\alpha_i}(w)}t'^{-1}) = \chi'(\Ad\tilde{w}'\tilde{s_i}'( \lu_{\alpha_i}(-t'^{-1}))).
\end{align*}
To see the second equality above, note that since $ t^{-1} \sim_{\Lambda_0} t'^{-1}$, we have \[c_{\alpha_i,w}\pi^{-a_{\alpha_i}(w)} t^{-1} \sim_{\Lambda_{m_\alpha-m-1}} c_{\alpha_i,w}'\pi'^{-a_{\alpha_i}(w)} t'^{-1}\] and \[\chi_{-x.\alpha_i} \restriction \calp^{m_\alpha-m-1}/\calp^{m_\alpha}  = \chi'_{-x.\alpha_i} \restriction \calp'^{m_\alpha-m-1}/\calp'^{m_\alpha}.\]

We finally deal with the case $x = \tilde{s_0}$ where $s_0 = s_0^{(j)} \in S_2$ and $\alpha_0 = \alpha_0^{(j)}$ the highest root of $\Phi^{(j)}$. By Lemma \ref{heckeaction}, we have
\[ \kappa(\sigma(f_{\tilde{s_0}}) h_{\tilde{w}b}) = \kappa \left(\sum_{\{t_\nu\}}h_{\tilde{w}b\lu_{-\alpha_0}(\pi^{m+1}t_\nu)\tilde{s_0}^{-1}}\right).\]
Write $b = \lu_{-\alpha_0}(\pi t)c$, where $c \in I \cap\Ad\tilde{s_0}(I)$ and $t \in\calo^\times$. Then\[\tilde{w}b\lu_{-\alpha_0}(\pi^{m+1}t_\nu)\tilde{s_0}^{-1} = \tilde{w}\lu_{-\alpha_0}(\pi t) \tilde{s_0}^{-1}\Ad\tilde{s_0}(c\lu_{-\alpha_0}(\pi^{m+1}t_\nu)).\]
Fix $ c' \in I' \cap\Ad\tilde{s_0}(I')$ such that $ c \sim_{\beta}c'$ and $\Ad\tilde{s_0}(c) \sim_{\beta} \Ad\tilde{s_0}'(c')$. We also have \[\Ad\tilde{s_0}(\lu_{-\alpha_i}(\pi^{m+1}t_\nu)) \sim_\beta\Ad\tilde{s_0}(\lu_{-\alpha_0}(\pi'^{m+1}t_\nu')).\] We distinguish two cases:

\noindent \textbf{Case $1_0$}: Suppose $\Ad\tilde{w}(\lu_{-\alpha_0}(\pi t)) \in U$. Then
\[ \kappa(h_{\tilde{w}b\lu_{-\alpha_0}(\pi^{m+1}t_\nu)\tilde{s_0}^{-1}}) = \chi(\Ad\tilde{w}(\lu_{-\alpha_0}(\pi t)))\kappa(h_{\tilde{w}\tilde{s_0}h_{-\alpha_0}(-1)\Ad\tilde{s_0}(c\lu_{-\alpha_0}(\pi^{m+1}t_\nu))}).\]
Now, arguing as before, using Lemma \ref{rhoS}, we see that $\tilde{w}\tilde{s_0} = \widetilde{ws_0}a$ where $a$ is an element of order $\leq 2$ in $T$. Similarly,  $\tilde{w}'\tilde{s_0}' = \widetilde{ws_0}'a'$ and $ a \sim_{\beta}a'$. Hence
\[ \kappa(h_{\tilde{w}\tilde{s_0}h_{-\alpha_0}(-1)\Ad\tilde{s_0}(c\lu_{-\alpha_0}(\pi^{m+1} t_\nu))}) = h'_{\tilde{w}'\tilde{s_0}'h_{-\alpha_0}(-1)\Ad\tilde{s_0}'(c'\lu_{-\alpha_0}(\pi'^{m+1}t_\nu'))}.\]
Hence it remains to see that $\chi(\Ad\tilde{w}(\lu_{-\alpha_0}(\pi t))) = \chi'(\Ad\tilde{w}'(\lu_{-\alpha_0}(\pi' t'))) $.
Note that \[\Ad\tilde{w}(\lu_{-\alpha_0}(\pi t)) = \lu_{\alpha}(c_{\alpha_0,w}\pi^{1-a_{\alpha_0}(w)}t), \] where $\alpha  = -x.\alpha_0$, with $w = (\lambda, x)$. Now if $\alpha$ is not simple, then \[\lu_{\alpha}(c_{\alpha_0,w}\pi^{1-a_{\alpha_0}(w)}t) \in \operatorname{Ker}(\chi) \text{ (respectively for $\chi'$)}.\]
If $\alpha$ is simple, then $1-a_{\alpha_0}(w) \geq m_\alpha-m$ by Equation \eqref{wIm1} (here, note that $a_{-\alpha_0}(w) = -a_{\alpha_0}(w))$.  In the latter case,
\begin{align*}
 \chi(\Ad\tilde{w}(\lu_{\alpha_0}(\pi t))) = \chi_{-\alpha}(c_{\alpha_0,w}\pi^{1-a_{\alpha_0}(w)}t) = \chi'_{\alpha}(c_{\alpha_0,w}\pi'^{1-a_{\alpha_0}(w)}t')= \chi'(\Ad\tilde{w}'(\lu_{-\alpha_0}(\pi' t'))).
 \end{align*}
 This finishes \textit{Case $1_0$}.

\noindent \textbf{Case $2_0$}: Suppose $\Ad\tilde{w}(\lu_{-\alpha_0}(\pi t)) \notin U$. Then, $\Ad\tilde{w}\tilde{s_0}(\lu_{-\alpha_0}(\pi t)) \in U$. Using relation \eqref{relation1},
\begin{align*}
\tilde{w}b\lu_{-\alpha_0}&(\pi^{m+1}t)\tilde{s_0}^{-1} = \tilde{w}\lu_{-\alpha_0}(\pi t)\tilde{s_0}^{-1}\Ad\tilde{s_0}(c\lu_{-\alpha_0}(\pi^{m+1}t_\nu))\\
& = \tilde{w}\tilde{s_0}^{-1}\Ad\tilde{s_0}(\lu_{-\alpha_0}(\pi t))\Ad\tilde{s_0}(c\lu_{-\alpha_0}(\pi^{m+1}t_\nu))\\
& = \tilde{w}\tilde{s_0}^{-1} \lu_{-\alpha_0}(-\pi t^{-1}) \tilde{s_0}\alpha_0^\vee(-t^{-1})\lu_{-\alpha_0}(-\pi t^{-1})\Ad\tilde{s_0}(c\lu_{-\alpha_0}(\pi^{m+1}t_\nu))\\
& = \tilde{w}\tilde{s_0} \lu_{-\alpha_0}(-\pi t^{-1}) \tilde{s_0}^{-1}\alpha_0^\vee(-t^{-1})\lu_{-\alpha_0}(-\pi t^{-1})\Ad\tilde{s_0}(c\lu_{-\alpha_0}(\pi^{m+1}t_\nu))\\
& =\Ad\tilde{w}\tilde{s_0}( \lu_{-\alpha_0}(-\pi t^{-1}))\tilde{w}\alpha_0^\vee(-t^{-1})\lu_{-\alpha_0}( - \pi t^{-1})\Ad\tilde{s_0}(c\lu_{-\alpha_0}(\pi^{m+1}t_\nu)).
\end{align*}
Hence,
\begin{align*}
\gamma(&h_{\tilde{w}b\lu_{-\alpha_0}(\pi^{m+1}t_\nu)\tilde{s_0}^{-1}}) \\
&= \chi(\Ad\tilde{w}\tilde{s_0}( \lu_{-\alpha_0}(-\pi t^{-1})) \gamma(h_{\tilde{w}\alpha_0^\vee(-t^{-1})\lu_{-\alpha_0}(-\pi t^{-1})\Ad\tilde{s_0}(c\lu_{-\alpha_0}(\pi^{m+1}t_\nu))})\\
& = \chi(\Ad\tilde{w}\tilde{s_0}( \lu_{-\alpha_0}(-\pi t^{-1})) h'_{\tilde{w}'\alpha_0^\vee(-{t'}^{-1})\lu_{-\alpha_0}(-\pi' t'^{-1})\Ad\tilde{s_0}'(c'\lu_{-\alpha_0}(\pi'^{m+1}t'_\nu))}.
\end{align*}
Now, we only need to prove that \[ \chi(\Ad\tilde{w}\tilde{s_0}( \lu_{-\alpha_0}(-\pi t^{-1})) = \chi'(\Ad\tilde{w}'\tilde{s_0}'( \lu_{-\alpha_0}(-\pi' t'^{-1})).\]
If $l(ws_0) = l(w) - 1$, then $\Ad\tilde{w}\tilde{s_0}(U_{-\alpha_0, \calp}) \subset I$ by Lemma \ref{conjI}. Hence
\[\Ad\tilde{w}\tilde{s_0}( \lu_{-\alpha_0}(-\pi t^{-1})) \in {\bf U}(\calo)\] and we argue as in \textit{Case $1_0$} to finish the proof.\\
If $l(ws_0) = l(w) + 1$, then $\Ad\tilde{w}(\lu_{-\alpha_0}(\calp)) \subset I \cap {\bf U}^-.$ More precisely, \[\Ad\tilde{w}(\lu_{-\alpha_0}(\calp) )= \lu_{-x.\alpha_0}(\calp^{1-a_{\alpha_0}(w)})\] with $a_{\alpha_0}(w) \leq 0$. 
Then, $\Ad\tilde{w}\tilde{s_0}(\lu_{-\alpha_0}(\calp)) = \lu_{x.\alpha_0}(\calp^{-1+a_{\alpha_0}(w)})$
(recall that $a_{-\alpha}(w) = -a_{\alpha}(w)$).
Now, if $\alpha = x.\alpha_0$ is not simple, then \[\Ad\tilde{w}\tilde{s_0}( \lu_{-\alpha_0}(-\pi t^{-1})) \in \operatorname{Ker}(\chi)\;(\text{respectively for $\chi'$)}.\]
Otherwise, since we are assuming that $w \in W_a^1$,  we have $a_{\alpha_0}(w) \geq m_\alpha-m$ by Equation \eqref{wIm1}. This implies $m_\alpha - m \leq a_{\alpha_0}(w) \leq 0$. This is possible only if $m_\alpha \leq m$. (Note that the case $l(ws_0) = l(w)+1$, $x.\alpha_0$ is simple, $m_\alpha>m$ and $w \in W_a^1$ is not possible, since if all these are true then the above inequality $m_\alpha - m \leq a_{\alpha_0}(w) \leq 0$ has to be satisfied, which is impossible).  Finally
\begin{align*}
\chi(\Ad\tilde{w}\tilde{s_0}( \lu_{-\alpha_0}(-\pi t^{-1})))& = \chi_{x.\alpha_0}(-c_{\alpha_0,w}\pi^{-1+a_{\alpha_0}(w)}t^{-1})\\
& = \chi'_{x.\alpha_0}(-c_{\alpha_0,w}'\pi'^{-1+a_{\alpha_0}(w)}t'^{-1})\\
& = \chi'(\Ad\tilde{w}'\tilde{s_0}'( \lu_{-\alpha_0}(-\pi' t'^{-1}))).
\end{align*}
Again, to see the second equality above, note that since $t \sim_{\Lambda_0}t'$, we have $t^{-1} \sim_{\Lambda_0} t'^{-1}.$ Therefore, $ -c_{\alpha_0,w}\pi^{-1+a_{\alpha_0}(w)} t^{-1} \sim_{\Lambda_{m_\alpha-m-1}} -c_{\alpha_0,w}'\pi'^{-1+a_{\alpha_0}(w)} t'^{-1}$.
\end{proof}

\begin{Corollary}
 Let $(\tau, V)$ be an irreducible, admissible $\chi$-generic representation of $G$. Let $m\geq 1$ be large enough such that $\tau^{I_m} \neq 0$. Let $F'$ be $(m+1)$-close to $F$ and let $(\tau', V')$ be the representation of $G'$ obtained using the Hecke algebra isomorphism $\zeta$. Let $\chi'$ correspond to $\chi$ as before. Then $(\tau', V')$ is $\chi'$-generic.
\end{Corollary}
\begin{proof}  By Theorem \ref{genericrepresentationshecke}, we have that $ \kappa:(\sigma, E(\chi)^m) \rightarrow (\sigma', E(\chi')^m)$
is an isomorphism of $\mathbb{C}$-vector spaces that is compatible with the Hecke algebra isomorphism $\zeta:\mathscr{H}(G, I_m) \rightarrow \mathscr{H}(G',I_m')$.
Now, if $(\tau, V)$ is $\chi$-generic, we get an embedding $(\tau, V) \hookrightarrow E(\chi)$. Hence $ \tau^{I_m} \hookrightarrow E(\chi)^m$.  Denote its image as $\mathcal{W}(\tau, \chi)^{I_m}$.
Since $\tau$ is irreducible and $\tau^{I_m} \neq 0$, $\tau$ is in fact generated by its $I_m$-fixed vectors.  Therefore $\tau \cong  \mathscr{H}(G) \otimes_{\mathscr{H}(G, I_m)} \tau^{I_m}$ by Proposition \ref{categories}. Since $(\tau', V')$ corresponds to $(\tau, V)$ via $\zeta$ and $\kappa$ is compatible with $\zeta$, we get an embedding $\tau'^{I_m'}\hookrightarrow E(\chi')^m$. In fact its image in $E(\chi')^m$ is $\kappa(\mathcal{W}(\tau, \chi)^{I_m})$. Hence the Whittaker model of $(\tau', V')$ is  $\mathscr{H}(G') \otimes_{\mathscr{H}(G', I_m')} \kappa(\mathcal{W}(\tau, \chi)^{I_m})$.
\end{proof}
\subsection{Square integrability and formal degrees}\label{Discreteseries}
 Recall that $K_m := \Ker(\bfg(\calo) \rightarrow \bfg(\calo/\calp^m))$.  Let $\sigma$ be an irreducible, admissible representation of $G$ such that $\sigma^{K_{m}} \neq 0$. 
 Let $F'$ be another non-archimedean local field with $\calo'$, $\calp'$, and $\pi'$ defined accordingly.  Then, by Theorem \ref{Kaziso}, we know that there is an integer $l \geq m$ such that if $F$ and $F'$ are $l$-close, then we have a Hecke algebra isomorphism  $\Kaz_m: \mathscr{H}(G,K_m) \rightarrow \mathscr{H}(G', K_m')$. Note that we can take $l = m$ by Corollary \ref{proofKazConj}.
 Thus we obtain a representation $\sigma'$ of $G'$ such that $\gamma: \sigma^{K_m} \rightarrow \sigma'^{K_m'}$ is an isomorphism compatible with $\Kaz_m$.
  We fix a  Haar measure $dg$ on $G$ such that $\vol({ K_m}; dg) = 1$. Similarly, we fix a Haar measure $dg'$ on $G'$ such that $\vol({K_m'}; dg') = 1$. We also fix Haar measures $dz$ on $Z$ and $dz'$ on $Z'$ such that  $\vol({Z \cap  K_m}; dz) = 1 = \vol({Z' \cap K'_m}; dz')$.
  
Our aim in this section is to prove the following theorem:
\begin{theorem}\label{SCDS}\begin{enumerate}[(a)]
\item If $\sigma$ is a square integrable representation of $G$,  $\sigma'$ is a square integrable representation of $G'$.
\item If $\sigma$ is a supercuspidal representation of $G$,  $\sigma'$ is a supercuspidal representation of $G'$.
\end{enumerate}
\end{theorem}
 \begin{proof} Note that this theorem has already been proved for $\GL_n(F)$ and all its inner forms in \cite{Bad02}. Our proof here is a straight forward generalization of the proof of Theorem 2.17 of \cite{Bad02}. Let $v \in V^{K_m},   v^\vee \in (V^{\vee})^{K_m}$ such that $\langle v^\vee, v\rangle \neq 0$ where $\langle \cdot, \cdot \rangle$ is the natural pairing between $V$ and $V^{\vee}$. Define $h_\sigma(g) = \langle v^\vee, \sigma(g)v\rangle$. Then $\sigma$ is square integrable means that the central character of $\sigma$ is unitary and
 \[ \int_{Z \backslash G} |h_{\sigma}(g)|^2 d\dot{g} < \infty, \]
 where $d\dot{g}$ is the Haar measure on $Z\backslash G$ such that $dg = dz \;d\dot{g}$. 
 First note that $h_{\sigma}(kgk') =  \langle v^\vee, \sigma(kgk')v\rangle = \langle v^\vee, \sigma(g)v\rangle$. Hence, $h_\sigma$ is constant on the double coset $K_mgK_m$ and therefore, \[h_{\sigma}(g) = vol( K_mgK_m; dg)^{-1} \langle v^\vee, \sigma(t_g)v\rangle,\] where $t_g = \Char({K_mgK_m})$. Recall that 
\begin{equation}\label{cartan}G = \displaystyle{\coprod_{\lambda \in X_*({\bf T})_-}\coprod_{(a_i,a_j) \in T_{\lambda}} K_m a_i\pi_\lambda a_j^{-1}K_m }
\end{equation}
(with notation as in Section \ref{KazhdanIsomorphism}). The first step is to understand the action of $Z$ on a given coset in the decomposition above. The map $\lambda \rightarrow \pi_\lambda$ gives us an injection $X_*({\bf T}) \hookrightarrow T$ and we denote its image as $X_*({\bf T})_\pi$. Hence every element $z \in Z$ can be uniquely written as $z_1\pi_\mu$ where $z_1 \in {\bf Z}(\calo)$ and $\pi_\mu \in X_*({\bf Z})_\pi$. Then\\
\[ z.K_ma_i\pi_\lambda a_j^{-1}K_m = \begin{cases} K_ma_i\pi_{\mu + \lambda} a_j^{-1}K_m & if \;\;  z_1 \in Z \cap K_m,\\
K_mb_i\pi_{\mu +\lambda} b_j^{-1}K_m & if \;\;  z_1 \notin Z \cap K_m.
\end{cases}\]
Note that $(b_i,b_j)$ is an element of $T_{\mu +\lambda}$ and is uniquely determined by the class of $z_1 \mod (Z \cap K_m)$. 
 Let $A_{F}^0$ be a set of representatives of the elements of $ X_*({\bf T})_- /X_*({\bf Z})$ in $X_*({\bf T})_-$. 
Then 
\begin{align*} \int_{Z \backslash G} |h_{\sigma}(g)|^2 d\dot{g}= \displaystyle{\sum_{\lambda \in A_F^{0}}\sum_{(a_i,a_j) \in T_\lambda}} &(\#{\bf Z}(\calo/\calp^m))^{-1} \vol(Z \cap K_m; dz)^{-1} \\
& \vol(K_ma_i\pi_\lambda a_j^{-1}K_m; dg) |h_\sigma(a_i\pi_\lambda a_j^{-1})|^2.
\end{align*}

The isomorphism $\gamma: \sigma^{K_m} \rightarrow \sigma'^{K'_m}$ gives rise to an isomorphism between the dual spaces $\gamma^{\vee}: ({\sigma^{\vee}})^{K_m} \rightarrow  ({\sigma^{'\vee}})^{K'_m}$. Let $v' = \gamma(v)$ and $v'^\vee = \gamma^\vee(v^\vee)$ and define $h'_{\sigma'}(g') = v'^\vee(\sigma'(g')(v'))$. Note that $w_{\sigma'}$ is unitary since it corresponds to $w_{\sigma}$ via the Kazhdan isomorphism.  

We observe the following:\\
\textbullet \hspace{2pt} We need to prove that $\vol(K_ma_i\pi_\lambda a_j^{-1}K_m; dg) = \vol(K_m'a_i'\pi_\lambda' a_j'^{-1}K_m'; dg)$. For this, it suffices to prove that $\vol(K_m\pi_\lambda K_m; dg) = \vol(K_m'\pi_\lambda' K_m'; dg')$. Since $\vol(K_m; dg)  = 1 = \vol(K_m';  dg')$, we will prove that $\#\; K_m/(\pi_\lambda K_m\pi_\lambda^{-1} \cap K_m) = \#\; K_m'/({\pi'}_\lambda K_m'{\pi'}_\lambda^{-1} \cap K_m' )$.
Writing down the Iwahori factorization of $K_m$ as
\[ K_m  = \displaystyle{ \prod_{\alpha \in \Phi^+} U_{\alpha, \calp^m}\;\;T_{\calp^m} \;\; \prod_{\alpha \in \Phi^+} U_{-\alpha , \calp^m}},\]
we have
\[ \pi_\lambda K_m\pi_{\lambda}^{-1} = \displaystyle{ \prod_{\alpha \in \Phi^+} U_{\alpha, \calp^{m+\langle\alpha,\lambda\rangle}}\;\;T_{\calp^m}\;\; \prod_{\alpha \in \Phi^+} U_{\alpha , \calp^{m -\langle\alpha,\lambda\rangle}}}, \]
and
\begin{align*}
K_m \cap  \pi_\lambda K_m\pi_{\lambda}^{-1} = & \displaystyle{ \prod_{\alpha \in \Phi^+, <\alpha,\lambda> \geq 0} U_{\alpha, \calp^{m+\langle\alpha,\lambda\rangle}} \prod_{\alpha \in \Phi^+, \langle\alpha, \lambda\rangle <0} U_{\alpha, \calp^m}\;\;T_{\calp^m} }\\&
 \displaystyle{  \prod_{\alpha \in \Phi^+, \langle\alpha,\lambda\rangle <0 } U_{\alpha , \calp^{m -\langle\alpha,\lambda\rangle}}  \prod_{\alpha \in \Phi^-, \langle\alpha,\lambda\rangle\geq 0} U_{\alpha , \calp^m}}.
 \end{align*}  
 Now it is clear that
 \[\#\; K_m/(\pi_\lambda K_m\pi_\lambda^{-1} \cap K_m) =  \prod_{\alpha \in \Phi^+} q^{|\langle\alpha,\lambda\rangle|} = \#\; K_m'/({\pi'}_\lambda K_m'{\pi'}_\lambda^{-1} \cap K_m' ). \]
\textbullet \hspace{2pt} Let ${f'}_{g'}$ denote the characteristic function of $K_m'g'K_m'$. By definition of  $v'$ and $v'^\vee$, it follows that
\begin{align*}
h'_{\sigma'}(a_i'\pi_\lambda'a_j'^{-1}) &= v'^\vee(\sigma'((a_i'\pi_\lambda'a_j'^{-1})(v'))\\
& =\vol(K_m'a_i'\pi_\lambda'a_j'^{-1}K_m'; dg)^{-1} v'^\vee(\sigma'({f'}_{a_i'\pi_\lambda'a_j'^{-1}}(v')) \\&
=  \vol(K_ma_i\pi_\lambda a_j^{-1}K_m;dg)^{-1}v^\vee(\sigma(t_{(a_i\pi_\lambda a_j^{-1}})(v))\\& = h_\sigma((a_i\pi_\lambda a_j^{-1})).
\end{align*} 
\textbullet \hspace{2pt} Since $F$ and $F'$ are $m$-close, we have that $ \#{\bf Z}(\calo/\calp^m) =  \#{\bf Z'}(\calo'/\calp'^m)$.\\
\textbullet \hspace{2pt} By our choice of the Haar measure $dz$ and $dz'$, we have $\vol(Z \cap K_m; dz) =1 = \vol(Z' \cap K_m'; dz')$.

Combining all of the above, we see that $h'_{\sigma'}$ is also square integrable mod center. Hence $\sigma'$ is a square integrable representation.

To prove $(b)$, just observe that $h_\sigma$ has compact support mod $Z$ simply means that there are finitely many $\lambda's$ in $A_F^0$ such that $h_\sigma$ is non-zero on the coset $K_ma_i\pi_\lambda a_j^{-1}K_m$. By definition of $h'_{\sigma'}$, it is clear that $h_{\sigma'}$ is nonzero precisely on the corresponding cosets of $G'$ mod $Z'$.
    \end{proof}
For a square integrable representation $\sigma$, there is a real number $d(\sigma)$, depending only on $\sigma$ and the measure $d\dot{g}$, such that
\begin{equation}\label{formaldegree} \displaystyle{\int_{Z \backslash G} \langle v^\vee, \sigma(g) v \rangle  \langle w^\vee, \sigma(g^{-1}) w \rangle\; d\dot{g} = d(\sigma)^{-1} \langle v^\vee, w \rangle \langle w^\vee, v \rangle}
\end{equation}
for all $v,w \in V$ and $v^\vee, w^\vee \in V^\vee$. The real number $d(\sigma)$ is called the \textit{formal degree} of $\sigma$. 

\begin{Corollary}
If $\sigma$ is square integrable, and $\sigma'$ corresponds to $\sigma$ via the Kazhdan isomorphism, then $d(\sigma) = d(\sigma')$ where the measure on $G'$ is chosen as in the beginning of the section.
\end{Corollary}
 \begin{proof}

Recall that a square integrable representation is unitary. In fact, for $\displaystyle{v_0^{\vee}} \in V^{\vee}$   
\[( v, w)_{v_0^{\vee}} = \displaystyle{\int_{Z\backslash G}  \langle v_0^\vee, \sigma(g)v \rangle \overline{\langle v_0^\vee,  \sigma(g)w\rangle}\; d\dot{g}  }\]
is a positive definite $G$-invariant Hermitian form on $V$. Hence there is a complex anti-linear $G$-isomorphism  $\Theta_{v_0^\vee}: (\sigma,V) \rightarrow (\sigma^\vee, V^\vee)$ such that $( v, w)_{v_0^{\vee}} = \langle\Theta_{v_0^\vee}(v), w \rangle$. Schur's lemma implies that $(\cdot , \cdot)_{v_0^\vee}$ is the unique $G$-invariant Hermitian form on $V$, up to a positive constant factor. Choose $v_0^\vee \in (V^\vee)^{K_m}$. Let $\Theta = \Theta_{v_0^\vee}$ and $(\cdot, \cdot) = (\cdot, \cdot)_{v_0^\vee}$. Let $v,w \in V^{K_m}$ and let $v^\vee = \Theta(w)$ and $w^\vee = \Theta(v)$. Then formula \eqref{formaldegree} simplifies as
\[ \displaystyle{\int_{Z \backslash G} |(w, \sigma(g) v )|^2}  \; d\dot{g} = d(\sigma)^{-1} (v,v)(w,w).\]

As before, let $v' = \gamma(v)$ and $w' = \gamma(w)$. With $\Theta' = \Theta_{v_0'^\vee}$ and $(\cdot, \cdot)'$ the corresponding inner product,  it is easy to see that $v'^\vee = \gamma^{\vee}(v^\vee) = \Theta'(w')$ and  $w'^\vee = \gamma^{\vee}(w^\vee) = \Theta'(v')$. From the proof of Theorem \ref{SCDS}, we have
\[ \displaystyle{\int_{Z \backslash G} |(w, \sigma(g) v )|^2}\;   d\dot{g} = \displaystyle{\int_{Z \backslash G} |(w', \sigma(g') v' )'|^2}  \; d\dot{g'}.\]
  Moreover, $(v,v) = (v',v')'$ and $(w,w) = (w',w')'$. Hence we see that $d(\sigma) = d(\sigma')$.
       \end{proof} 
\subsection{Parabolic Induction}\label{IndRep}
We retain the notation of Section \ref{stdnotations}. So $\bfg$ is a split connected reductive group defined over $\Z$ with Chevalley basis $\{\lu_\alpha| \alpha \in \Phi\}$. Let $\theta \subset \Delta$ and let ${\bf P = P_\theta}$ be the standard parabolic subgroup with Levi ${\bf M = M_\theta}$ and unipotent radical ${\bf N = N_\theta} \subset \bfu$. Note that $\{\lu_\alpha | \alpha \in \Phi_\theta\}$ is a Chevalley basis for $\bfm_\theta$. Let $I$ be the standard Iwahori subgroup of $G$ and $I_m$ be the $m$-th Iwahori filtration subgroup of $G$.  We write $I_M = M \cap I$ and $I_{m,M}=M \cap I_m$ for the corresponding Iwahori and the $m$-th Iwahori congruence subgroup of $M$ respectively. Let $\sigma$ be an irreducible, admissible representation of $M$ such that $\sigma^{I_{m,M}} \neq 0$. Set $l = m+3$. Let $F'$ be another non-archimedean local field such that $F'$ is $l$-close to $F$. Let $\pi$ and $\pi'$ be uniformizers of $F$ and $F'$ respectively with $\pi \sim_\Lambda \pi'$. The representatives of elements in the extended affine Weyl group are chosen using the fixed Chevalley basis as in Section \ref{reps}. Let $\zeta_m: \mathscr{H}(G,I_m) \rightarrow \mathscr{H}(G',I_m')$ be the isomorphism in Theorem \ref{closealgebra}. Similarly, let $\zeta_{l,M}: \mathscr{H}(M,I_{l,M}) \rightarrow \mathscr{H}(M',I_{l,M'}')$  be as in Theorem \ref{closealgebra}.   Let $\sigma'$ be the representation of $M'$ such that $\kappa_{l,M}: \sigma^{I_{l,M}} \rightarrow \sigma'^{I_{l,M'}'}$ is an isomorphism. In this section, we will construct an isomorphism $\kappa_m:(\Ind_P^G \sigma)^{I_m} \rightarrow (\Ind_{P'}^{G'} \sigma')^{I'_m}$ of $\mathbb{C}$-vector spaces that is compatible with the Hecke algebra isomorphism $\zeta_m: \mathscr{H}(G,I_m) \rightarrow \mathscr{H}(G',I_m')$ (Also check Proposition 3.15 of \cite{BHLS10} for an alternate proof for $\GL_n$).

 We fix a Haar measure $dg$ on $G$ such that $\vol(I_{l}; dg) = 1$, and a Haar measure $dm$ on $M$ satisfying $\vol(I_{l,M}; dm) = 1$. Similarly, we choose Haar measures on $G'$ and $M'$ such that their corresponding $l$-th Iwahori filtration subgroups have volume 1.

  Let us first recall the following lemma.
  \begin{lemma}\label{casselman} Let $\theta$ be a non-empty subset of $\Delta$. There exists in any right coset of $W_\theta$ in $W$, a unique element $w$ characterized by any of these properties:
  \begin{enumerate}[(a)]
  \item For $x \in W_\theta, \; l(xw) = l(x) + l(w)$,
  \item $w^{-1} \theta >0$,
  \item The element $w$ is of least length in $W_\theta w$.
  \end{enumerate}
  \end{lemma}
  \begin{proof}
  This is Lemma 1.1.2 of \cite{Cas95}.
  \end{proof}
  Hence, $[W_\theta \backslash W] = \{w \in W \,|\, w^{-1}\theta >0\}$. As in Section \ref{reps}, we choose representatives $\tilde{w}$ of $w \in [W_\theta \backslash W]$.
  \begin{lemma}\label{conjMP} For $w \in [W_\theta \backslash W]$, we have
  \begin{itemize}
  \item $M \cap \tilde{w}I\tilde{w}^{-1}  = M \cap I$,
  \item $P \cap \tilde{w}I\tilde{w}^{-1} \subset P \cap I$.
  \end{itemize}
  \end{lemma}
  \begin{proof}
  Consider the Iwahori factorization of $I$:
 \[ I  = \displaystyle{ \prod_{\alpha \in \Phi^+} U_{\alpha, \calo} \;\;{\bf T}(\calo) \;\; \prod_{\alpha \in \Phi^-} U_{\alpha , \calp}}.\]
 Now a simple computation yields
   \begin{align*}
P\cap \tilde{w} I \tilde{w}^{-1} = &\displaystyle{ \prod_{\alpha \in \Phi^+, w.\alpha \in \Phi_\theta^+ }  U_{w.\alpha, \calo}  \prod_{\alpha \in \Phi^+, w.\alpha \in \Phi^+ \backslash \Phi_\theta^+ } U_{w.\alpha, \calo} \prod_{\alpha \in \Phi^+, w.\alpha \in \Phi_\theta^- } U_{w.\alpha, \calo}} \\
&{\bf T}(\calo)\\
&  \displaystyle{ \prod_{\alpha \in \Phi^+, w.\alpha \in \Phi_\theta^+ }  U_{-w.\alpha, \calp}  \prod_{\alpha \in \Phi^+, w.\alpha \in \Phi_\theta^- } U_{-w.\alpha, \calp} \prod_{\alpha \in \Phi^+, w.\alpha \in \Phi^- \backslash \Phi_\theta^- } U_{-w.\alpha, \calp}},
 \end{align*}
 and
 \begin{align*}
 M \cap \tilde{w} I \tilde{w}^{-1} &= \displaystyle{ \prod_{\alpha \in \Phi^+, w.\alpha \in \Phi_\theta^+ }  U_{w.\alpha, \calo}  \prod_{\alpha \in \Phi^+, w.\alpha \in \Phi_\theta^- } U_{w.\alpha, \calo}} \;{\bf T}(\calo)\\
&  \displaystyle{ \prod_{\alpha \in \Phi^+, w.\alpha \in \Phi_\theta^+ }  U_{-w.\alpha, \calp}  \prod_{\alpha \in \Phi^+, w.\alpha \in \Phi_\theta^- } U_{-w.\alpha, \calp} }.
 \end{align*}
 Since $w \in [W_\theta \backslash W]$, we have $w^{-1} \theta >0$. Hence $\{\alpha \in \Phi^+ \,|\,w.\alpha \in \Phi_\theta^- \} = \emptyset$ and the lemma follows.
  \end{proof}
  Note that this lemma clearly holds when $I$ is replaced by the $m$-th Iwahori congruence subgroup $I_m$. 

Next, we construct a map $\kappa_m:  (\Ind_P^G \sigma)^{I_m} \rightarrow (\Ind_{P'}^{G'} \sigma')^{I'_m}$.  First, note that
$G = P\bfg(\calo) = \displaystyle{\coprod_{w \in [W_\theta \backslash W]}}P\tilde{w}I$. Here, $\tilde{w}$ denotes the lifting of $w$ using a minimal decomposition, as in Section \ref{reps}.
 Let $R(\tilde{w})$ be the system of representatives of $(I \cap \tilde{w}^{-1}P\tilde{w}) \backslash I / I^m$ in $I$.
 Then, $G = \displaystyle{\coprod_{w \in [W_\theta \backslash W]} \coprod_{b \in R(\tilde{w})}}P\tilde{w}b I_m$.   By Lemma \ref{ind}, we know
 \begin{align*}
   (\Ind_P^G  \sigma)^{I_{m}} &\longrightarrow \displaystyle{\prod_{w \in [W_\theta \backslash W]} \prod_{b \in R(\tilde{w})} \sigma^{M \cap \tilde{w}   I_m\tilde{w}^{-1}}},\\
   h& \longrightarrow h(\tilde{w}b),
    \end{align*}
    is an isomorphism as $\mathbb{C}$-vector spaces.  Hence an element $h \in (\Ind_P^G  \sigma)^{I_{m}}$ is completely determined by its values $h(\tilde{w}b), \; w \in [W_\theta \backslash W],\; b \in R(w)$.

Via the isomorphism $\zeta_{l,M}: \mathscr{H}(M, I_{l,M}) \rightarrow \mathscr{H}(M', I_{l,M'}')$, we obtain $\kappa_{l,M}: \sigma^{I_{l,M}} \rightarrow (\sigma')^{I_{l,M'}'} $ the corresponding isomorphism of the modules over these Hecke algebras.  Hence $\kappa_{{m,M}}: \sigma^{I_{m,M}} \rightarrow (\sigma')^{I_{m,M'}'} $ is also an isomorphism. Moreover, for $v \in \sigma^{I_{m,M}},\; \kappa_{{l,M}}(v) = \kappa_{{m,M}}(v)$. Note that $ \sigma^{M \cap \tilde{w}I_m \tilde{w}^{-1} } = \sigma^{M \cap I_m}$ by Lemma \ref{conjMP}.
 Hence we obtain a map
  \begin{align*}
  (\Ind_P^G  \sigma)^{I_{m}} &\overset{\kappa_m}{\longrightarrow} (\Ind_{P'}^{G'} \sigma')^{I_{m}'},\\
  h & \rightarrow h',
  \end{align*}
  where $h'(\tilde{w}'b') = \kappa_{m,M}(h(\tilde{w}b))$ for $w \in [W_\theta \backslash W]$ and  $\;b \sim_{\beta} b'$.
The following lemma is clear.
\begin{lemma}\label{Indvsiso} The map $\kappa_m:  (\Ind_P^G \sigma)^{I_m} \rightarrow (\Ind_{P'}^{G'} \sigma')^{I'_m}$ constructed above is an isomorphism of $\mathbb{C}$-vector spaces.\qedhere
\end{lemma}

  It remains to prove that this map $\kappa_m$ is compatible with the Hecke algebra isomorphism $\zeta_m:\mathscr{H}(G,I_m) \rightarrow \mathscr{H}(G',I_m')$.

  \begin{Remark}\label{Remark1}
  A simple consequence of assuming that the fields are $l$-close is the following. We have $ I_{l,M} \subset M \cap \tilde{w}_0I_m\tilde{w}_0^{-1}$ and  $ I_{l,M} \subset M \cap \tilde{w}_0\tilde{s}_0I_m\tilde{s}_0^{-1}\tilde{w}_0^{-1} \; \forall \; w_0 \in W$. Therefore, $0 \neq \sigma^{I_{m,M}} \subset \sigma^{I_{l,M}},\;\sigma^{M \cap \tilde{w}_0I_{m}\tilde{w}_0^{-1}} \subset \sigma^{I_{l,M}}$ and $\sigma^{M \cap \tilde{w}_0\tilde{s}_0I_{m}\tilde{s}_0^{-1}\tilde{w}_0^{-1}} \subset \sigma^{I_{l,M}}$.
 Consequently, we can also study the subspaces $\sigma^{M \cap \tilde{w}_0I_{m}\tilde{w}_0^{-1}}$ and $\sigma^{M \cap  \tilde{w}_0\tilde{s}_0I_{m}\tilde{s}_0^{-1}\tilde{w}_0^{-1}}  \; \forall \; w_0 \in W$ using the isomorphism $\kappa_{{l,M}}: \sigma^{I_{l,M}} \rightarrow (\sigma')^{I_{l,M'}'} $.
\end{Remark}

For simplicity, we assume for the remainder of this section that $\gder$ is simple, that is the underlying root system is irreducible. Let $s_i = s_{\alpha_i}, i=1,2,\ldots n$, denote the simple roots of $\bfg$ and let $s_0 = (-\alpha_0^\vee, s_{\alpha_0})$ denote the highest root.  We note that all the arguments would go through even without this assumption.
  \begin{lemma}\label{indclf} Let $h \in (\Ind_P^G  \sigma)^{I_{m}}$. Let $h' \in (\Ind_{P'}^{G'}  \sigma')^{I_{m}'}$ correspond to $h$ as in  Lemma \ref{Indvsiso}. Then
  \begin{enumerate}[(a)]
  \item  $\kappa_{l,M}(h(\tilde{w}_0b)) = h'({\tilde{w}'}_0b')\; \forall\;  w_0 \in W,\; b \in I$.
  \item Let $s_0 = (-\alpha_0^\vee, s_{\alpha_0})$ and $\tilde{s}_0 = \pi_{-\alpha_0^{\vee} }\lw_{\alpha_0}(1)$ as in Section \ref{reps}. Then  \[ \kappa_{l,M}(h(\tilde{w}\tilde{s}_0b)) = h'({\tilde{w}'}\tilde{s}_0'b')\; \forall\;  w \in [W_\theta \backslash W],\; b \in I.\]
  \item Let $\rho\in \{\tilde{\rho}_1, \ldots \tilde{\rho}_l, \tilde{\rho}_1^{-1}, \ldots \tilde{\rho_l}^{-1}, \tilde{\mu}_1, \ldots \tilde{\mu}_k\}$ and $\tilde{\rho} = \pi_{\lambda}\tilde{x}$ as in Section \ref{reps}(b). Then \[\kappa_{l,M}(h(\tilde{w}\tilde{\rho}b)) = h'({\tilde{w}'}\tilde{\rho}'b')\; \forall\;  w \in [W_\theta \backslash W],\; b \in I.\]
  \end{enumerate}
  \end{lemma}
  \begin{proof}Notice that from the construction of $h'$, we have $h'(\tilde{w}'b') = \gamma(h(\tilde{w}b))$ for $w \in [W_\theta \backslash W],  b \in R(\tilde{w})$. Now, let $ w_0 \in W$. We can write $w_0 = x.w$ for $x \in W_\theta$ and $w \in [W_\theta \backslash W]$. As always, let $\tilde{x}$ and $\tilde{w}$ denote the liftings of $x$ and $w$ as in Section \ref{reps}. Since $l(w_0) = l(x) + l(w)$ by Lemma \ref{casselman}, we see that $\tilde{w}_0 = \tilde{x}.\tilde{w}$. Similarly, we have ${\tilde{w}'}_0 = \tilde{x}'.\tilde{w}'$.
  Now,
  \[ h'(\tilde{w}_0'b') = h'(\tilde{x}'\tilde{w}'b') = \sigma'(\tilde{x}')h'(\tilde{w}'b').\]
  We have \[ h'(\tilde{w}_0'b')  \in  (\sigma')^{M' \cap \tilde{w}'_0I'_{m}\tilde{w}_0'^{-1}} \subset (\sigma')^{I'_{l,M'}},\] and similarly,  \[h'(\tilde{w}'b')  \in  (\sigma')^{M' \cap \tilde{w}'I'_{m}\tilde{w}'^{-1}} \subset (\sigma')^{I'_{l,M'}}\] by Remark \ref{Remark1}.
  Hence, $ \sigma'(i'\tilde{x}'i'')h'(\tilde{w}'b') = \sigma'(\tilde{x}')h'(\tilde{w}'b')  \; \forall \; i', i''\in I_{l,M'}'$. As before, for $g' \in M'$,  let $f'_{g',M'} $ be the characteristic function of the double coset $I_{l,M'}'g' I_{l,M'}'$. Then
  \begin{align*}
      h'(\tilde{w}_0'b') &=\sigma'(\tilde{x}')h'(\tilde{w}'b')\\
      &=\vol(I'_{l,M'}\tilde{x}'I'_{l,M'}; dm)^{-1}\sigma'(f'_{\tilde{x}',M'})h'(\tilde{w}'b')\\
      & = q^{-l(x)}\sigma'(f'_{\tilde{x}',M'})h'(\tilde{w}'b') \text{ (by Lemma \ref{volW})}.
      \end{align*}
  Now, we write $b' = b_1'b_2'b_3'$, where $b_1' \in \tilde{w}'^{-1}M'\tilde{w}' \cap I'$, $b_2' \in \tilde{w}'^{-1}N'\tilde{w}' \cap I'$,  and $b_3' \in R(\tilde{w}')$. Hence $\tilde{w}'b_1'\tilde{w}'^{-1} \in M' \cap \tilde{w}'I\tilde{w}'^{-1} = M' \cap I'$  and $\tilde{w}'b_2'\tilde{w}'^{-1} \in N' \cap \tilde{w}'I\tilde{w}'^{-1} \subset N' \cap I'$ by Lemma \ref{conjMP}.
  Then,
 \begin{align*}
  h'(\tilde{w}_0'b') &=q^{-l(x)}\sigma'(f'_{\tilde{x}',M'}) h'(\tilde{w}'b_1'b_2'b_3')\\
  & = q^{-l(x)}\sigma'(f'_{\tilde{x}',M'}) \sigma'(\tilde{w}'b_1'b_2'\tilde{w}'^{-1})h'(\tilde{w}'b_3') \\
  &= q^{-l(x)}\sigma'(f'_{\tilde{x}',M'}) \sigma'(\tilde{w}'b_1'\tilde{w}'^{-1})h'(\tilde{w}'b_3') \text{ (since $\tilde{w}'b_2'\tilde{w}'^{-1} \in N'$)}\\
  & =q^{-l(x)} \sigma'(f'_{\tilde{x}',M'}* f'_{\tilde{w}'b_1'\tilde{w}'^{-1},M'})h'(\tilde{w}'b_3').
  \end{align*}
  Let $x = s_{i_1}....s_{i_r}$ be its reduced expression. Then $\tilde{x} = \tilde{s}_{i_1}...\tilde{s}_{i_r}$ and $f_{\tilde{x},M} = f_{\tilde{s}_{i_1},M}*\ldots *f_{\tilde{s}_{i_r},M}$. It immediately follows that $f'_{\tilde{x}',M'} = f_{\tilde{s}_{i_1}',M'}*\ldots *f_{\tilde{s}_{i_r}',M'} = \zeta_{l,M}(f_{\tilde{s}_{i_1},M}*\ldots * f_{\tilde{s}_{i_r},M})$ from Section \ref{CLF}. Also, $b_1 \sim_{\beta} b_1'$ implies that $\tilde{w}b_1\tilde{w}^{-1} \sim_{\beta}\tilde{w}'b_1'\tilde{w}'^{-1}$. Hence, $\zeta_{l,M}(f_{\tilde{w}b_1\tilde{w}^{-1},M} )= f'_{\tilde{w}'b_1'\tilde{w}'^{-1},M'}$.
Hence we finally have that
 \begin{align*}
  h'(\tilde{w}_0'b') &=q^{-l(x)} \sigma'(f'_{\tilde{x}',M'}* f'_{\tilde{w}'b_1'\tilde{w}'^{-1},M'})h'(\tilde{w}'b_3') \\
  & = q^{-l(x)}\sigma'(\zeta_{l,M}( f_{\tilde{x},M}* f_{\tilde{w}b_1\tilde{w}^{-1},M}))\kappa_{l,M}(h(\tilde{w}b_3))\\
  & =q^{-l(x)} \sigma'(\zeta_{l,M}( f_{\tilde{x},M}* f_{\tilde{w}b_1\tilde{w}^{-1},M}))(\kappa_{l,M}( \sigma(\tilde{w}b_2\tilde{w}^{-1})h(\tilde{w}b_3)))\\
  & =q^{-l(x)} \kappa_{l,M}(\sigma( f_{\tilde{x},M}* f_{\tilde{w}b_1\tilde{w}^{-1},M})h(\tilde{w}b_2b_3))\\
  & = \kappa_{l,M}(h(\tilde{w}_0b)).
  \end{align*}
  For $(b)$, we have
\begin{align*}
h'(\tilde{w}'\tilde{s}_0'b') &= h'(\tilde{w}'{\pi'}_{-\alpha_0^\vee} \lw_{\alpha_0}(1)'b') \\
&= \sigma'(\tilde{w}' {\pi'}_{-\alpha_0^\vee} \tilde{w}'^{-1})h'(\tilde{w}'\lw_{\alpha_0}(1)'b') \\
&= q^{-l(-w.\alpha_0^\vee)}\sigma'(f'_{\tilde{w}' {\pi'}_{-\alpha_0^\vee} \tilde{w}'^{-1},M'})h'(\tilde{w}'\lw_{\alpha_0}(1)'b')\\
&(\text{since }  h'(\tilde{w}'\tilde{s}_0'b'), h'(\tilde{w}'\lw_{\alpha_0}(1)'b') \in \sigma'^{I'_{l,M'}})\\
& = q^{-l(-w.\alpha_0^\vee)}\sigma'(f'_{\tilde{w}' {\pi'}_{-\alpha_0^\vee} \tilde{w}'^{-1},M'}) \kappa_{l,M}(h(\tilde{w}\lw_{\alpha_0}(1)b) \;\;(\text{using } (a)).
\end{align*}
To complete the proof of $(b)$, it remains to see that \[\zeta_{l,M}(f_{\tilde{w}{\pi}_{-\alpha_0^\vee} \tilde{w}^{-1},M}) = f'_{\tilde{w}' {\pi'}_{-\alpha_0^\vee} \tilde{w}'^{-1},M'}.\] By Remark \ref{Remark2}, we have that
\[ I'_{l,M'}\tilde{w}{\pi'}_{-\alpha_0^\vee} \tilde{w}'^{-1}I'_{l,M'} = I'_{l,M'}{\pi'}_{-w. \alpha_0^\vee} I'_{l,M'} = I'_{l,M'}\tilde{w}_a'c' I'_{l,M'}\] where $w_a = (-w.\alpha_0^{\vee},1) \in X_*(\bft) \rtimes W_\theta$ and $c'$ is the constant in Remark \ref{Remark2}. Now it follows that
\[ f'_{\tilde{w}' {\pi'}_{-\alpha_0^\vee} \tilde{w}'^{-1},M'} = f'_{\tilde{w}_a'c',M'} = \zeta_{l,M}( f_{\tilde{w}_ac,M}) = \zeta_{l,M}(f_{\tilde{w}{\pi}_{-\alpha_0^\vee} \tilde{w}^{-1},M}).\]
For $(c)$ write $\tilde{\rho} = \pi_{\lambda}\tilde{x}$ and argue exactly as in $(b)$ to complete the proof.
  \end{proof}
  The final step is to understand the action of the generators of the Hecke algebra $\mathscr{H}(G,I_m)$ on an element $h \in (\Ind_P^G \sigma)^{I_m}$.
  \begin{lemma}\label{heckeaction1} Let $h \in (\Ind_P^G \sigma)^{I_m}$. Let $f_g =\vol(I_m; dg)^{-1}. \Char(I_mgI_m)$. Let $w \in [W_\theta \backslash W]$ and $b \in I$. 
  \begin{enumerate}[(a)]
  \item For $b_0 \in I$, we have \[(f_{b_0} *h)(\tilde{w}b) = h(\tilde{w}bb_0).\]
  \item For $ \tilde\rho \in \{\tilde{\rho}_1, \ldots \tilde{\rho}_l, \tilde{\rho}_1^{-1}, \ldots \tilde{\rho_l}^{-1}, \tilde{\mu}_1, \ldots \tilde{\mu}_k\}$, where $\rho = ( \lambda,x)$ and $\tilde{\rho} = \pi_{\lambda}\tilde{x}$, we have \[(f_{\tilde\rho} *h)(\tilde{w}b)=  h(\tilde{w}\tilde{\rho}\Ad\tilde{\rho}(b)).\]
  \item For $i \geq 1$, writing $b = \lu_{\alpha_i}(t)b_1$, we have

 $ (f_{\tilde{s}_i} *h)(\tilde{w}b) =$\\
  \begin{displaymath}
   \left\{ \begin{array}{lll}
  \displaystyle{\sum_{\{t_\nu\}}\sigma(f_{\tilde{w}\lu_{\alpha_i}(t)\tilde{w}^{-1},M})h(\tilde{w}\tilde{s}_i\Ad\tilde{s}_i^{-1}(\lu_{\alpha_i}({\pi^mt_\nu})b_1)) } &:  w.\alpha_i \in \Phi_\theta^+,\\
   \displaystyle{\sum_{\{t_\nu\}}h(\tilde{w}\tilde{s}_i\Ad\tilde{s}_i^{-1}(\lu_{\alpha_i}({\pi^mt_\nu})b_1)) }&:w.\alpha_i \in \Phi^+ \backslash \Phi_\theta^+,\\
   \displaystyle{\sum_{\{t_\nu\}}h(\tilde{w}\alpha_i^\vee(-t)\lu_{\alpha_i}(-t^{-1})\Ad\tilde{s}_i^{-1}(\lu_{\alpha_i}({\pi^mt_\nu})b_1)) } &:w.\alpha_i \in \Phi^-. \\
  \end{array}
  \right.
  \end{displaymath}
where $\{t_\nu\}$ is a  set of representatives of $\calo/\calp$. 
  \item For $i = 0$, writing $b = \lu_{-\alpha_0}(\pi t)b_1$, we have

 $ (f_{\tilde{s}_0} *h)(\tilde{w}b) =$\\
\hspace*{-1in} \begin{displaymath}
   \left\{ \begin{array}{lll}
  \displaystyle{\sum_{\{t_\nu\}}}h(\tilde{w}\tilde{s}_0\Ad\tilde{s}_0^{-1}(\lu_{-\alpha_0}({\pi^{m+1}t_\nu})b_1))&: -w.\alpha_0 \in \Phi^+,\\
 \displaystyle{\sum_{\{t_\nu\}}}\sigma(f_{\Ad\tilde{w}(\lu_{-\alpha_0}(\pi t)),M})h(\tilde{w}\tilde{s}_0\Ad\tilde{s}_0^{-1}(\lu_{-\alpha_0}({\pi^{m+1}t_\nu})b_1))&  : -w.\alpha_0 \in \Phi_\theta^-,\\
\displaystyle{\sum_{\{t_\nu\}}}h(\tilde{w}\alpha_0^\vee(t^{-1})\lu_{-\alpha_0}(-\pi t^{-1})\Ad\tilde{s}_0^{-1}(\lu_{-\alpha_0}({\pi^{m+1}t_\nu})b_1))&  :otherwise.\\
  \end{array}
  \right.
  \end{displaymath}
where $\{t_\nu\}$ is a set of representatives of $\calo/\calp$.
  \end{enumerate}

  \end{lemma}
  \begin{proof} For $g \in G$, a simple calculation yields
  \begin{equation}\label{equation}
  (f_g *h)(\tilde{w}b) = (\vol(I_m; dg))^{-1}\displaystyle{\int_{\tilde{w}bI_mgI_m}h(x)\; dx}.
    \end{equation}
    To prove $(a)$, we observe that $\tilde{w}bI_mb_0I_m = \tilde{w}bb_0I_m$. Since $h$ is fixed by $I_m$, we immediately deduce that $(f_{b_0} *h)(\tilde{w}b) = h(\tilde{w}bb_0)$ and $(a)$ follows. \\
For $(b)$, we again observe that $\tilde{\rho}$ normalizes $I $ and $I_m$. Hence,  $\tilde{w}bI_m\tilde{\rho}I_m = \tilde{w}b\tilde{\rho}I_m$. Hence, integral \eqref{equation} becomes
\[  (f_{\tilde\rho} *h)(\tilde{w}b) = h(\tilde{w}b\tilde{\rho}) = h(\tilde{w}\tilde{\rho}\Ad\tilde{\rho}( b)),\]
and  $(b)$ follows.

To prove $(c)$, we first write $b= \lu_{\alpha_i}(t)b_1$ with $t \in \calo^\times$ and $b_1 \in I \cap Ad\tilde{s}_i(I)$, and note that
\begin{align*}
\tilde{w}bI_m\tilde{s}_iI_m &= \displaystyle{\coprod_{\{t_\nu\}} \tilde{w}\lu_{\alpha_i}(t)\lu_{\alpha_i}({\pi^mt_\nu})\tilde{s}_i\Ad\tilde{s}_i^{-1}(b_1)I_m}\\
& = \displaystyle{\coprod_{\{t_\nu\}} \tilde{w}\lu_{\alpha_i}(t)\tilde{s}_i\Ad\tilde{s}_i^{-1}(\lu_{\alpha_i}({\pi^mt_\nu})b_1)I_m}.
\end{align*}
Hence, Integral \eqref{equation} becomes
\[ (f_{\tilde{s}_i} *h)(\tilde{w}b) = \displaystyle{\sum_{\{t_\nu\}}h(\tilde{w}\lu_{\alpha_i}(t)\tilde{s}_i\Ad\tilde{s}_i^{-1}(\lu_{\alpha_i}({\pi^mt_\nu})b_1) }.\]
\textbullet \hspace{2pt} If $w. \alpha_i \in \Phi_\theta^+$, then
\begin{align*}
 (f_{\tilde{s}_i} *h)(\tilde{w}b) &= \displaystyle{\sum_{\{t_\nu\}}\sigma(\tilde{w}\lu_{\alpha_i}(t)\tilde{w}^{-1})h(\tilde{w}\tilde{s}_i\Ad\tilde{s}_i^{-1}(\lu_{\alpha_i}({\pi^mt_\nu})b_1) )}\\
&=\displaystyle{\sum_{\{t_\nu\}}\sigma(f_{\tilde{w}\lu_{\alpha_i}(t)\tilde{w}^{-1},M})h(\tilde{w}\tilde{s}_i\Ad\tilde{s}_i^{-1}(\lu_{\alpha_i}({\pi^mt_\nu})b_1)) .}
\end{align*}
The second equality above holds since $\tilde{w}\lu_{\alpha_i}(t)\tilde{w}^{-1} \in M \cap I$.\\
\textbullet \hspace{2pt} If $w.\alpha_i \in \Phi^+ \backslash \Phi_\theta^+$, then $\tilde{w}\lu_{\alpha_i}(t)\tilde{w}^{-1} \in N$ and hence,
\begin{align*}
 (f_{\tilde{s}_i} *h)(\tilde{w}b) &= \displaystyle{\sum_{\{t_\nu\}}\sigma(\tilde{w}\lu_{\alpha_i}(t)\tilde{w}^{-1})h(\tilde{w}\tilde{s}_i\Ad\tilde{s}_i^{-1}(\lu_{\alpha_i}({\pi^mt_\nu})b_1)) }\\
&=\displaystyle{\sum_{\{t_\nu\}}h(\tilde{w}\tilde{s}_i\Ad\tilde{s}_i^{-1}(\lu_{\alpha_i}({\pi^mt_\nu})b_1)). }
\end{align*}
\textbullet \hspace{2pt} When $w.\alpha_i \in \Phi^{-}$, note that  $w.\alpha_i \notin \Phi_\theta^-$ since $w^{-1}(\Phi_\theta^+) >0$ for all $w \in [W_\theta \backslash W]$. Hence $w.\alpha_i \in \Phi^- \backslash \Phi_\theta^-$. We now use relation \ref{relation} to obtain
\begin{align*}
\tilde{w}\lu_{\alpha_i}(t)\tilde{s}_i\Ad\tilde{s}_i^{-1}&(\lu_{\alpha_i}({\pi^mt_\nu})b_1)\\ &=\tilde{w}\tilde{s}_i\Ad\tilde{s}_i^{-1}(\lu_{\alpha_i}(t))\Ad\tilde{s}_i^{-1}(\lu_{\alpha_i}({\pi^mt_\nu})b_1)\\
& =  \tilde{w}\tilde{s}_i\Ad\tilde{s}_i(\lu_{\alpha_i}(t))\Ad\tilde{s}_i^{-1}(\lu_{\alpha_i}({\pi^mt_\nu})b_1)\\
& = \tilde{w}\tilde{s}_i\lu_{\alpha_i}(-t^{-1})\tilde{s}_i\alpha_i^\vee(t)\lu_{\alpha_i}(-t^{-1})\Ad\tilde{s}_i^{-1}(\lu_{\alpha_i}({\pi^mt_\nu})b_1)\\
& = \tilde{w}\tilde{s}_i\lu_{\alpha_i}(-t^{-1})\tilde{s}_i^{-1}\alpha_i^\vee(-t)\lu_{\alpha_i}(-t^{-1})\Ad\tilde{s}_i^{-1}(\lu_{\alpha_i}({\pi^mt_\nu})b_1)\\
& =\Ad\tilde{w}\tilde{s}_i(\lu_{\alpha_i}(-t^{-1}))\tilde{w}\alpha_i^\vee(-t)\lu_{\alpha_i}(-t^{-1})\Ad\tilde{s}_i^{-1}(\lu_{\alpha_i}({\pi^mt_\nu})b_1).
\end{align*}

Moreover, since $w.\alpha_i \in \Phi^- \backslash \Phi_\theta^-$, we have $\Ad\tilde{w}\tilde{s}_i(\lu_{\alpha_i}(t)) \in N$. Therefore
\begin{align*}
& (f_{\tilde{s}_i} *h)(\tilde{w}b) \\
&= \displaystyle{\sum_{\{t_\nu\}}\sigma(\Ad\tilde{w}\tilde{s}_i(\lu_{\alpha_i}(-t^{-1})))h(\tilde{w}\alpha_i^\vee(-t)\lu_{\alpha_i}(-t^{-1})\Ad\tilde{s}_i^{-1}(\lu_{\alpha_i}({\pi^mt_\nu})b_1)) }\\
&=\displaystyle{\sum_{\{t_\nu\}}h(\tilde{w}\alpha_i^\vee(-t)\lu_{\alpha_i}(-t^{-1})\Ad\tilde{s}_i^{-1}(\lu_{\alpha_i}({\pi^mt_\nu})b_1)). }
\end{align*}

To prove $(d)$, we first write $b= \lu_{-\alpha_0}(\pi t)b_1$ with $t \in \calo^\times$, $b_1 \in I \cap Ad \tilde{s}_0(I)$, and note that
\begin{align*}
\tilde{w}bI_m\tilde{s}_0I_m &= \displaystyle{\coprod_{\{t_\nu\}} \tilde{w}\lu_{-\alpha_0}(\pi t)\lu_{-\alpha_0}({\pi^{m+1}t_\nu})\tilde{s}_0\Ad\tilde{s}_0^{-1}(b_1)I_m}\\
& = \displaystyle{\coprod_{\{t_\nu\}} \tilde{w}\lu_{-\alpha_0}(\pi t)\tilde{s}_0\Ad\tilde{s}_0^{-1}(\lu_{-\alpha_0}({\pi^{m+1}t_\nu})b_1)I_m}.
\end{align*}
Hence, Integral \eqref{equation} becomes
\[ (f_{\tilde{s}_0} *h)(\tilde{w}b) = \displaystyle{\sum_{\{t_\nu\}}h(\tilde{w}\lu_{-\alpha_0}(\pi t)\tilde{s}_0\Ad\tilde{s}_0^{-1}(\lu_{-\alpha_0}({\pi^{m+1}t_\nu})b_1)) .}\]
\textbullet \hspace{2pt} If $-w.\alpha_0 \in \Phi^+$, then $-w. \alpha_0 \notin \Phi_\theta^+$ since $w^{-1}(\Phi_\theta^+)>0$. Hence,  $-w.\alpha_0 \in \Phi^+ \backslash \Phi_\theta^+$. This implies $\tilde{w}\lu_{-\alpha_0}(\pi t)\tilde{w}^{-1} \in N$
and therefore,
\begin{align*}
 (f_{\tilde{s}_0} *h)(\tilde{w}b) &= \displaystyle{\sum_{\{t_\nu\}}\sigma(\tilde{w}\lu_{-\alpha_0}(\pi t)\tilde{w}^{-1})h(\tilde{w}\tilde{s}_0\Ad\tilde{s}_0^{-1}(\lu_{-\alpha_0}({\pi^{m+1}t_\nu})b_1)) }\\
&=\displaystyle{\sum_{\{t_\nu\}}h(\tilde{w}\tilde{s}_0\Ad\tilde{s}_0^{-1}(\lu_{-\alpha_0}({\pi^{m+1}t_\nu})b_1)). }
\end{align*}
\textbullet \hspace{2pt} If $-w.\alpha_0 \in \Phi_\theta^-$ then we note that $\Ad\tilde{w}(\lu_{-\alpha_0}(\pi t)) \in M \cap I$ and therefore,
\begin{align*}
 (f_{\tilde{s}_i} *h)(\tilde{w}b) &= \displaystyle{\sum_{\{t_\nu\}}\sigma(\Ad\tilde{w}(\lu_{-\alpha_0}(\pi t)))h(\tilde{w}\tilde{s}_0\Ad\tilde{s}_0^{-1}(\lu_{-\alpha_0}({\pi^{m+1}t_\nu})b_1)) }\\
&=\displaystyle{\sum_{\{t_\nu\}}\sigma(f_{\Ad\tilde{w}(\lu_{-\alpha_0}(\pi t)),M})h(\tilde{w}\tilde{s}_0\Ad\tilde{s}_0^{-1}(\lu_{-\alpha_0}({\pi^{m+1}t_\nu})b_1)). }
\end{align*}
\textbullet \hspace{2pt} If $-w.\alpha_0\in \Phi^- \backslash \Phi_\theta^-$ then $-ws_{\alpha_0}.\alpha_0\in \Phi^+ \backslash \Phi_\theta^+$. We now use relation \eqref{relation1} to obtain
 \begin{align*}
&\tilde{w}\lu_{-\alpha_0}(\pi t)\tilde{s}_0\Ad\tilde{s}_0^{-1}(\lu_{-\alpha_0}({\pi^{m+1}t_\nu})b_1)\\
 &= \tilde{w}\tilde{s}_0\Ad\tilde{s}_0^{-1}(\lu_{-\alpha_0}(\pi t))\Ad\tilde{s}_0^{-1}(\lu_{-\alpha_0}({\pi^{m+1}t_\nu})b_1)\\
& =  \tilde{w}\tilde{s}_0\Ad\tilde{s}_0(\lu_{-\alpha_0}(\pi t))\Ad\tilde{s}_0^{-1}(\lu_{-\alpha_0}({\pi^{m+1}t_\nu})b_1)\\
& = \tilde{w}\tilde{s}_0\lu_{-\alpha_0}(-\pi t^{-1})\tilde{s}_0\alpha_0^\vee(-t^{-1})\lu_{-\alpha_0}(-\pi t^{-1})\Ad\tilde{s}_0^{-1}(\lu_{-\alpha_0}({\pi^{m+1}t_\nu})b_1)\\
& = \tilde{w}\tilde{s}_0\lu_{-\alpha_0}(-\pi t^{-1})\tilde{s}_0^{-1}\alpha_0^\vee(t^{-1})\lu_{-\alpha_0}(-\pi t^{-1})\Ad\tilde{s}_0^{-1}(\lu_{-\alpha_0}({\pi^{m+1}t_\nu})b_1)\\
& =\Ad\tilde{w}\tilde{s}_0(\lu_{-\alpha_0}(-\pi t^{-1}))\tilde{w}\alpha_0^\vee(t^{-1})\lu_{-\alpha_0}(-\pi t^{-1})\Ad\tilde{s}_0^{-1}(\lu_{-\alpha_0}({\pi^{m+1}t_\nu})b_1).
\end{align*}

 Again, since $-ws_{\alpha_0}.\alpha_0\in \Phi^+ \backslash \Phi_\theta^+$, we have $\Ad\tilde{w}\tilde{s}_0(\lu_{-\alpha_0}(\pi t^{-1})) \in N$. Therefore
\begin{align*}
 (f_{\tilde{s}_0} *h)(\tilde{w}b)= \displaystyle{\sum_{\{t_\nu\}}}&
 \sigma(\Ad\tilde{w}\tilde{s}_0(\lu_{-\alpha_0}(-\pi t^{-1})))\\
 &h(\tilde{w}\alpha_0^\vee(t^{-1})\lu_{-\alpha_0}(-\pi t^{-1})\Ad\tilde{s}_0^{-1}(\lu_{-\alpha_0}({\pi^{m+1}t_\nu})b_1)) \\[5mm]
=\displaystyle{\sum_{\{t_\nu\}}}&h(\tilde{w}\alpha_0^\vee(t^{-1})\lu_{-\alpha_0}(-\pi t^{-1})\Ad\tilde{s}_0^{-1}(\lu_{-\alpha_0}({\pi^{m+1}t_\nu})b_1)).\qedhere
\end{align*}

  \end{proof}
  \begin{theorem}
Let $m  \geq 1$ be such that $\sigma^{I_{m,M}} \neq 0$. Let $l = m+3$ and assume $F$ and $F'$ are $l$-close. Let $\zeta_{m}$ and $\zeta_{l,M}$ be the isomorphisms as in the beginning of this subsection. Let $\sigma'$ be the representation of $M'$ obtained using $\zeta_{l,M}$.  Then for each $f \in \mathscr{H}(G, I_m)$, the following diagram commutes.
\begin{displaymath}
    \xymatrix{
      (\Ind_P^G  \sigma)^{I_{m}} \ar[r]^f \ar[d]_{\kappa_m} & (\Ind_P^G  \sigma)^{I_{m}} \ar[d]^{\kappa_m} \\
        (\Ind_{P'}^{G'}  \sigma')^{I_{m}'} \ar[r]_{\zeta_m(f)}       &(\Ind_{P'}^{G'}  \sigma')^{I_{m}'} }
\end{displaymath}
\end{theorem}
\begin{proof}
 Let $ h \in  (\Ind_P^G  \sigma)^{I_{m}}$. We need to prove that $\kappa_m(f * h) = \zeta_m(f)*\kappa_m(h)$. Since an element $h \in (\Ind_P^G  \sigma)^{I_{m}}$ is completely determined by its values $h(\tilde{w}b),\; w \in [W_\theta \backslash W], b \in R(w)$, we only need to prove that $(\kappa_m(f * h))(\tilde{w}'b') =( \zeta_m(f)*\kappa_m(h))(\tilde{w}'b'),\;\;  w \in [W_\theta \backslash W], b \in R(w).$
This follows from Lemma
 \ref{indclf} and Lemma \ref{heckeaction1}.
\end{proof}

\section{The Langlands-Shahidi local coefficients over close local fields}\label{LocalCoefficients}
We retain the notation of Section \ref{stdnotations} in this section. 

\subsection{Unramified characters}\label{unrchar}
Fix $\theta \subset \Delta$. Let $H_\theta:\mtheta \rightarrow \cala_\theta$ be defined by
\[q^{\langle \lambda, H_\theta(m)\rangle} = |\lambda(m)|, \; m \in \mtheta, \lambda \in X^*(\bfm_\theta).\]
Let $\mtheta^1:=\{m \in M_\theta\; | \; |\lambda(m)| = 1 \text{ for all } \lambda \in X^*(\bfm_\theta)\}$. 
Then we know that $\mtheta/\mtheta^1$ is a free abelian group  of rank equal to $\dim(\bfa_\theta)$.

We write $\Unr(\mtheta)$ for the group of unramified characters of $\mtheta$. We recall that $\Unr(\mtheta)$ consists of all continuous characters $\eta:\mtheta \rightarrow \C^\times$ trivial on $\mtheta^1$.
Thus we may identify
\[ \Unr(\mtheta) = \Hom_\Z(\mtheta/\mtheta^1, \C^\times) \cong {(\C^\times)}^{\dim{\bfa_\theta}}.\]

The above equality enables us to define a canonical structure of a complex algebraic torus on $\Unr(\mtheta)$. Its algebra of regular functions $\C[\Unr(\mtheta)]$ is generated as a $\C$-algebra by the characters $\eta \rightarrow \eta(m)$, where $m$ runs over $\mtheta$.

Next, note that we can define an unramified character for each $\nu \in \cala_{\theta, \C}^*$ by setting
\[ \eta_\nu(m) = q^{\langle \nu, H_\theta(m)\rangle}.\]
In fact, the map $\nu \rightarrow \eta_\nu$ is an epimorphism $\cala_{\theta, \C}^* \rightarrow \Unr(\mtheta)$ with kernel a discrete subgroup of $\cala_{\theta, \C}^*$ consisting of all $\nu \in \cala_{\theta, \C}^*$ such that $\langle \nu, H_\theta(m) \rangle \in (2\pi i/\log q)\Z \; \forall \; m \in \mtheta$.

Everything above is compatible with the action of the Weyl group. Let $w_0 \in W$ be such that $\Omega:=w_0(\theta)$ is a subset of $\Delta$. If $\tw_0$ is a  representative of $w_0$ as in Section \ref{reps}, then $M_{\Omega} = \tw_0M_\theta \tw_0^{-1}$ and $\cala_\Omega^* = w_0(\cala_\theta^*)$. The conjugation by $\tw_0$ induces an isomorphism of algebraic groups
$\Unr(\mtheta) \rightarrow \Unr(M_\Omega)$, $\eta \rightarrow w_0(\eta)$, where $w_0(\eta)(\tw_0m\tw_0^{-1}) = \eta(m)$ for $m \in \mtheta$. 
Note that $w_0(\eta_\nu) = \eta_{w_0(\nu)}$.

\subsection{Induced representations}\label{Induced}
For the rest of Section \ref{LocalCoefficients}, we will assume $\bfm_\theta$ is maximal, and let $\alpha$ be the simple root of $\Delta$ such that  $\theta = \Delta \backslash \{\alpha\} \subset \Delta$. Let $w_0 = w_{l, \Delta}w_{l, \theta}$ and let $\Omega = w_0(\theta)$. Let $\bfp_\Omega$ denote the standard parabolic with Levi subgroup $\bfm_{\Omega}$ and unipotent radical $\bfn_{\Omega}$.  
Let $(\sigma, V)$ be an irreducible, admissible representation of $\mtheta$. For $\eta \in \Unr(\mtheta)$,  let
\[ I(\eta, \sigma) = \ind_{P_\theta}^G \sigma \eta\]
denote the normalized parabolically induced representation. 
Let $\tw_0$ denote the lifting of $w_0$ as in Section \ref{reps}. Let $w_0(\sigma)$ be an irreducible, admissible smooth representation of $\momega$ defined by
\[w_0(\sigma)(m)(v) = \sigma(\tw_0^{-1}m \tw_0)(v)\text{ for all }v \in V.\]
and let
\[I(w_0(\nu), w_0(\sigma)) = \ind_{P_{\Omega}}^G w_0(\sigma\eta_\nu).\]
Let $\calu_\sigma = \{ \eta \in \Unr(\mtheta)| I(\eta, \sigma) \text { is irreducible} \}$. Then $\calu_\sigma$ is a non-empty Zariski open subset of $\Unr(\mtheta)$ (cf. Theorem 3.2 of \cite{Sau97} and Remark 1.8.6.2 of \cite{Yu09}).
Define
\[ \calu_1 = \calu_\sigma \cap w_{0}^{-1}\calu_{w_0(\sigma)}\]
Then $\calu_1$ is a non-empty Zariski open subset of $\Unr(\mtheta)$ and for $\eta \in \calu_1$,
\[ I(\eta, \sigma) \text{ and } I(w_0(\eta), w_0(\sigma)) \text{ are both irreducible. }\]

Using the surjection $\cala_{\theta, \C}^* \rightarrow \Unr(\mtheta)$, $\nu \rightarrow \eta_\nu$, we sometimes write $I(\nu, \sigma)$ instead of $I(\eta_\nu, \sigma)$. By the above, there exists an open dense subset $\calv_1 $ of $\cala_{\theta, \C}^*$ such that
\[ I(\nu, \sigma)   \text{ and } I(w_0(\nu), w_0(\sigma)) \text{ are both irreducible. }\]

The theory of local coefficients arise from a study of certain intertwining operators between these induced representations, and the uniqueness of their Whittaker models. Let us briefly review this theory before proving our main theorem about local coefficients for representations that correspond over close local fields.

\subsection{Whittaker functionals}\label{WF} Given a generic representation $(\sigma, V)$ of $M_\theta$, there is a  Whittaker functional associated to the representation $I(\nu, \sigma)$ making the induced representation generic. Let us recall the definition of this Whittaker functional.

Let $\chi: U \rightarrow \C^\times$ be a generic character of $U$ as in Section \ref{GR}  and let $\chi_\theta$ denote its restriction to $U_\mtheta :=  U \cap \mtheta$.  Assume $\chi$ and $\tilde{w_0}$ are compatible, that is,
\[\chi(\tw_0u\tw_0^{-1}) = \chi(u)\; \forall u \in U_\mtheta,\]
where $\tw_0$ is the lifting of $w_0$ as described in Section \ref{reps}. 
Assume $(\sigma, V)$ is $\chi_{\theta}$-generic and let $\lambda: V \rightarrow \C$ be a non-zero Whittaker functional on $V$ satisfying
\[\lambda(\pi(u)v) = \chi_\theta(u)\lambda(v)\; \forall u \in U_\mtheta , v \in V.\]
With $\chi$ as above, the induced representation $I(\nu, \sigma)$ is $\chi$-generic. More precisely, we have the following:
\begin{proposition}[Proposition 3.1 of \cite{Sha81}] Given $f \in I(\nu, \sigma)$, the integral
\begin{equation}\label{Whit}
 \lambda_\chi(\nu, \sigma)(f) = \int_{N_\Omega}\lambda(f(\tw_0^{-1}n)) \overline{\chi(n)} dn
\end{equation}
 is convergent and consequently defines a Whittaker functional for the space $V(\nu, \sigma)$. This is an entire function of $\nu$, and there exists a function $f \in V(\nu, \sigma)$ such that $\lambda_\chi(\nu, \sigma)f$ is non-zero.
\end{proposition}
In fact, $\lambda_\chi(\nu, \sigma)$ is a polynomial in $\eta_\nu$ (See Section 1.2 of \cite{Lom09}).
\subsection{Intertwining operators}\label{IO}
We retain the assumptions of Section \ref{Induced}. We shall recall the theory of intertwining integrals. Given $f \in V(\nu, \sigma)$, define
\begin{align}\label{inter}
 A(\nu, \sigma, \tw_0)f(g) = A(\eta_\nu, \sigma, \tw_0)f(g)= \int_{U \cap \tw_0\bar{N}_\theta \tw_0^{-1}}f(\tw_0^{-1}ng)dn.
\end{align}
Here,  $\bar{N}_\theta$ is the unipotent radical of the opposite parabolic $\bar{P}_\theta$ of $P_\theta$ (that is $ \bar{P}_\theta = \mtheta \bar{N}_\theta)$.  Since $P_\theta$ is maximal and $w_0 = w_{l, \Delta}w_{l, \theta}$, a simple calculation shows that 
\[  \tw_0\bar{N}_\theta \tw_0^{-1} = N_{\Omega} \subset U,\] 
and this can be substituted in Equation \eqref{inter}.
This integral converges absolutely whenever the following condition holds:
\begin{equation}\label{conv}
 \langle Re(\nu), \beta^\vee\rangle >>0 \text{ for each $\beta \in \Phi^+ \backslash \Phi_\theta^+$}.
\end{equation}

For such $\nu$, $A(\nu, \sigma, \tw_0)f \in V(w_0(\nu), w_0(\sigma))$. Moreover, this is a meromorphic function of $\nu$, and  in fact a rational function of $\eta_\nu$ (Section 2 of \cite{Sha81} and Theorem IV.I.I of \cite{Wal03}). Away from its poles, 
\[A(\nu, \sigma, \tw_0) \in \Hom_G(I(\nu, \sigma), I(w_0(\nu), w_0(\sigma)).\] In fact, Mui\'{c} \cite{Mui08} constructed explicitly a Zariski open dense subset $\mathcal{U}(\sigma, w_0)$ of $\Unr(M_\theta)$ where the intertwining operator is defined (cf. Lemma 4.6, Remark 4.16 and Theorem 5.6 (ii) of  \cite{Mui08}), that is, 
\begin{equation}\label{dimensionMuic1}\dim_\C \;\Hom_G(I(\nu, \sigma), I(w_0(\nu), w_0(\sigma)) =1.
\end{equation}
Let $\calv(\sigma, w_0)$ denote the corresponding open dense subset of $\cala_{\theta, \C}^*$ such that $A(\nu, \sigma,\tw_0)$ is defined for all $\nu \in \calv(\sigma, w_0)$. Let $[W_\theta\backslash W/W_\Omega] = \{w \in W| w^{-1} \theta >0, w\Omega >0\}$. Then
\[G = \coprod_{w \in [W_\theta\backslash W /W_\Omega]} P_\theta w P_\Omega\] and there is a total order $\leq$ on $[W_\theta\backslash W /W_\Omega]$ such that for each $ w \in [W_\theta\backslash W /W_\Omega]$, the set 
\[G^{\leq w} :=\displaystyle{\cup_{w_1 \leq w} P_\theta w_1 P_\Omega}\] is open (cf. Section 3 of \cite{Mui08}). With $ I(\nu, \sigma)^{\leq w_0^{-1}} = \{f \in I(\nu, \sigma) | \supp(f) \subset G^{\leq w_0^{-1}}\}$,  he showed that for each $ \eta_\nu \in \mathcal{U}(\sigma, w_0)$ the intertwining operator is determined by the following requirement:
\begin{equation}\label{muicio}
A(\nu, \sigma, \tw_0)(f)(1) = \int_{N_{\Omega}}f(\tw_0^{-1}n)dn, \;\; f \in I(\nu, \sigma)^{\leq w_0^{-1}}.
\end{equation}
(See Equation (4.20) and Lemma 4.21 of \cite{Mui08}).

\subsection{Local coefficients}\label{LCdefn}

Let us recall the following theorem:
\begin{theorem}[Theorem 3.1 of \cite{Sha81}]
 There exists a complex number $C_\chi(\nu, \sigma,\tw_0)$ such that
\begin{align}\label{localcoeff}
\lambda_\chi(\nu, \sigma) = C_\chi(\nu, \sigma,\tw_0) \lambda_\chi(w_0(\nu), w_0(\sigma)) \circ A(\nu, \sigma, \tw_0).
  \end{align}
  Furthermore, as a function of $\nu$, it is meromorphic in $\cala_{\theta,\C}^*$, and its value depends only on the class of $\sigma$.
\end{theorem}
The scalar $C_\chi(\nu, \sigma, \tw_0)$ is called the \textit{local coefficient} associated to $\nu$ and $\sigma$. In fact, it can be shown that $C_\chi(\nu, \sigma, \tw_0)$ is a rational function of $\eta_\nu$ (cf. Theorem 2.1 of \cite{Lom09}).

\subsection{The main theorem}\label{maintheorem}
Our aim in this section is to prove that the local coefficients are compatible with the Deligne-Kazhdan correspondence. Recall that $I$ is the Iwahori subgroup of $G$, $I_m$ is the $m$-th Iwahori filtration subgroup of $G$. Similarly,  $I_\mtheta = \mtheta \cap I$ and $I_{m,\mtheta}=\mtheta \cap I_m$ denote the corresponding Iwahori and the $m$-th Iwahori filtration subgroup of $\mtheta$ respectively.
Fix a character $\chi$ of $U$ as in Section \ref{WF} and let $(\sigma, V)$ be an irreducible, admissible, $\chi_\theta$-generic representation of $\mtheta$. Let $\calw(\sigma, \chi_\theta)$ be its Whittaker model and $\lambda$ be the Whittaker functional that determines this model. 
 More precisely, if $W_v$ denotes the image of the vector $v$ under the embedding $\sigma \hookrightarrow \Ind_{U_\mtheta}^\mtheta \chi_\theta$, then
$W_v(m) = \lambda(\sigma(m)v), m \in M_\theta.$

Let $m \geq 1$. We impose the following condition on $m$.
\begin{Condition}\label{Condition 1}
There exists $0 \neq v_0 \in V^{I_{m,\mtheta}}$ such that $\lambda(v_0) \neq 0$, and additionally $\cond(\chi_\alpha) \leq m\; \forall\; \alpha \in \Delta.$
\end{Condition}
For each $\alpha \in \theta$ and  $\lu_\alpha(x_\alpha) \in U_\alpha \cap I_{m, \mtheta}$, we have $\chi_\alpha(x_\alpha) \lambda(v_0) = \lambda(\lu_\alpha(x_\alpha) v_0) = \lambda(v_0)$. If $\lambda(v_0) \neq 0$, this implies that $\cond(\chi_\alpha) \leq m$ for each $\alpha \in \theta$. 
Hence Condition \ref{Condition 1} is equivalent to 
\begin{Condition}\label{Condition 2}
There exists $0 \neq v_0 \in V^{I_{m,\mtheta}}$ such that $\lambda(v_0) \neq 0$, and additionally 
$\cond(\chi_\alpha) \leq m\; \forall\; \alpha \in \Delta \backslash \theta.$
\end{Condition}

  Recall that we have assumed $\chi$ is compatible with $\tw_0$.  Set $l = m+4$. Let $F'$ be another non-archimedean local field such that $F'$ is $l$-close to $F$.
Let $\zeta_{l,\mtheta}: \mathscr{H}(\mtheta,I_{l,\mtheta}) \rightarrow \mathscr{H}(\mtheta',I_{l,\mtheta'}')$ be the Hecke algebra isomorphism as in Theorem \ref{closealgebra}. 
Since $\sigma^{I_{m,\mtheta}} \neq 0$, we obtain a representation $\sigma'$ of $\mtheta'$ such that $\kappa_{l,\mtheta}: \sigma^{I_{l,\mtheta}} \rightarrow \sigma'^{I_{l,\mtheta'}'}$ is an isomorphism. 
Let $\chi'$ be a generic character of $U'$ that corresponds to $\chi$ as in Section \ref{GR}. By the results of that section, we know that $(\sigma', V')$ is $\chi_\theta'$-generic.
 Let $\calw(\sigma',\chi_\theta')$ be the Whittaker model of $\sigma'$ corresponding to $\calw(\sigma, \chi_\theta)$, and let $\lambda'$ be the Whittaker functional corresponding to $\lambda$.

 Before proving our main theorem, we fix our Haar measures suitably and make some essential observations.

Let $dm_\theta$, $dm_\theta'$, $dm_\Omega$, and $dm_\Omega'$ be Haar measures on $\mtheta$, $\mtheta'$, $M_\Omega$, and $M_\Omega'$ respectively, such that 
\[\vol(\mtheta \cap I, dm_\theta) = \vol(\mtheta' \cap I', dm_\theta') = \vol(M_\Omega \cap I, dm_\Omega) = \vol(M_\Omega' \cap I', dm_\Omega').\]

Also, we assume that the measures  $dn_\theta$, $dn_\theta'$, $dn_\Omega$ and $dn_\Omega'$  on $N_\theta$, $N_\theta'$, $N_\Omega$, and $N_\Omega'$ satisfy
\[ \vol(N_\theta \cap I, dn_\theta) = \vol(N_\theta' \cap I', dn_\theta') \text{ and } \vol(N_\Omega \cap I, dn_\Omega) = \vol(N_\Omega' \cap I', dn_\Omega').\]
Note that
\[ \bar{N}_\theta = \tilde{w}_0^{-1}{N}_\Omega \tilde{w}_0;\;\; \bar{N}_{\Omega} = \tilde{w}_0N_\theta\tilde{w}^{-1}_0\]
Using the above, we fix Haar measures $d\bar{n}_\theta$ and $d\bar{n}_\Omega$ on $\bar{N}_\theta$ and  $\bar{N}_\Omega$ respectively by transport of structure, and similarly, we fix Haar measures $d\bar{n}_\theta'$ and $d\bar{n}_\Omega'$ on $\bar{N}_\theta'$ and $\bar{N}_\Omega'$ by transport of structure.

\noindent \textbf{Observations:}
\begin{enumerate}[(a)]
\item By the choice of measures made above, we have
\begin{align*}
\vol(\bar{N}_\theta \cap I; d\bar{n}_\theta) &=\vol(N_\Omega \cap \tilde{w}_0I\tilde{w}_0^{-1}; dn_\Omega)= \vol(N_\Omega \cap I_1; dn_\Omega)\\
& =\frac{\vol(N_\Omega \cap I, dn_\Omega)}{\#{\bf N}_\Omega(\calo/\calp)}=\frac{\vol(N_\Omega' \cap I', dn_\Omega')}{\#{\bf N}_\Omega(\calo'/\calp')} =  \vol(\bar{N}_\theta' \cap I'; d\bar{n}_\theta').\end{align*}

Similarly,
\begin{equation*}
\vol(\bar{N}_\Omega \cap I, d\bar{n}_\Omega)
= vol(\bar{N}'_\Omega \cap I', d\bar{n}_\Omega').
\end{equation*}
Now, it is clear that for all $m \geq 1$, 
\begin{equation}\label{voloppuni}\vol(\bar{N}_\Omega \cap I_m, d\bar{n}_\Omega) = \vol(\bar{N}'_\Omega \cap I_m', d\bar{n}'_\Omega);\;\; \vol(\bar{N}_\theta \cap I_m; d\bar{n}_\theta) = \vol(\bar{N}'_\theta \cap I_m'; d\bar{n}'_\theta)
\end{equation}
\item Let $\tau$ be an irreducible, admissible representation of $G$ such that $\tau^{I_m} \neq 0$ for some $m \geq 1$. Then we also have that $\tau^{I_{m+1}} \neq 0$ . Since $F'$ is $(m+1)$-close to $F$ we obtain an irreducible, admissible representation $\tau'$ using the Hecke Algebra isomorphism with $\kappa_{m+1}: \tau^{I_{m+1}}\rightarrow \tau'^{I_{m+1}'}$. Note that for $w \in W$, the Weyl group of $G$, we have $I_{m+1} \subset \tilde{w}I_{m}\tilde{w}^{-1}  $ and $I_{m+1} \subset I_{m}$. This  implies that if $\kappa_{m+1}(v) = v'$ for $v \in \tau^{I_m}$, then $\kappa_{m+1}(\tilde{w}.v) = \tilde{w}'.v'$.

\item In Section \ref{IndRep}, we constructed an isomorphism $(\Ind_{P_\theta}^G \sigma)^{I_m} \cong (\Ind_{P_\theta'}^{G'} \sigma')^{I'_m}$ of $\mathbb{C}$-vector spaces that is compatible with the Hecke algebra isomorphism $\zeta_m: \mathscr{H}(G,I_m) \rightarrow \mathscr{H}(G',I_m')$. 
Consider the representation $I(\eta_\nu, \sigma)$ for $\eta_\nu \in \Unr(\mtheta)$. Since $F$ and $F'$ are $m$-close, there is a natural isomorphism between $\Unr(\mtheta) \cong \Unr(\mtheta')$, $\eta_\nu \rightarrow \eta_\nu'$. 
 Hence, if $\sigma \overset{\kappa_{l,\mtheta}}\leftrightarrow \sigma'$, then $\sigma\eta_\nu \overset{\kappa_{l,\mtheta}} \leftrightarrow \sigma'\eta_\nu'$. Consequently, we obtain an isomorphism
 \[I(\eta_\nu, \sigma)^{I_m} \xrightarrow[\kappa_{m}(\nu)]{\cong}  I(\eta_\nu', \sigma')^{I_m'}.\]
\item Let $\nu \in \Unr(\mtheta)$. The representation $I(\eta_\nu, \sigma)$ is generated by its $I_m$-fixed vectors and any non-zero $G$-stable subspace of it has non-zero $I_m$-fixed vectors ( Proposition \ref{categories}) and
\begin{align}\label{homext}
  \mathscr{H}(G) \otimes_{\mathscr{H}(G, I_m)} I(\eta_\nu, \sigma)^{I_m} &\overset{\cong}\longrightarrow I(\eta_\nu, \sigma),\nonumber\\
  \gamma \otimes f & \longrightarrow \gamma.f.
\end{align}
Moreover, since any non-zero $G$-stable subspace $X$ of $I(\eta_\nu, \sigma)$ has non-zero $I_m$-fixed vectors, this will correspond to a non-zero $G$-stable subspace $X'$ of $I(\eta_\nu', \sigma')$ with non-zero $I_m'$-fixed vectors. Hence  $I(\eta_\nu, \sigma)$ is irreducible if and only if $I(\eta_\nu', \sigma')$ is. 

\item Note that the definition of local coefficients involves  representations induced from unramified twists of $\sigma$ and ${w}_0(\sigma)$. To proceed with our study over close local fields, we need to observe that if $\sigma$ corresponds to $\sigma'$, then ${w}_0(\sigma)$ corresponds to ${w}_0(\sigma')$. 
To see this, notice that since $w_0(\Phi^+_\theta) = \Phi^+_\Omega$, we have that $M_\Omega \cap \tilde{w}_0I_m\tilde{w}_0^{-1} = M_\Omega \cap I_m$ (See Lemma \ref{conjMP}) and the map
\begin{align}\label{changelevis}
\fH(M_\theta, M_{\theta} \cap I_m) &\longrightarrow \fH(M_\Omega, M_\Omega \cap I_m)\\\nonumber
f &\longrightarrow \phi_f:m_\Omega \rightarrow f(\tilde{w}_0^{-1}m_\Omega \tilde{w}_0)
\end{align}
is an algebra isomorphism (with the Haar measures on the respective groups as above) and furthermore, the following diagram in commutative.

\begin{displaymath}
    \xymatrix{
        \fH(M_\theta, M_\theta \cap I_m) \ar[r]^{\zeta_{m,\mtheta}} \ar[d]_{} & \fH(\mtheta', \mtheta'\cap I_m') \ar[d]^{} \\
        \fH(M_\Omega, M_\Omega \cap I_m) \ar[r]_{\zeta_{m,M_\Omega}}      & \fH(M_\Omega', M_\Omega' \cap I_m')}
\end{displaymath}
Here the vertical maps are as constructed above, and the horizontal maps over close local fields is described in the beginning of the section. The required claim now follows.
\end{enumerate}
We are now ready the prove the main theorem of the section.

\begin{theorem}\label{maintheoremlocalcoeff}
Let $\chi$ be a character of $U$ as in Section \ref{WF} and assume $\chi$ is compatible with $\tw_0$. Let $(\sigma, V)$ be an irreducible, admissible $\chi_\theta$-generic representation of $M_\theta$ with Whittaker model $\calw(\sigma, \chi_\theta)$. Let $m \geq 1 $ be large enough such that $\sigma^{I_{m,\mtheta}} \neq 0$ and that there exists $v_0 \in V^{I_{m,\mtheta}}$ such that $W_{v_0}(e) \neq 0$, and additionally $\cond(\chi_\alpha) \leq m\; \forall\; \alpha \in \Delta$. Set $l = m+4$. Let $F'$ be $l$-close to $F$ and $(\sigma', V')$ be the corresponding $\chi_\theta'$-generic representation of $\mtheta'$. Then, with Haar measures chosen compatibly as in the beginning  of this subsection, we have
\[C_\chi(\nu, \sigma, \tilde{w}_0) = C_{\chi'}(\nu, \sigma', \tilde{w}_0').\]
\end{theorem}
\begin{proof}
Let $\calv $ be the open dense subset of $\cala_{\theta,\C}^*$ obtained by taking the intersection of $\calv_1$ in Section \ref{Induced} and, $\calv(\sigma, w_0)$ and $\calv(\sigma', w_0)$ in Section \ref{IO}. For $\nu \in \calv$, the following holds:
\begin{itemize}
 \item $I(\nu, \sigma) \text{ and } I(w_0(\nu), w_0(\sigma)) \text{ are irreducible}$.
\item $A(\nu, \sigma, \tilde{w}_0)$ and $A(\nu, \sigma' ,\tw_0')$ are defined and \[\dim_\C\Hom_G(I(\nu, \sigma), I(w_0(\nu), w_0(\sigma)) =1 = \dim_\C\Hom_{G'}(I(\nu, \sigma'), I(w_0(\nu), w_0(\sigma')).\]
 \end{itemize}
STEP 1: It suffices to prove
\begin{equation}\label{NLCC}
 C_\chi(\nu, \sigma, \tw_0) = C_{\chi'}(\nu, \sigma',\tw_0') \text { for } \nu \in \mathcal{V}.
\end{equation}
This is because the local coefficient $C_\chi(\nu, \sigma, \tw_0)$ is a meromorphic function of $\nu$, and in fact a rational function of $\eta_\nu$. Furthermore, note that if $C_\chi(\nu, \sigma, \tw_0)$ has a pole at $ \nu  = \nu_0 \in \calv$, then in view of Equation \eqref{localcoeff}, we see that
 $\lambda_\chi(\tw_0(\nu_0), \tw_0(\sigma)) \circ A(\nu_0, \sigma, \tw_0)(f) = 0$ for all $f \in I(\nu_0, \sigma)$. This implies that $I(\tw_0(\nu_0), \tw_0(\sigma))$ is reducible and the image of $A(\nu_0, \sigma, \tw_0)$ in $I(\tw_0(\nu_0), \tw_0(\sigma))$ is degenerate. But this contradicts our definition of $\calv$.  Therefore for each $\nu \in \calv$, we see that $C_\chi(\nu, \sigma, \tw_0)$ is holomorphic. Moreover, since we are staying away from the poles of the intertwining operator, we also have that $C_\chi(\nu, \sigma, \tw_0)$ is non-zero for each $\nu \in \calv$ (cf. Proposition 3.3.1 of \cite{Sha81}).\\\\
STEP 2: The Whittaker functional $\lambda_\chi(\nu, \sigma)$ attached to the induced representation $I(\nu, \sigma)$ is described in Section \ref{Whit}. 
Let $\nu \in \calv$. Then $I(\nu, \sigma)$ is irreducible, and hence the dimension of $\Hom_U(I(\nu, \sigma), \chi)$ is 1. Let $\phi_{\chi}(\nu, \sigma)$ be the $G$-embedding
\[ I(\nu, \sigma) \xrightarrow{\phi_{\chi}(\nu, \sigma)} \Ind_U^G \chi\] corresponding to $\lambda_\chi(\nu, \sigma)$ under Frobenius reciprocity
$
 \Hom_U(I(\nu, \sigma), \chi) \overset{\cong} \longrightarrow \Hom_G(I(\nu, \sigma), \Ind_U^G \chi).$

It is unique up to scalars and restricts to give an embedding of $\mathscr{H}(G, I_{m+1})$-modules
\[ I(\nu, \sigma)^{I_{m+1}}\xrightarrow{\phi_{\chi}(\nu, \sigma)} (\Ind_U^G \chi)^{I_{m+1}}.\]
Now, consider the following diagram:
\begin{equation*}
    \xymatrix@C=40pt @R  = 30pt{
        I(\nu, \sigma)^{I_{m+1}} \ar[r]^{\phi_{\chi}(\nu, \sigma)} \ar[d]_{\kappa_{m+1}(\nu)} & (\Ind_U^G \chi)^{I_{m+1}} \ar[d]^{\kappa_{m+1}} \\
        {I(\nu, \sigma')}^{I_{m+1}'} \ar[r]_{\phi_{{\chi'}}(\nu, \sigma')}      & (\Ind_{U'}^{G'} \chi')^{I_{m+1}'}}
\end{equation*}
Here, $\phi_{{\chi'}}(\nu, \sigma')$ corresponds to $\lambda'_{\chi'}(\nu, \sigma')$, the Whittaker functional on $I(\nu, \sigma')$ defined using the Whittaker functional $\lambda'$ on $(\sigma', V')$ as in Section \ref{WF}, $\kappa_{m+1}$ is as in Section \ref{GR}, and $\kappa_{m+1}(\nu)$ is as in Lemma \ref{Indvsiso}. Define $\phi_1 = \kappa_{m+1}^{-1} \circ \phi_{{\chi'}}(\nu, \sigma') \circ \kappa_{m+1}(\nu)$ on $I(\nu, \sigma)^{I_{m+1}}$ and extend it to $I(\nu, \sigma)$ as in Equation \eqref{homext}.
Then we get another $G$-embedding $\phi_1: I(\nu, \sigma)\hookrightarrow \Ind_U^G \chi$ (Recall that since $\nu \in \calv, I(\nu, \sigma)$ is irreducible).
We want to show that $\phi_1= \phi_{\chi}(\nu, \sigma)$. A priori, we have that $\phi_1 = a(\nu, \sigma, \chi). \phi_{\chi}(\nu, \sigma)$, where $a(\nu, \sigma, \chi)$ is a complex number. We only need to show that $a(\nu, \sigma, \chi) = 1$. We will prove that for a suitably chosen $f \in I(\nu, \sigma)^{I_{m+1}}$, we have $\phi_1(f)(1) = \phi_{{\chi} }(\nu, \sigma)(f)(1) \neq 0$, and deduce that $a(\nu, \sigma, \chi) = 1$.

For $v_0 \in \sigma^{I_{m,\mtheta}}$ as before (that is $\lambda(v_0) \neq 0$), consider the function $f$ defined as follows:
\begin{itemize}
 \item $\supp (f) = P_\theta I_m = P_\theta(\bar{N}_\theta \cap I_m)$,
 \item $f(m_\theta ni) = \sigma(m_\theta)\eta_\nu(m_\theta)\delta_{P_\theta}^{1/2}(m_\theta)\vol(\bar{N}_\theta \cap I_m,d\bar{n}_\theta )^{-1}. v_0 \text{ for } m_\theta \in M_\theta, n \in N_\theta, i \in I_m$, where $d\bar{n}_\theta$ is the Haar measure fixed in this section.
\end{itemize}
Clearly, $ f \in I(\nu, \sigma)^{I_m}\subset I(\nu, \sigma)^{I_{m+1}}$. We compute Integral \ref{Whit} on $R_{\tw_0}f$.
\begin{align}\label{Whitcomp}
\phi_{{\chi} }(\nu, \sigma)(R_{\tw_0}f)(1) &= \lambda_\chi(\nu, \sigma)(R_{\tw_0}f) \nonumber\\
& = \int_{N_\Omega} \lambda(R_{\tw_0}f(\tw_0^{-1}n))\overline{\chi(n)} dn_\Omega\nonumber \\
& = \int_{N_\Omega} \lambda(f(\tw_0^{-1}n\tilde{w}_0))\overline{\chi(n)} dn_\Omega\nonumber \\
&= \int_{\bar{N}_\theta} \lambda(f(\bar{n})) \overline{\chi(\tilde{w}_0 \bar{n}\tilde{w}_0^{-1})} d\bar{n}_\theta \nonumber\\
&=\int_{\bar{N}_\theta \cap I_m} \lambda(f(\bar{n})) \overline{\chi(\tilde{w}_0 \bar{n}\tilde{w}_0^{-1})} d\bar{n}_\theta \nonumber\\
&= \lambda(v_0) = W_{v_0}(e).
\end{align}
Note that the Haar measure $d\bar{n}_\theta$ in the beginning of this section was chosen in that manner so that the fourth equality above would hold. To get the last equality, recall that we have assumed that $\cond(\chi_\alpha) \leq m\; \forall \; \alpha \in \Delta$  and therefore $\chi|_{\tilde{w}_0 (\bar{N}_\theta\cap I_m)\tilde{w}_0^{-1}} = 1$.

We will now compute $\phi_1(f)(1)$. To do this, recall that by Section \ref{IndRep} we have an isomorphism $\kappa_m(\nu)$,
 \begin{align*}
  I(\nu, \sigma)^{I_{m}} &\xrightarrow{\kappa_m(\nu)}I(\nu, \sigma' )^{I_{m}'},\\
  h & \rightarrow h',
  \end{align*}
that is compatible with the Hecke algebra isomorphism $\zeta_m$.

Using the construction of $\kappa_m(\nu)$ in Lemma  \ref{Indvsiso} and Equation \eqref{voloppuni}, it is easy to see that $\kappa_m(\nu)(f) = f'$ where $f'$ satisfies the following conditions:
\begin{itemize}
 \item $\supp (f') = P_\theta'I_m' = P_\theta'(\bar{N}_\theta' \cap I_m')$,
 \item $f(m_\theta'n'i') = \sigma'(m_\theta')\eta_\nu'(m_\theta')\delta_{P_\theta'}^{1/2}(m_\theta')\vol(\bar{N}'_\theta \cap I_m', d\bar{n}_\theta')^{-1}. v_0' \text{ for } m \in M_\theta', n \in N_\theta', i' \in I_m'$, where $v_0' = \kappa_{m,M_\theta}(v_0)$.
\end{itemize}
Moreover, $\kappa_{m+1}(\nu)(f) = \kappa_m(\nu)(f)\; \forall\; f \in I(\nu, \sigma)^{I_m}$ and $\kappa_{m+1}(\nu)(R_{\tw_0} f) = R_{\tw_0'}f'$ by Observation (b).

Now, $\phi_1(f)(1) = \kappa_{m+1}^{-1} \circ \phi_{{\chi'}}(\nu, \sigma') \circ \kappa_m(\nu)(f)(1) = \kappa_{m+1}^{-1} \circ \phi_{{\chi'}}(\nu, \sigma')(f')(1)$. By construction of the map $\kappa_m$, the following diagram is commutative:

\begin{displaymath}
   \xymatrix@C=40pt @R  = 30pt{
       (\Ind_U^G \chi)^{I_{m+1}}  \ar[d]^{\kappa_{m+1}} \ar[r]^-{g \rightarrow g(1)}  & \C \ar[d]^{Id} \\
       (\Ind_{U'}^{G'} \chi')^{I_{m+1}'}  \ar[r]_-{g' \rightarrow g'(1)}      & \C }
\end{displaymath}
Using the above and proceeding as in \ref{Whitcomp},  we obtain \[\kappa_{m+1}^{-1} \circ \phi_{{\chi'}}(\nu, \sigma')(f')(1) = \phi_{{\chi'}}(\nu, \sigma')(f')(1) = \lambda'(v_0') = W_{v_0'}(e) = W_{v_0}(e)\neq 0,\] since the Whittaker models of $\sigma$ and $\sigma'$ are chosen in a compatible manner as explained in the beginning of this subsection. This shows that $\phi_1 = \phi_{\chi}(\nu, \sigma)$. In particular, we have proved that

\begin{equation}\lambda_\chi(\nu, \sigma)(f) = \lambda_\chi'(\nu, \sigma') \circ \kappa_{m+1}(\nu)(f) \; \forall \; f \in I(\nu, \sigma)^{I_{m+1}}, \nu \in \calv.
\end{equation}\\
STEP 3:  We will now study the intertwining operators over close local fields. Let $\nu \in \mathcal{V}$. Consider the following diagram:
 \begin{equation}
    \xymatrix@C=40pt @R  = 30pt{
        I(\nu, \sigma)^{I_{m+1}} \ar[r]^-{A(\nu, \sigma,\tw_0)} \ar[d]_{\kappa_{m+1}(\nu)} & I(w_0(\nu), w_0(\sigma))^{I_{m+1}} \ar[d]^{\kappa_{m+1}(w_0(\nu))} \\
        {I(\nu, \sigma')}^{I_{m+1}'} \ar[r]_-{A(\nu, \sigma',\tw_0')}      & I(w_0(\nu), w_0(\sigma'))^{I_{m+1}'}}
\end{equation}

Define $A_1 = \kappa_{m+1}({w}_0(\nu))^{-1} \circ A(\nu, \sigma', \tw_0') \circ \kappa_{m+1}(\nu)$ on $I(\nu, \sigma)^{I_{m+1}}$ and extend it to $I(\nu, \sigma)$ using Equation \eqref{homext}. Then \[A_1 \in \Hom_G(I(\nu, \sigma), I(w_0(\nu), w_0(\sigma))).\] Since $\nu \in \calv$, these induced representations are irreducible, the intertwining operator is defined, and \[\dim_\C \Hom_G(I(\nu, \sigma), I(w_0(\nu), w_0(\sigma))) =1.\]
Hence, there exists a scalar $b(\nu, \sigma, \tw_0) \in \C$ such that $A_1 = b(\nu, \sigma,\tw_0)\cdot A(\nu, \sigma, \tw_0)$. We want to show that $b(\nu, \sigma, \tw_0) = 1$. Let $f$ be the element constructed in STEP 2. Since $\supp(f)  = P_\theta (\bar{N}_\theta \cap I_m)$, we see that $\supp(R_{\tw_0}f) \subset P_\theta \tw_0^{-1}N_\Omega$, and hence $R_{\tw_0}(f) \in I(\nu, \sigma)^{\leq w_0^{-1}}$ (with notation explained in Section \ref{IO}). We will show that $A_1(R_{\tw_0}f)(1) = A(\nu, \sigma, \tw_0)(R_{\tw_0}f)(1) \neq 0$ and deduce that  $b(\nu, \sigma, \tw_0) = 1$. By Equation \eqref{muicio}, we know that
\begin{align*}
 A(\nu, \sigma, \tw_0)(R_{\tw_0}f)(1) &= A(\nu, \sigma, \tw_0)f(\tw_0)= \int_{N_\Omega} f(\tilde{w}_0^{-1}n\tilde{w}_0) dn_\Omega= \int_{\bar{N}_\theta} f(\bar{n}) d\bar{n}_\theta= v_0.
\end{align*}
We will now compute $A_1(f)(\tilde{w}_0)$. First note that for $g \in I(w_0(\nu), w_0(\sigma))^{I_{m+1}}$, 
\[g(\tw_0) \in (w_0(\sigma\eta_\nu))^{M_\Omega \cap \tw_0I_{m+1}\tw_0^{-1}} = (w_0(\sigma\eta_\nu))^{M_\Omega \cap I_{m+1}}.\] By Observation (e), we know that $w_0(\sigma\eta_\nu)\leftrightarrow w_0(\sigma'\eta_\nu')$. Using this and the construction of the isomorphism between the induced representations (Lemma  \ref{Indvsiso}), we see that the following diagram is commutative.
\begin{equation}
    \xymatrix@C=40pt @R  = 30pt{
       I(w_0(\nu), w_0(\sigma))^{I_{m+1}}  \ar[r]^-{g \rightarrow g(\tilde{w}_0)}  & (w_0(\sigma\eta_\nu))^{M_\Omega \cap I_{m+1}} \ar[d]^{\kappa_{m+1,M_\Omega}} \\
       I(w_0(\nu), w_0(\sigma'))^{I_{m+1}'}  \ar[u]_{\kappa_{m+1}(w_0(\nu))^{-1}} \ar[r]_-{g' \rightarrow g'(\tilde{w}_0')}      & w_0(\sigma'\eta_\nu')^{M_\Omega' \cap I_{m+1}'}}
\end{equation}
Now,
\begin{align*}
A_1(f)(\tw_0) &= \kappa_{m+1}({w}_0(\nu))^{-1} \circ A(\nu, \sigma', \tw_0') \circ \kappa_{m+1}(\nu)(f)(\tw_0)\\
& =(\kappa_{m+1}({w}_0(\nu))^{-1} \circ A(\nu, \sigma', \tw_0')(f'))(\tw_0)\\
& = \kappa_{m+1,M_{\Omega}}^{-1} \circ (A(\nu, \sigma', \tw_0')(f')(\tw_0'))\\
 &= \kappa_{m+1,M_{\Omega}}^{-1}(v_0') = v_0
\end{align*}
In particular, we have proved that
\begin{equation}
A(\nu, \sigma, \tw_0)(f) = \kappa_{m+1}({w}_0(\nu))^{-1} \circ A(\nu, \sigma', \tw_0') \circ \kappa_{m+1}(\nu)(f)\; \forall \; f \in I(\nu, \sigma)^{I_{m+1}}, \nu \in \mathcal{V}.
\end{equation}\\
STEP 4: Let $\nu \in \mathcal{V}$. Recall that the local coefficient is defined using the equation
\[ \lambda_\chi(\nu, \sigma) = C_\chi(\nu, \sigma,\tw_0) \lambda_\chi(w_0(\nu), w_0(\sigma)) \circ A(\nu, \sigma, \tw_0).\]
By STEP 2 and STEP 3, we have the following on  $ I(\eta_\nu, \sigma)^{I_{m+1}}$:
\begin{enumerate}
\item $ \phi_{\chi}(\nu, \sigma)(f) = \kappa_{m+1}^{-1} \circ \phi_{{\chi'}}(\nu, \sigma') \circ \kappa_{m+1}(\nu)(f)$.
\item $ A(\nu, \sigma, \tw_0)(f) = \kappa_{m+1}(w_0(\nu))^{-1} \circ A(\nu, \sigma', \tw_0') \circ \kappa_{m+1}(\nu)(f)$.
\end{enumerate}
Using arguments similar to STEP 2, we can also prove that for each $g \in I(w_0(\nu), w_0(\sigma))^{I_{m+1}}$,

\hspace*{4pt}(c) $\phi_{\chi}(w_0(\nu), w_0(\sigma))(g)= \kappa_{m+1}^{-1} \circ \phi_{{\chi'}}(w_0(\nu), w_0(\sigma')) \circ \kappa_{m+1}(w_0(\nu))(g).$\\
Finally, for $f \in I(\nu, \sigma)^{I_{m+1}}$, we have
\begin{align}\label{whitclose}
 \lambda_\chi(\nu, \sigma)(f) =\phi_{\chi}(\nu, \sigma)(f)(1)= \kappa_{m+1}^{-1} \circ \phi_{{\chi'}}(\nu, \sigma')(f')(1)= \phi_{{\chi'}}(\nu, \sigma')(f')(1) = \lambda_{\chi'}(\nu, \sigma')(f')
 \end{align}
and
\begin{align}\label{whit1close}
 \lambda_\chi(w_0(\nu), w_0(\sigma)) \circ A(\nu, \sigma, \tw_0)(f) =& (\kappa_{m+1}^{-1} \circ \phi_{{\chi'}}(w_0(\nu), w_0(\sigma')) \circ \kappa_{m+1}(w_0(\nu))\nonumber\\
 &\circ \kappa_{m+1}(w_0(\nu))^{-1} \circ A(\nu, \sigma', \tw_0') \circ \kappa_{m+1}(\nu)(f))(1) \nonumber\\
=& (\kappa_{m+1}^{-1} \circ \phi_{{\chi'}}(w_0(\nu), w_0(\sigma'))\circ A(\nu, \sigma', \tw_0')(f'))(1) \nonumber\\
=& \phi_{{\chi'}}(w_0'(\nu), w_0'(\sigma'))\circ A(\nu, \sigma', \tw_0')(f')(1)\nonumber\\
=& \lambda'_{\chi'}(w_0(\nu), w_0(\sigma')) \circ  A(\nu, \sigma', \tw_0')(f')
\end{align}
Choosing $f \in I(\nu, \sigma)^{I_{m+1}}$ such that $\lambda_\chi(\nu, \sigma)(f) \neq 0$, and  combining Equations \eqref{whitclose} and \eqref{whit1close}, we obtain that
\[ C_\chi(\nu, \sigma, \tw_0) = C_{\chi'}(\nu, \sigma', \tw_0') \; \forall \nu \in \mathcal{V}.\]
Now use STEP 1 to complete the proof of the theorem.\qedhere

\end{proof}

\section{Plancherel measures over close local fields}\label{Plancherel}

We retain the notation of the previous section. For an irreducible, admissible representation $(\sigma, V)$ of $\mtheta$ (maximal) and $\nu \in \cala_{\theta,\C}^*$, we consider the induced representations
\[I(\nu, \sigma) \;\; \text{ and }\;\; I(w_0(\nu), w_0(\sigma))\]
where $w_0 = w_{l, \Delta}w_{l, \theta}$ with $w_0(\theta) = \Omega$, and $\tw_0$ denotes  its lifting as in Section \ref{reps} (cf. Section \ref{Induced}). Let $A(\nu, \sigma, \tw_0)$ denote the standard intertwining operator in Section \ref{IO}. Let us recall the definition of the Plancherel measure from \cite{Sha90}. Define
\[\gamma_{w_0}(G/P_\theta) = \int_{\bar{N}_\theta} q^{\langle 2\rho_{P_\theta}, H_\mtheta(\bar{n}) \rangle}d\bar{n}_\theta,\]
where $d\bar{n}_\theta$ is the Haar measure fixed in the previous section, $\rho_{P_\theta}$ is the half sum of the roots in $N_\theta$, $H_{\mtheta}$ is defined on $\mtheta$ as in Section \ref{unrchar} and extended to a function on $G$ by letting $H_\mtheta(mnk) = H_\mtheta(m), m \in \mtheta, n \in N_\theta$ and $k \in \bfg(\calo)$ (cf. Section 2 of \cite{Sha90}). 
Similarly, we define $\gamma_{w_0^{-1}}(G/P_\Omega)$. There exists a constant $\mu(\nu, \sigma, w_0)$  such that
\[A(w_0(\nu), w_0(\sigma), \tw_0^{-1})  \circ A(\nu, \sigma, \tw_0) = \mu(\nu, \sigma, w_0)^{-1}\gamma_{w_0}(G/P_\theta) \gamma_{w_0^{-1}}(G/P_\Omega).\]
This constant depends only on $\nu$,  the class of $\sigma$ and on $w_0$, but not on the choice $\tw_0$  (cf. Page 280 of \cite{Sha90}). The scalar $\mu(\nu, \sigma, w_0)$ is a meromorphic function of $\nu$ and is called the Plancherel measure associated to $\nu, \sigma$ and $w_0$. 

Our aim in this section is to study  Plancherel measures over close local fields and prove that they are compatible with the Deligne-Kazhdan correspondence.
\begin{theorem}\label{PlancherelCLF} Let $(\sigma, V)$ be an irreducible, admissible representation of $M_\theta$ and let $m \geq 1$ such that $\sigma^{I_{m, \mtheta}} \neq 0$. Set $l =m+4$ and let $F'$ be another non-archimedean local field that is $l$-close to $F$. Let $(\sigma',V')$ be the irreducible, admissible representation of $\mtheta'$ such that  $\kappa_{l,\mtheta}: \sigma^{I_{l, \mtheta}}\rightarrow (\sigma')^{I_{l, \mtheta'}'}$ is an isomorphism that is compatible with the Hecke algebra isomorphism $\zeta_{l,\mtheta}:\fH(\mtheta, I_{l,\mtheta}) \rightarrow \fH(\mtheta', I_{l, \mtheta'}')$. Then, with Haar measures chosen compatibly as in Section \ref{maintheorem}, we have
\[
\mu(\nu, \sigma, w_0) = \mu(\nu, \sigma', w_0).
\]
 \end{theorem}
\begin{proof}
The proof of this theorem is similar to the proof of Theorem \ref{maintheoremlocalcoeff}.  First, using Remark (3), Page 240 of \cite{Wal03} and Observation (a) of Section \ref{maintheoremlocalcoeff}, it is clear that $\gamma_{w_0}(G/P) = \gamma_{w_0}(G'/P')$ and $\gamma_{w_0^{-1}}(G/P_\Omega) = \gamma_{w_0^{-1}}(G'/P_\Omega')$. 

Let $\calv$ be an open dense subset of $\cala_{\theta, \C}^*$ such that 
\begin{itemize}
\item $I(\nu, \sigma)$ and $I(w_0(\nu), w_0(\sigma))$ are both irreducible. 
\item $A(\nu, \sigma, \tw_0)$, $A(\nu, \sigma', \tw_0')$, $A(w_0(\nu), w_0(\sigma), \tw_0^{-1})$, and $A(w_0(\nu), w_0(\sigma'), \tw_0'^{-1})$ are all defined, that is Equation \eqref{dimensionMuic1} is satisfied for the respective spaces.
\end{itemize}
Again, such a $\calv$ exists by Sections \ref{Induced} and \ref{IO}.

Moreover, for $\nu \in \calv$, $\mu(\nu, \sigma, w_0)$ is holomorphic and non-zero. Arguing as in Theorem \ref{maintheoremlocalcoeff}, we can prove that for each $\nu \in \calv$, the following diagrams are commutative:
 \begin{equation}
    \xymatrix@C=50pt @R  = 40pt{
        I(\nu, \sigma)^{I_{m+1}} \ar[r]^-{A(\nu, \sigma,\tw_0)} \ar[d]_{\kappa_{m+1}(\nu)} & I(w_0(\nu), w_0(\sigma))^{I_{m+1}} \ar[d]^{\kappa_{m+1}(w_0(\nu))} \\
        {I(\nu, \sigma')}^{I_{m+1}'} \ar[r]_-{A(\nu, \sigma',\tw_0')}      & I(w_0(\nu), w_0(\sigma'))^{I_{m+1}'}}
\end{equation}
and
 \begin{equation}
    \xymatrix@C=75pt @R  =40pt{
        I(w_0(\nu), w_0(\sigma))^{I_{m+1}} \ar[r]^-{A(w_0(\nu), w_0(\sigma),\tw_0^{-1})} \ar[d]_{\kappa_{m+1}(w_0(\nu))} &  I(\nu, \sigma)^{I_{m+1}} \ar[d]^{\kappa_{m+1}(\nu)} \\
         I(w_0(\nu), w_0(\sigma'))^{I_{m+1}'} \ar[r]_-{A(w_0(\nu), w_0(\sigma'),\tw_0'^{-1})}      & {I(\nu, \sigma')}^{I_{m+1}'}}
\end{equation}

Consequently, for each $\nu \in \calv$ and for each $f \in I(\nu, \sigma)^{I_{m+1}}$, we have
\begin{align*}
\mu(\nu, \sigma, w_0)^{-1}\gamma_{w_0}(G/P_\theta)& \gamma_{w_0^{-1}}(G/P_\Omega) f= A(w_0(\nu), w_0(\sigma),\tw_0^{-1}) \circ A(\nu, \sigma,\tw_0)(f)\\
&  =\kappa_{m+1}(\nu)^{-1} \circ A(w_0(\nu), w_0(\sigma'),\tw_0'^{-1})  \circ \kappa_{m+1}(w_0(\nu)) \\
&\circ  \kappa_{m+1}(w_0(\nu))^{-1} \circ A(\nu, \sigma',\tw_0') \circ \kappa_{m+1}(\nu)(f)\\
& =  \kappa_{m+1}(\nu)^{-1} \circ A(w_0(\nu), w_0(\sigma'),\tw_0'^{-1}) \circ A(\nu, \sigma',\tw_0') \circ \kappa_{m+1}(\nu)(f)\\
& =  \mu(\nu, \sigma', w_0)^{-1}\gamma_{w_0}(G'/P'_\theta) \gamma_{w_0^{-1}}(G'/P'_\Omega) \kappa_{m+1}(\nu)^{-1} \circ \kappa_{m+1}(\nu)(f)\\
& = \mu(\nu, \sigma', w_0)^{-1}\gamma_{w_0}(G'/P'_\theta) \gamma_{w_0^{-1}}(G'/P'_\Omega) f
\end{align*}

Hence
\[ \mu(\nu, \sigma, w_0) = \mu(\nu, \sigma', w_0)\; \forall\; \nu \in \calv.\]
Note that $\mu(\nu, \sigma, w_0)$ is a meromorphic function of $\nu$, and $\calv$ is open dense in $\cala_{\theta, \C}^*$. This completes the proof of the theorem. 
\end{proof}

\section{The local Langlands correspondence for $\GL_n$ over close local fields}\label{LLCGLNCLF}
The local Langlands correspondence (LLC) for $\GL_n$ establishes a bijection between the set of isomorphism classes of $n$-dimensional semisimple  representations of the Weil-Deligne group $\WD_F$ and the set of isomorphism classes of irreducible, admissible representations of $\GL_n(F)$, uniquely characterized by certain properties (see Theorem \ref{LLC} below). Deligne's theory provides us a tool to analyze properties of representations of the Weil group over close local fields. Similarly, Kazhdan's theory enables us to analyze properties of representations of $\GL_n$ over close local fields. In this section, we will show that the LLC for $\GL_n$ in positive characteristic is related to the LLC for $\GL_n$ in characteristic 0 via the Deligne-Kazhdan correspondence. We will first recall some important results relating to the LLC for $\GL_n$ before proving the main theorem.  Let \[\mathcal{A}_F(n):= \{ \text{ Iso.  classes of irr. ad. representations of } \GL_n(F)\}\] and \[\mathcal{G}_F(n):= \{ \text{  Iso. classes of semisimple representations of $\WD_F$ of dimension $n$}\}.\] 
 Similarly let \[\mathcal{A}_F^0(n):= \{ \text{ Iso.  classes of irr. ad. supercuspidal representations of } \GL_n(F)\}\] and \[\mathcal{G}_F^0(n):= \{ \text{  Iso. classes of irr. smooth representations of $W_F$ of dimension $n$}\}.\] 

\begin{theorem}[LLC]\label{LLC} There exists a unique family of bijective maps $\rec_n: \mathcal{A}_F(n) \rightarrow \mathcal{G}_F(n)$ such that
\begin{enumerate}[(a)]
\item For $n = 1$, it is given by local class field theory.
\item For $\sigma \in \mathcal{A}_F(n)$ and $\tau \in \mathcal{A}_F(t)$, we have 
\[ L(s, \sigma \times \tau) = L(s, \rec_n(\sigma) \otimes \rec_t(\tau)), \]
\[ \epsilon(s, \sigma \times \tau, \psi) = \epsilon(s,\rec_n(\sigma) \otimes \rec_t(\tau), \psi)\]
for all non-trivial characters $\psi$ of $F$.
\item For $ \sigma \in \mathcal{A}_F(n)$ and $\chi \in \mathcal{A}_F(1)$,
\[ \rec_n(\sigma\chi) = \rec_n(\sigma) \otimes \rec_1(\chi).\]
\item For $\sigma \in \mathcal{A}_F(n)$ with central character $\omega_\sigma$,
\[ \rec_1(\omega_\sigma) = \det \circ \rec_n(\sigma).\]
\item For $\sigma \in \mathcal{A}_F(n)$, 
\[\rec_n(\sigma^\vee) = \rec_n(\sigma)^\vee.\]\qedhere
\end{enumerate}
\end{theorem}
The factors in the LHS of Property $(b)$ are the Rankin-Selberg $\gamma$-factors and those on the right are the Artin factors. The proof of this theorem is due to Laumon, Rapoport, and Stuhler \cite{LRS93} for local fields of positive characteristic. The proof for local fields of characteristic 0 is due to Harris-Taylor \cite{HT01} and Henniart \cite{Hen00} (and recently, Scholze \cite{Sch13}). 

 Let $\sigma$ be an irreducible, admissible representation of $\GL_n(F)$. Let $\cond(\sigma)$ and $\cond(\phi_\sigma)$ denote the conductor of $\sigma$ and $\phi_\sigma$ respectively. The conductor of $\sigma$ is also given by the smallest integer $m$ such that $\sigma^{K_n(m)} \neq 0$, where \[K_n(m) = \left\{\left(\begin{array}{cc}
A & b \\
c & d \\
\end{array}\right)\in \GL_n(\calo) :A \in \GL_{n-1}(\calo),\, c\equiv 0\, mod \, \calp^m,\, d \equiv 1\, mod \, \calp^m\right\}.\]
(See Section 5 of \cite{JPSS81}). Let  $\phi_\sigma = \rec_n(\sigma)$ be the Langlands parameter attached to $\sigma$ as in Theorem \ref{LLC}. Since the LLC preserves $\epsilon$-factors, we see that $\cond(\sigma) = \cond(\phi_\sigma)$. 

In \cite{MP94, MP96}, Moy-Prasad defined filtrations of parahoric subgroups associated to points in the Bruhat-Tits building and used it to define the notion of depth of a representation of a reductive algebraic group. Let us briefly recall the notion of depth of a representation from \cite{MP94}. Let $\bfg$ be a split connected reductive group over $\Z$ with maximal torus $\bft$ and Borel $\bfb = \bft \bfu$, and $\Delta$ is the set of simple roots of $\bft$ in $\bfb$ (cf. Section \ref{stdnotations}). Let
\[\calA(T) =( X_*(\bft) \otimes_\Z \R)/(X_*(\bfz) \otimes _\Z \R)\]
be the reduced apartment associated to $T$. Let $\calb$  denote the reduced Bruhat-Tits building of $G$.  Recall that it is canonically defined, and  $G$ acts on $\calb$.  Setting $\calb^1= \calb \times (X_*(\bfz) \otimes \R)$ (the extended Bruhat-Tits building), and observing that 
\[G/G^1 \otimes \R= X_*(\bfz) \otimes \R\]
where $G^1= \{g \in G| ord_F(\chi(g)) = 0 \; \forall \chi \in \Hom_F(G, \bG_m)\}, $ 
we see that $G$ acts on $\calb^1$ as well.

Let $G_{x,0}$ be the parahoric subgroup attached to $x \in \calA(T)$.  Similarly, let $G^{\text{der}}_{x,0}$ denote the parahoric subgroup of $G^{\text{der}}$ attached to the point $x$.  For $r \geq 0$, let $G_{x,r}$ be the Moy-Prasad filtration subgroup as in \cite{MP94, MP96} and let $G_{x,r+} = \cup_{s>r}G_{x,s}$. 
More explicitly, for each $ r \geq 0$ and each $x \in \calA(T)$, 
\begin{equation}\label{MPpresentation}G_{x,r} = \langle T_{\calp^{\lceil r\rceil }}, \bfu_\alpha(\calp^{-\lfloor \alpha(x)-r\rfloor}):\;  \alpha \in \Phi \rangle; \hspace*{0.2in} G_{x,r+} = \langle T_{\calp^{1+\lfloor r\rfloor }}, \bfu_\alpha(\calp^{1- \lceil \alpha(x)-r\rceil}):\;  \alpha \in \Phi \rangle,
\end{equation}
and
\begin{equation}\label{MPpresentation1}
G^{\text{der}}_{x,r} = \langle T^{\text{der}}_{\calp^{\lceil r\rceil }}, \bfu_\alpha(\calp^{-\lfloor \alpha(x)-r\rfloor}):\;  \alpha \in \Phi \rangle ; \hspace*{0.2in} G^{\text{der}}_{x,r+} = \langle T^{\text{der}}_{\calp^{1+\lfloor r\rfloor }}, \bfu_\alpha(\calp^{1- \lceil \alpha(x)-r\rceil}):\;  \alpha \in \Phi \rangle
\end{equation}
where $\bft^{\text{der}} = \bfg^{\text{der}} \cap \bft$. For each $r >0$ and each $x \in \calA(T)$, the Moy-Prasad filtration subgroup $G_{x,r}$ admits an Iwahori factorization  with respect to all standard parabolic subgroups $\bfp_\theta$ (Theorem 4.2 of \cite{MP96}). 

For an irreducible, admissible representation $(\sigma, V)$ of $G$, define
\[\depth(\sigma) := \inf\{r | \text{there exists } x \in \calb \text{ such that } V^{G_{x,r+}} \neq 0  \}.\]
This is a non-negative rational number and the infimum is attained for some $x \in \calb$.  We prove the following simple lemma that will be useful later. 
 
\begin{lemma}\label{MPUC} Let $\bfg$ be a split connected reductive group over $\Z$. Assume $\bfg^{\text{der}}$ is simple. Put $G = \bfg(F)$. Let $(\sigma, V)$ be an irreducible, admissible representation of $G$ of depth $r$. Let $m = \lceil r \rceil +1$. Then $V^{I_m} \neq 0$, where $I_m$ is the $m$-th Iwahori filtration subgroup of $G$. 
\end{lemma}
\begin{proof}
Since $\depth(\sigma) =r$, there exists $x \in \calb$ such that $V^{G_{x, r+}} \neq 0$, and $r$ is the smallest non-negative real number with this property.   Let $\calc$ be the maximal facet in $\calA(T)$ defined by 
\[\calc:= \{x \in \calA(T)\; |\; \alpha(x) >0 \;\forall\; \alpha \in \Delta; 1- \alpha_0(x)>0\}\]
where $\alpha_0$ is the highest root of $\Phi^+$. Then for each $x \in \calb$, there is a unique $y \in \bar{\calc}$ (here $\bar{\calc}$ is $\calc$ together with its walls) and $g \in G$ such that $y = g.x$. Moreover, 
$G_{y,r} = g G_{x,r}g^{-1}.$
Hence $\depth(\sigma) = r$ implies that there is $y \in \bar{\calc}$ such that $V^{G_{y,r+}} \neq 0$ and $r$ is the smallest non-negative number with this property.
Furthermore, note that the parahoric subgroup attached to the point $0 \in \bar{\calc}$ is $\bfg(\calo)$ and the stabilizer of $\calc$ is the standard Iwahori subgroup $I$ of $G$. 
 Now, since $y \in \bar{\calc}$, we see that
 \[ \alpha(y)\geq 0\; \forall \;\alpha \in \Delta \text{ and } \alpha_0(y)\leq 1.\]
 Since $\alpha_0(y)\leq 1$ and $\alpha_0$ is the highest root, we see that $\alpha(y)\leq 1$ for all $\alpha \in \Phi^+$ (cf. Proposition 25 of \cite{Bou02}).
 Hence we obtain that
 \[0 \leq |\alpha(y)| \leq 1 \;\forall\;  \alpha \in \Phi.\]
  With $m$ as in the statement of the lemma, a simple calculation shows that
\begin{align*}
& m-1 \leq 1 - \lceil \alpha(y) -r \rceil \leq m\; \forall \; \alpha \in \Phi^+,\\
&m \leq 1 - \lceil \alpha(y) -r \rceil \leq m+1 \; \forall \; \alpha \in \Phi^-,
\end{align*}
 which implies that $\bfu_\alpha(\calp^{m}) \subset \bfu_\alpha(\calp^{ 1 - \lceil \alpha(y) -r \rceil})$ for all $\alpha \in \Phi^+$  and  $\bfu_\alpha(\calp^{m+1}) \subset \bfu_\alpha(\calp^{ 1 - \lceil \alpha(y) -r \rceil})$ for all $\alpha \in \Phi^-$ .  Now, since \[I_{m}  = \langle T_{\calp^{m}}, \bfu_\alpha(\calp^{m}), \bfu_{-\alpha}(\calp^{m+1});\; \alpha \in \Phi^+ \rangle,\] we deduce that $I_{m} \subset G_{y,r+}$. Hence $V^{I_{m}} \neq 0$. Note that this immediately implies that $V^{K_{m+1}} \neq 0$.
  \end{proof}

The depth of a Langlands parameter $\phi: \WD_F \rightarrow \LG$ is defined as follows:
\[\depth(\phi) : = \inf\{r| \phi|_{I_F^{r+}} = 1\},\]
where $I_F \subset W_F \subset \WD_F$ denotes the inertia group and the filtration is the upper numbering filtration of ramification subgroups (See Chapter IV of \cite{Ser79}). We have the following theorem regarding depth preservation of $\GL_n(F)$.

 \begin{theorem}[Depth preservation]\label{depth} Let $\sigma$ be an irreducible, admissible representation of $\GL_n(F)$ and let $\phi_\sigma$ be its Langlands parameter as in Theorem \ref{LLC}. Then 
 \[ \depth(\sigma) = \depth(\phi_\sigma).\]
 \end{theorem}
 This is Theorem 2.3.6.4 of \cite{Yu09}. 
 As noted there, if $\sigma$ is supercuspidal, then
 \begin{equation}\depth(\sigma) = \displaystyle{\frac{\cond(\sigma) - n}{n}} = \displaystyle{\frac{\cond(\phi_\sigma)-n}{n}} = \depth(\phi_\sigma).
 \end{equation}
 More generally, if $\sigma$ is essentially square integrable, then by the main theorem of \cite{LR03},
 \begin{equation}\label{JosLanDepth}
 \depth(\sigma) = \displaystyle{\max\left\{0, \frac{\cond(\sigma) - n}{n}\right\}}.
 \end{equation}

\begin{theorem}[Stability of epsilon factors]\label{stability}
  Let $\sigma_i, \, i = 1,2$ be two irreducible, admissible supercuspidal representations of $\GL_n(F),\, n \geq 2$, such that $\omega_{\sigma_1}=\omega_{\sigma_2}$ and let $a:= \max\{\depth(\sigma_1), \depth(\sigma_2)\}$.  If $\tau$ is an irreducible, admissible supercuspidal representation of $\GL_t(F),\, 1 \leq t \leq n-1$, such that $\depth(\tau) > 2a$, then
\[ \epsilon(s, \sigma_1 \times \tau, \psi) = \epsilon(s, \sigma_2 \times \tau, \psi). \]
\end{theorem}
\begin{proof}
The proof of this theorem follows by combining Theorem 4.6 of \cite{DH81} and Theorem \ref{LLC}. The author thanks Professor Jiu-Kang Yu for pointing to the reference of \cite{DH81} and explaining this proof. Put $\phi_{\sigma_i} = \rec_n(\sigma_i)$ and $\phi_\tau = \rec_t(\tau)$. 
Using Theorem \ref{LLC}, it suffices to prove that 
\[\epsilon(s, \phi_{\sigma_1} \otimes \phi_\tau, \psi) = \epsilon(s, \phi_{\sigma_2} \otimes \phi_\tau, \psi).\]
Viewing these representations as virtual representations of the Grothendieck group, this is equivalent to proving that 
\[\epsilon(s, ([\phi_{\sigma_1}] - [\phi_{\sigma_2}])  \otimes\phi_\tau, \psi) = 1.\]
Now, $W: = [\phi_{\sigma_1}] - [\phi_{\sigma_2}]$ is a virtual representation of dimension 0. Note that the hypothesis of Theorem 4.6 of \cite{DH81} are satisfied. In fact, with notation as in that theorem, we have \[\beta(W) = \max\{\depth(\phi_{\sigma_1}), \depth(\phi_{\sigma_2})\} = a.\] Since $\depth(\tau)> 2a$, we have \[\depth(\phi_\tau)> 2 \beta(W) \geq 0.\] Therefore $\phi_\tau$ is not tamely ramified. Hence, by Theorem 4.6 of \cite{DH81} there is an element $\gamma \in F^\times$ depending only on $\phi_\tau$ and $\psi$ such that 
\begin{align*}
\epsilon(s, ([\phi_{\sigma_1}] - [\phi_{\sigma_2}])  \otimes \phi_\tau, \psi) &= \det([\phi_{\sigma_1}] - [\phi_{\sigma_2}])(\gamma)\\
&=\displaystyle{\frac{ \det[\phi_{\sigma_1}](\gamma)}{\det[\phi_{\sigma_2}](\gamma)}}\\
& =1 \text { (by Property $(d)$ of Theorem \ref{LLC})}.\qedhere
\end{align*}
\end{proof}
\begin{theorem}[Converse theorem]\label{charac}
Suppose $\sigma_i, \, i = 1,2$  are two irreducible, admissible supercuspidal representations of $\GL_n(F)$. Assume $\omega_{\sigma_1} = \omega_{\sigma_2}$ and let $a:=\max\{\depth(\sigma_1), \depth(\sigma_2)\}$. Suppose $n \geq  3$, assume for each integer $t$ with $1 \leq t \leq n-2$, and each irreducible, admissible supercuspidal representation $\tau$ of $\GL_t(F)$ with $\depth(\tau) \leq 2a$, 
\[ \epsilon(s, \sigma_1 \times \tau, \psi) = \epsilon(s, \sigma_2 \times \tau, \psi),\]
then $ \sigma_1 \cong \sigma_2$.
For $n=2$, assume that for each character $\chi$ of $F^\times$ with $\depth(\chi)\leq 2a$, 
\[ \epsilon(s, \sigma_1 \times \chi, \psi) = \epsilon(s, \sigma_2 \times \chi, \psi),\]
then $\sigma_1 \cong \sigma_2$. 
\end{theorem}
\begin{proof} Suppose $n \geq 3$. By Theorem \ref{stability}, the above condition is equivalent to the condition where the twisted epsilon factors agree for all supercuspidal representations $\tau$ of $\GL_t(F), \, \ 1\leq t \leq n-2$. Since $L(s, \sigma_i \times \tau),  L(1-s, \sigma_i^\vee \times \tau^\vee) = 1$, we have $\epsilon(s, \sigma_i \times \tau, \psi) = \gamma(s, \sigma_i \times \tau, \psi), \; i=1,2$. Consequently, we have
\[ \gamma(s, \sigma_1 \times \tau, \psi) = \gamma(s, \sigma_2 \times \tau, \psi),\]
for all supercuspidal representations of $\GL_t(F), \; 1 \leq t \leq n-2$. Since the $\gamma$-factors are multiplicative with respect to parabolic induction, we get that the above equality holds for all generic representations $\tau$ of $\GL_t(F), \; 1 \leq t \leq n-2$. Now apply Theorem 1.1 of \cite{Che06}. For $n=2$, combine Theorem \ref{stability} with Theorem 1.1 of \cite{Hen93}.
\end{proof}

\begin{theorem}[Rankin-Selberg gamma factors over close local fields]\label{RSGF}  Let $n \geq 2$ and $ 1 \leq t \leq n$. Fix $m \geq 1$.  Let $\sigma$ and $\tau$ be two irreducible, admissible generic representations of $\GL_n(F)$ and $\GL_t(F)$ respectively.  Assume $\depth(\sigma), \depth(\tau) \leq m$. Let $l = n^2m+n^2+4$ and let $F$ and $F'$ be $l$-close. Let $\sigma'$ and $\tau'$ be representations of $\GL_{n}(F')$ and $\GL_{t}(F')$ obtained using the Hecke algebra isomorphisms $\mathscr{H}(\GL_n(F), I_{l,n})  \xrightarrow{\zeta_l} \mathscr{H}(\GL_{n}(F'), I_{l,n}')$ and $\mathscr{H}(\GL_t(F), I_{l,t})  \xrightarrow{\zeta_l} \mathscr{H}(\GL_{t}(F'), I_{l,t}')$, respectively. Then 
\[ \gamma(s, \sigma \times \tau, \psi) = \gamma(s, \sigma' \times \tau', \psi')\]
where the factors above are the Rankin-Selberg $\gamma$-factors. 
\end{theorem}
\begin{proof}
It is enough to prove this theorem assuming that $\cond(\psi) = 0$. By \cite{Sha84} (Also cf. \cite{HL13} for an alternate proof), we know that the $\gamma$-factor arising out of the Langlands-Shahidi method (which is simply the associated local coefficient) agrees with the Rankin-Selberg $\gamma$-factor. 
Hence we only need to check the $l$ chosen in the statement above is large enough for the hypothesis of Theorem \ref{maintheoremlocalcoeff} to be satisfied for all representations of $\sigma$ of depth $\leq m$. 
View $\GL_n(F) \times \GL_t(F)$ as a maximal Levi subgroup of $\GL_{n+t}(F)$. Let $\bft, \bfb$ and $\bfu$ a maximal torus, Borel subgroup and unipotent radical of the Borel subgroup of $\GL_{n+t}(F)$ respectively. 
 Let $\Delta$ be the set of simple roots with respect to $\bft$ and $\bfb$, and let $\theta \subset \Delta$ such that $M_\theta = \GL_n(F) \times \GL_t(F)$. Let $\chi = \prod_{\alpha \in \Delta}\psi \circ \lu_\alpha^{-1}$.
 Let $W(\sigma, \chi)$ and $W(\tau, \chi)$ denote the Whittaker models of $\sigma$ and $\tau$ respectively (Note that here we are abusing notation and writing $\chi$ to denote its restriction to the Levi as well). Since $\depth(\sigma)\leq m$, using Equation \eqref{JosLanDepth}, it easily follows that $\depth(\sigma) \leq \cond(\sigma) \leq n^2 \depth(\sigma) + n^2$ (Also see Equation \eqref{DepthConductorRelation}).
Similarly, $\cond(\tau) \leq t^2m+t^2$.   Let $v_1 \in \sigma^{K_n(l)}$ and $v_2 \in \tau^{K_r(l)}$ be the essential vectors (cf. Theorem 2 of \cite{BH97}). 
Let $W_{v_1}$ and $W_{v_2}$ be the corresponding Whittaker functions in $W(\sigma, \chi)$ and $W(\tau, \chi)$ respectively. Then $W_{v_1}(e) \neq 0$ and $W_{v_2}(e) \neq 0$ by Proposition 1 of \cite{BH97}. 
This explains our choice of $l$. It is clear that $\chi$ is compatible with $\tw_0$.  Let $\chi'$ be the character of $U'$ obtained as in Section \ref{GR}. Then $\sigma'$ and $\tau'$ are $\chi'$-generic and Theorem \ref{maintheoremlocalcoeff} applies. 
\end{proof}
We are now ready to prove the main theorem of this section.
For $m \geq 1$, let
 \[ \mathcal{A}_F^0(n)_m:= \{  \sigma \in \mathcal{A}_F^0(n)\,|\,  \depth(\sigma) \leq m\}. \]
Similarly, let
 \[ \mathcal{G}_F^0(n)_m:= \{  \phi \in \mathcal{G}_F^0(n)\,|\,  \depth(\phi) \leq m\}. \]
Let $L_{n} = \rec_n^{-1}$ with $\rec_n$ as in Theorem \ref{LLC}. 
Since the Langlands map preserves the conductor of a representation, $L_n$ indeed maps $\mathcal{G}_F^0(n)_m$ to $\mathcal{A}_F^0(n)_m$ by Theorem \ref{depth}.
 For $\sigma \in  \mathcal{A}_F^0(n)_m$ and $I_{m,n}$ denoting the $m$-th Iwahori filtration subgroup of $\GL_n(F)$, we have $\sigma^{I_{m,n}} \neq 0$ by Lemma \ref{MPUC}. 
If $F$ and $F'$ are $(m+1)$-close, then since $\sigma$ is generic, we obtain a generic representation $\sigma'$ of $\GL_{n}(F')$ such that ${\sigma'}^{I'_{m,n}} \neq 0$ and $\depth(\sigma) = \depth(\sigma')$ (by Theorem \ref{depth} and Proposition 3.5.2 of \cite{Lem01}). 
Now, since $\sigma^{K_{m,n}} \cong \sigma'^{K_{m,n}'}$ using Corollary \ref{proofKazConj}, we see that $\sigma'$ is indeed supercuspidal by Theorem \ref{SCDS} $(b)$ (Here $K_{m,n}$ is the $m$-th usual congruence subgroup of $\GL_n(F)$). Writing $\sigma' = \kappa_m(\sigma)$, we obtain a well-defined map from $\kappa_m: \mathcal{A}_F^0(n)_m \rightarrow  \mathcal{A}_{F'}^0(n)_m$. Similarly, it is clear the Deligne isomorphism preserves depth,  and hence we obtain a well-defined map from $\Del_m: \mathcal{G}_{F}^0(n)_m \rightarrow  \mathcal{G}_{F'}^0(n)_m$ as long as $F$ and $F'$ are $(m+1)$-close.
\begin{theorem}\label{GLNDK}
For each $m \geq 1$, there exists $l = l(n,m) \geq m+1$  such that whenever $F$ and $F'$ are atleast $l$-close, the following diagram commutes:

\begin{displaymath}
    \xymatrix{
        \mathcal{G}_F^0(n)_m \ar[r]^{L_{n}} \ar[d]_{\Del_l} & \mathcal{A}_F^0(n)_m \ar[d]^{\kappa_l} \\
        \mathcal{G}_{F'}^0(n)_m \ar[r]_{L'_{n}}      &\mathcal{A}_{F'}^0(n)_m}
\end{displaymath}

where $\Del_l, \kappa_l, L_n$ are as above.
\end{theorem}
\begin{proof} We will prove this theorem by induction on $n$. When $n=1$, the fact that the Deligne-Kazhdan philosophy is compatible with local class field theory is Property (i) of Section \ref{Delignetheory}.
 Now, assume the theorem holds for $1 \leq t  \leq n-1$. Fix $m \geq 1$. For $l \geq m+1$, we see that the diagram above is well-defined.
 We now prove that this diagram is commutative for sufficiently large $l$. Let $\phi \in \mathcal{G}_F^0(n)_m$ and $\sigma = L_n(\phi)$. Let $A = 2m$. 
 Let $l_t = l(t,A)$ be an integer such that the theorem \label{LLCDK} is valid for $ \mathcal{G}_F^0(t)_{A}$. Such an $l_t$ exists by the induction hypothesis. 
 Now, let $l = \max\{2n^2m+n^2+4, l_1,...l_{n-1}\}$ and let $F$ and $F'$ be $l$-close. The $l$ chosen in this manner ensures that the theorem is valid for all $\tau \in \mathcal{A}_F^0(t)_{A},\; 1 \leq t \leq n-1$, and Theorem \ref{RSGF} holds.  We claim that the above diagram is commutative for this choice of $l$.  
 Let $\sigma_1 =\kappa_l^{-1} \circ L'_n \circ \Del_l(\phi)$. We need to show that $\sigma_1 \cong \sigma$. Note that $\omega_{\sigma_1} = \omega_\sigma$ by $(d)$ of Theorem \ref{LLC} and Property (i) of Section \ref{Delignetheory}. 
 Moreover, since $\kappa_l$, $L_n$, and $\Del_l$ preserve depth, which implies $\depth(\sigma) = \depth(\sigma_1)$. Also,  $L(s, \sigma \times \tau) = 1 = L(s, \sigma_1 \times \tau)\, \forall \, \tau \in \mathcal{A}_F^0(t),\; 1 \leq t \leq n-1$, since $\sigma, \sigma_1, \tau$ are all supercuspidal. Therefore 
 \[\epsilon(s, \sigma \times \tau, \psi) = \gamma(s, \sigma \times \tau, \psi) \text{ and } \epsilon(s, \sigma_1 \times \tau, \psi) = \gamma(s, \sigma_1 \times \tau, \psi).\]
  By Theorem \ref{charac}, it is sufficient to verify that
\begin{equation}\label{gcharac}
 \gamma(s, \sigma \times \tau, \psi) = \gamma(s, \sigma_1 \times \tau, \psi) 
 \end{equation}
 for $\tau \in  \mathcal{G}_F^0(t)_{A}, 1 \leq t \leq n-1$. By induction hypothesis, we have that \[\tau = L_t(\phi_{\tau}) = \kappa_{l}^{-1} \circ L'_t \circ \Del_{l}(\phi_{\tau}).\]
 Hence 
 \begin{align*}
 \gamma(s, \sigma_1 \times \tau, \psi) &= \gamma(s, (\kappa_l^{-1} \circ L'_n \circ \Del_l(\phi)) \times (\kappa_{l}^{-1} \circ L'_t \circ \Del_{l}(\phi_{\tau})), \psi)\\
 & = \gamma(s, (L'_n \circ \Del_l(\phi)) \times  (L'_t \circ \Del_{l}(\phi_{\tau})), \psi') \,\,\text { (by Theorem \ref{RSGF}) }\\
 & = \gamma(s,   \Del_l(\phi) \otimes \Del_{l}(\phi_{\tau}), \psi')\,\, \text { (by Theorem \ref{LLC}) }\\
 & =  \gamma(s,  \phi \otimes \phi_{\tau}, \psi)\,\, \text { (by Property (iv) of Section \ref{Delignetheory}) }\\
 & = \gamma(s, \sigma \times \tau, \psi) \,\, \text { (by Theorem \ref{LLC}). }\qedhere
  \end{align*}
\end{proof}
\begin{Corollary}\label{corGLNDK}
For each $m \geq 1$, there exists $l = l(n,m) \geq m+1$  such that whenever $F$ and $F'$ are atleast $l$-close, the following diagram commutes:

\begin{displaymath}
    \xymatrix{
        \mathcal{G}_F(n)_m \ar[r]^{L_{n}} \ar[d]_{\Del_l} & \mathcal{A}_F(n)_m \ar[d]^{\kappa_l} \\
        \mathcal{G}_{F'}(n)_m \ar[r]_{L'_{n}}      &\mathcal{A}_{F'}(n)_m}
\end{displaymath}
\end{Corollary}
\begin{proof}
The existence and uniqueness of the LLC for all irreducible, admissible representations of $\GL_n$ follows from the supercuspidal case (see \cite{Wed08}). Hence the corollary follows from the above theorem, and the results of Section \ref{PropertiesrepsCLF} that  various representation theoretic properties are compatible with the Deligne-Kazhdan correspondence.
\end{proof}

\section{The Langlands-Shahidi method for $\GSpin_5 \times \GL_t, t \leq 2$}\label{LSMETHODGSP4}
In Section \ref{LocalCoefficients}, we studied the local coefficients associated to generic representations of Levi subgroups of split reductive groups and proved its compatibility with the Deligne-Kazhdan correspondence (Theorem \ref{maintheoremlocalcoeff}). 
The Langlands-Shahidi method uses this theory of local coefficients to define the $\gamma$-factors. The theory of $\gamma$-factors and their expected properties can be found in great generality over local fields of characteristic 0 in \cite{Sha90}. A theory of $\gamma$-factors is available for split classical groups over local function fields in \cite{Lom09} and \cite{HL11}.

 The LLC for $\GSp_4 \cong \GSpin_5$ in \cite{GT11} over a local field $F'$ of characteristic 0 matches the first $L$- and $\gamma$-factor for generic (and non-supercuspidal) representations of $\GSpin_5(F') \times \GL_t(F'), t \leq 2$, with the Artin factors of the corresponding Langlands parameter.
 In order to establish the LLC for $\GSpin_5(F)$ for a local function field $F$ using the Deligne-Kazhdan correspondence, we need a theory of $L$- and $\gamma$-factors for representations of $\GSpin_5(F) \times \GL_t(F), t \leq 2$, and furthermore, the analogue of Theorem \ref{maintheoremlocalcoeff} for the $L$- and $\gamma$-factors.

There are two $\gamma$-factors that arise in the Langlands-Shahidi method for $\GSpin_{5}$: the first $\gamma$-factor is associated to the pair of generic representations of $\GSpin_5(F) \times \GL_t(F)$, and the second factor is the twisted symmetric square $\gamma$-factor. We will adapt the ideas from \cite{Lom09} to define the first  $L$-function and $\gamma$-factor for generic representations of $\GSpin_5(F) \times \GL_t(F), t \leq 2$, and prove its uniqueness (this $\gamma$-factor is needed for our main theorem). The desired properties and a unique characterization of the twisted symmetric square $\gamma$-factor can be found in \cite{GL14}, which addresses the Seigel Levi case of $\GSpin$ groups more generally.

Consider the pair $(\bfg_t, \bfm_t) = (\GSpin_{5+2t}, \GSpin_5 \times \GL_t)$ with Borel subgroup $\bfb_t = \bft_t\bfu_t$. Let $\Phi_t$ denote the set of roots of $\bft_t$  and let $\Phi^+_t$ (resp. $\Delta_t$) denote the set of positive roots (resp. simple roots) of $\bft_t$ in $\bfb_t$.  Fix a Chevalley basis $\{\lu_{\alpha,t}| \alpha \in \Phi_t\}$. Let $\bfp_t = \bfm_t\bfn_t$ be the standard parabolic subgroup with unipotent radical $\bfn_t \subset \bfu_t$. Let $\theta_t = \Delta_t \backslash \{\beta_t\}$ be the subset of $\Delta_t$ that determines $\bfp_t$. Let $\rho_{\bfp_t}$ denote the half sum of the roots in $\bfn_t$. Define 
\[\tilde{\beta}_{t} = \langle \rho_{\bfp_{t}}, \beta_{t}\rangle^{-1} \beta_{t}.\]
 Note that $^LM_t = \GSp_4(\C) \times \GL_t(\C)$ and $^LG = \GSp_{4+2t}(\C)$. 
Let $w_t = w_{l,\Delta_t}w_{l, \theta_t}$ and note that $w_t(\theta_t) = \theta_t$ that is, the parabolic subgroup $\bfp_t$ is self-associate. Let $r_t$ denote the adjoint action of $^LM_t$ on $^L\fn_t$, where $^L\fn_t$ denotes the Lie algebra of the complex dual group $^LN_t$ of $\bfn_t$ (Check Section 2 of \cite{Asg02} for the structure theory). Let $\Sim$ denote the similitude character and $\Spin$ denote the standard $4$-dimensional representation of $\GSp_4(\C)$ and let $\Std_t$ denote the standard representation of $\GL_t(\C)$.  Then by Proposition 5.6 of \cite{Asg02}, we have $r_{t} = r_{1,t} \oplus r_{2,t}$ with \[r_{1,t} = \Spin^\vee \otimes \Std_t \hspace{0.2in} \text{ and } \hspace{0.2in} r_{2,t} = \Sim^{-1} \otimes \Sym^2 (\Std_t).\]

Now, one can associate $L$- and $\gamma$-factors to irreducible generic representations of $M_t$ corresponding to each $r_{i,t}, \; i=1,2$. The one associated to $r_{1,t}$ is called the first $\gamma$-factor and the one associated to $r_{2,t}$ is called the second $\gamma$-factor. 

To do this, let us recall the crude functional equation satisfied by the local coefficients from Section 5.4 of \cite{Lom09}.  Let $\mathbb{A}$ denote the ring of adeles over a global function field $k$ such that $k_{v_0} = F$ for some place $v$. Let $\Psi = \otimes_{v}\Psi_v$ be a non-trivial additive character of $k \backslash \A$. Such a character has the property that $\Psi_v$ is unramified for almost all $v$. For each place $v$ define \begin{equation}\label{globalgenchar}
\calx_{v,t} = \prod_{\alpha \in \Delta_{t}} \Psi_v \circ \lu_{\alpha,t}^{-1}.
\end{equation} Then $\calx_{t} = \otimes_v \mathcal{X}_{v,t}$ is a generic character of $\bfu_{t}(k) \backslash \bfu_{t}(\A)$  that is unramified for almost all $v$. 
  Let $\Sigma = \otimes' \Sigma_v$ be a cuspidal automorphic representation of $\GSpin_{5}(\A)$ and let $\fT = \otimes'\fT_v$ be a cuspidal automorphic representation of $\GL_t(\A)$. Assume $\Sigma \times \fT$ is $\calx_{t}$-generic. 
Let $S$ be a finite set of places such that $\Sigma_v$,$\fT_v$ and $\calx_{v,t}$ are all unramified for all $v \notin S$.  
Then the partial $L$-function is defined by
\[L^S(s, \Sigma \times \fT, r_{i,t}) = \displaystyle{\prod_{v \notin S}L(s, \Sigma_v \times \fT_v, r_{i,t})}\]
where each of the factors in the right are defined using their Satake parametrization (cf. Definition 5.1 of \cite{Lom09}). 
 Note that the results in Section 5.4 of \cite{Lom09} are written in the context of split reductive groups, including $\GSpin_{5}$. We recall Theorem 5.14 (Crude Functional Equation) of that section below.  
\begin{equation}\label{Crf.eq} 
\displaystyle{\prod_{i=1}^2 L^S(is, \Sigma_v \times \fT_v, r_{i,t}) = \prod_{v \in S} C_{{\calx}_{v,t}}(s\tilde{\beta}_{t}, \;\Sigma_v \otimes \fT_v,\; \tw_{t}) \prod_{i=1}^2L^S(1-is, \Sigma_v \times \fT_v, r_{i, t}^\vee)}.
\end{equation}
 Note that the second $L$-function associated to $r_{2,t}$ is just the twisted symmetric square $L$-function. More precisely,
\[L^S(s, \Sigma \times \fT, r_{2,t}) = L^S(s,   \fT,   \Sym^2(\Std_t)\otimes \omega_\Sigma^{-1})\]
where $\omega_\Sigma$ denotes the central character of $\Sigma$ (cf. \cite{Sha97} and \cite{GL14}).
\subsection{Induction step}
Consider the pair $(\bfg_{0,t}, \bfm_{0,t}) = (\GSpin_{1+2t}, \GL_1 \times \GL_t), \; t \leq 2$. Let $\bfb_{0,t} = \bft_{0,t}\bfu_{0,t}$, $\Phi_{0,t}$, $\Delta_{0,t}$, $\theta_{0,t}$, $\beta_{0,t}$, $\tilde{\beta}_{0,t}$ and $w_{0,t}$ be the corresponding objects for $\bfg_{0,t}$. Fix a Chevalley basis $\{\lu_{\alpha,0,t}|\alpha \in \Phi_{0,t}\}$ of $\bfg_{0,t}$. Let $r_{0,t}$ denote the adjoint action of $^LM_{0,t}$ on $^L\fn_{0,t}$. Then again by Proposition 5.6 of \cite{Asg02},
$r_{0,t} = \Sim^{-1} \otimes \Sym^2(\Std_t)$ is irreducible. 

For each place $v$ define \begin{equation}\label{globalgenchar1}
\calx_{v,0,t} = \prod_{\alpha \in \Delta_{0,t}} \Psi_v \circ \lu_{\alpha,0,t}^{-1}.
\end{equation} Then $\calx_{0,t} = \otimes_v \mathcal{X}_{v,0,t}$ is a generic character of $\bfu_{0,t}(k) \backslash \bfu_{0,t}(\A).$  
  Let $\xi= \otimes' \xi_v$ be a Hecke character of $\A^\times$ and let $\fT = \otimes'\fT_v$ be a globally $\calx_{0,t}$-generic cuspidal automorphic representation of $\GL_t(\A)$. Let $S$ be a finite set of places such that $\xi_v$, $\fT_v$ and $\calx_{v,0,t}$ are all unramified for all $v \notin S$. As before, the partial $L$-function is defined by
\[L^S(s, \xi \times \fT, r_{0,t}) = \displaystyle{\prod_{v \notin S}L(s, \xi_v \times \fT_v, r_{0,t})}\]
where each of the factors in the right are defined using their Satake parametrization (cf. Definition 5.1 of \cite{Lom09}). 
 Now the crude functional equation becomes 

\begin{equation}\label{Crf.eq1} 
\displaystyle{ L^S(s, \xi_v \times \fT_v, r_{0,t}) = \prod_{\nu \in S} C_{{\calx}_{v,0,t}}(s\tilde{\beta}_{0,t},\; \xi_v \otimes \fT_v,\; \tw_{0,t}) L^S(1-s, \xi_v \times \fT_v, r_{0, t}^\vee)}.
\end{equation}

Recall that $v_0$ is a place of $k$ such that $k_{v_0} = F$. We write $\Psi_{v_0} = \psi$, $\calx_{v_0,t} = \chi_t $, and $\calx_{v_0, 0, t} = \chi_{0,t}$. 
\subsection{$\gamma$-factors - Definition and Properties }
Let $\psi$ be a non-trivial additive character of $F$ and let\begin{equation}\label{gencharaddchar}
\chi_t = \prod_{\alpha \in \Delta_t} \psi \circ \lu_{\alpha,t}^{-1}. 
\end{equation}
 Similarly, let \begin{equation}\label{gencharaddchar1}
\chi_{0,t} = \prod_{\alpha \in \Delta_{0,t}} \psi \circ \lu_{\alpha,0,t}^{-1}.
\end{equation}
  Let $\sigma$ be an irreducible, admissible representation of $\GSpin_5(F)$ and $\tau$ be an irreducible, admissible representation of $\GL_t(F)$. Note that since the center $\bfz_t$ of $\bfg_t$ (and the center of $\bfm_t$) is connected, $(\bft_t/\bfz_t)(F) = T_t/Z_t$ and hence $T_t$ acts transitively on the set of generic characters of $U_t$ (and similarly for $U_{M_t})$. Hence, if $\sigma \times \tau$ is an irreducible, admissible, generic representation of $M_t$, then it  is automatically generic with respect any generic character of $U_{M_t}$.
Now, the central character $\omega_\sigma$ of $\sigma$ is a character of $\GL_1(F)$. Following \cite{Sha90} in characteristic 0 and \cite{Lom09} in characteristic $p$ (also cf. \cite{GL14} for $r_{2,t}$), we define the $\gamma$-factors as follows.
Let 
\begin{equation}\label{secondgammafactor}
\gamma(s, \sigma \times \tau , r_{2,t}, {\psi}):= C_{\chi_{0,t}}(s\tilde{\beta}_{0,t}, \omega_\sigma \otimes \tau, \tw_{0,t}),
\end{equation}
 and define the first $\gamma$-factor so that it satisfies the equation
\begin{equation}\label{firstgammafactor}
C_{\chi_t}(s\tilde{\beta}_t, \sigma \otimes \tau, \tw_t) = \gamma(s, \sigma \times \tau, r_{1,t}, {\psi})\gamma(2s, \sigma \times \tau, r_{2,t}, {\psi}).
\end{equation}
We now want to apply Theorem \ref{maintheoremlocalcoeff} and deduce a similar result about the $\gamma$-factors defined above. To do this, we prove the following lemma.

\begin{lemma}\label{refinedlocalcoeff}Let $(\bfg_t, \bfm_t)$, $(\bfg_{0,t}, \bfm_{0,t})$ be as above, and let $\chi_t$, $\chi_{0,t}$ be defined using a nontrivial additive character $\psi$ as in Equations \eqref{gencharaddchar} and \eqref{gencharaddchar1}.  Let $\sigma \times \tau$ be an irreducible, admissible $\chi_t$- generic representation of $M_t$ with $\depth(\sigma), \depth(\tau) \leq m$. Let $F'$ be another non-archimedean local field and assume $F$ and $F'$ are $(m+4)$-close. Let $\sigma' \times \tau'$, $\chi_t'$ and $\chi_{0,t}'$ be the corresponding objects for $F'$. Then for $t=1,2$,
\[C_{\chi_t}(\nu, \sigma \otimes \tau, \tw_t) = C_{\chi'_t}(\nu, \sigma' \otimes \tau', \tw_t')\]
and
\[C_{\chi_{0,t}}(\nu, \omega_\sigma \otimes \tau, \tw_{0,t}) = C_{\chi_{0,t}'}(\nu, \omega_\sigma' \otimes \tau', \tw_{0,t}').\]
\end{lemma}
\begin{proof}
Note that this is a slight refinement of Theorem \ref{maintheoremlocalcoeff} for these groups. More precisely, we will show that the $m$, for which all the hypothesis Theorem \ref{maintheoremlocalcoeff} are satisfied,  is completely determined by the depth of the representation $\sigma \times \tau$. 
 
 We will prove this lemma for the case $(\bfg_t, \bfm_t)$ for $t=2$. The other cases would follow similarly. We drop the subscript $t$ throughout the proof of this lemma. So ${\bf M} = \GSpin_5 \times \GL_2 \hookrightarrow \GSpin_9  =\bfg$. Let $K_m$ be the $m$-th usual congruence subgroup of $G$. Then by Lemma \ref{MPUC}, $\sigma$ is a representation of $\GSpin_5(F)$ such that $\sigma^{K_{m+2} \cap \GSpin_5(F) }\neq 0$ and $\tau$ is a representation of $\GL_2(F)$ such that $\tau^{K_{m+2}\cap \GL_2(F)} \neq 0$. Furthermore,
\[\Delta = \{e_1 -e_2, e_2 - e_3, e_3 - e_4, e_4\}\]
and $\theta = \theta_1 \cup \theta_2$ where $\theta_1 = \{e_1 -e_2\}$  and  $\theta_2 = \{ e_3-e_4, e_4\}$
(with notation as in Section 2 of \cite{Asg02}). Let $\beta = e_2 - e_3$. 
Note that $\theta_1$ corresponds to the Dynkin diagram of $\GL_2$ and $\theta_2$ corresponds to the Dynkin diagram of $\GSpin_5$. The longest element for a group of type $B_n$ is $w_{l, \Delta} = -I$. The longest element $w_{l,\theta}$ of $M$ is  given by\begin{align*}
w_{l,\theta} :e_1 \rightarrow e_2,\,\,e_2 \rightarrow e_1,\,\,e_3 \rightarrow -e_3,\,\,e_4 \rightarrow - e_4.
\end{align*}
Therefore $w = w_{l, \Delta}w_{l, \theta}$ is given by
\begin{align*}
w:e_1 \rightarrow -e_2,\,\,e_2 \rightarrow -e_1,\,\,e_3 \rightarrow e_3,\,\,e_4 \rightarrow e_4.
\end{align*}
We see that $w$ acts as identity on $\theta$, and $w(\beta) = -e_1-e_3$. Fix $\psi_0$ an additive character of $F$ of conductor $\leq m$ and let $\chi_0 = \prod_{\alpha \in \Delta} \psi_0 \circ \lu_\alpha^{-1}$. Let $\chi_{0,1}= \psi_0 \circ \lu_{e_1-e_2}^{-1}$ and $\chi_{0,2} =(\psi_0 \circ \lu_{e_3-e_4}^{-1}) \times (\psi_0 \circ \lu_{e_4}^{-1})$.  Consider the Whittaker models $\calw(\tau, \chi_{0,1})$ and $\calw(\sigma, \chi_{0,2})$.  Using the Iwasawa decomposition, it is clear that there exists $v_1 \in \tau^{K_{m+2} \cap \GL_2(F)}$ and $\mu_1 \in X_*(\bf T \cap \GL_2)$ such that $W_{v_1}(\mu_1(\pi)) \neq 0$ where $W_{v_1} \in \calw(\tau, \chi_{0,1})$. Similarly, it is clear there exists $v_2 \in \sigma^{K_{m+2} \cap \GSpin_5(F)}$ and $\mu_2 \in X_*(\bf T \cap \GSpin_5)$ such that $W_{v_2}(\mu_2(\pi)) \neq 0$ where $W_{v_2} \in \calw(\sigma, \chi_{0,2})$.

Now we define $\chi = \prod_{\alpha \in \Delta}\chi_\alpha \circ \lu_\alpha^{-1}$ where
\[\chi_{e_1-e_2} = \pi^{\langle e_1-e_2, \mu_1 \rangle}\psi_0, \;\;\chi_\beta = \psi_0,\;\; \chi_{e_3-e_4} = \pi^{\langle e_3-e_4, \mu_2 \rangle}\psi_0,\;\; \chi_{e_4} = \pi^{\langle e_4, \mu_2 \rangle}\psi_0.\]
Then $\chi$ is a generic character of $U$ and $\chi$ is compatible with $\tw$ since $w$ acts as identity on $\theta$. Set $\chi_1 = \chi_{e_1-e_2} \circ \lu_{e_1-e_2}^{-1}$ and $ \chi_2 =(\chi_{e_3-e_4} \circ \lu_{e_3-e_4}^{-1}) \times (\chi_{e_4} \circ \lu_{e_4}^{-1})$. Let $W_{v_1}^1$ and $W_{v_2}^2$ denote the image of $v_1$ and $v_2$ in the Whittaker models $\calw(\tau, \chi_1)$ and $\calw(\sigma, \chi_2)$ respectively. Then $W_{v_1}^1(e), W_{v_2}^2(e) \neq 0$. In fact, we have defined $\chi$ so that this happens. Note that $\cond(\chi_\beta) = \cond(\psi_0) \leq m$. Hence all the hypothesis of Theorem \ref{maintheoremlocalcoeff} are satisfied for this $\chi$ and by considering the Whittaker models of $\sigma$ and $\tau$ with respect to this $\chi$ (See Condition \ref{Condition 1} and Condition \ref{Condition 2} in Section \ref{maintheorem}). Now, let $\chi'$ be the corresponding character of $U'$ as in Section \ref{GR}. Then $\sigma'\times \tau'$ is $\chi'$-generic and by Theorem  \ref{maintheoremlocalcoeff}, 
\[C_\chi(\nu, \sigma \otimes \tau, \tw) = C_{\chi'}(\nu, \sigma' \otimes \tau', \tw').\]
Now, there exists $a \in T$ such that $\chi = \chi_0 \circ \Ad_{\bf U}\tw_{l,\Delta}(a)$. Moreover, since both $\chi$ and $\chi_0$ are compatible with $\tw$, $\tw(a)a^{-1}$ lies in the center of ${\bf M}(F)$. It is also clear that $\chi' = \chi_0' \circ \Ad_{\bf U}\tw_{l,\Delta}(a')$ where $a \sim a'$ under the isomorphism ${\bf T}(F)/T_{\calp^m} \rightarrow{\bf T}(F')/T_{\calp'^m}$. Since
\[C_\chi(\nu, \sigma \otimes \tau, \tw) = \omega_\nu(\tw(a)a^{-1})C_{\chi_0}(\nu, \sigma \otimes \tau, \tw)\]
where $\omega_\nu$ is the central character of $\sigma\eta_\nu$, we have proved the lemma for all additive characters $\psi$ of conductor $\leq m$. Now, we can argue as above and deduce it for all non-trivial additive characters $\psi$.  
\end{proof}

\begin{proposition}\label{gammafactorCLF}
 Let $F$ be a non-archimedean local field and let $\chi_t$, $\chi_{0,t}$ be defined using a nontrivial additive character $\psi$ as in Equations \eqref{gencharaddchar} and \eqref{gencharaddchar1}. Let  $\sigma \times \tau$ be an irreducible, admissible $\chi_t$-generic representation of $M_t$, $t \leq 2$, of depth $\leq m$. Let $F'$ be another non-archimedean local field that is $(m+4)$-close to $F$. Let $\sigma' \times \tau'$ be the $\chi_t'$-generic representation of $M'_t$ as in Section \ref{GR}. Then
\begin{align*}
\gamma(s, \sigma \times \tau, r_{i,t}, \psi) &= \gamma(s, \sigma' \times \tau',r_{i,t}, \psi'), \; i=1,2.
\end{align*}
\end{proposition}
\begin{proof}
By Equation \eqref{secondgammafactor}, the equality of the second $\gamma$-factors over close local fields is clear by Lemma \ref{refinedlocalcoeff} (Here note that if $\sigma \leftrightarrow \sigma'$, then $\omega_\sigma \leftrightarrow \omega_{\sigma'}$).  By Equation \eqref{firstgammafactor} (or property (2) of Theorem 3.5 of \cite{Sha90}), we have
\begin{equation}\label{LSM0}
 C_{\chi_t}(s\tilde{\beta}_t, \sigma \otimes \tau, \tilde{w}_t)=\gamma(s, \sigma \times \tau, r_{1,t}, {\psi})\gamma(2s, \sigma \times \tau, r_{2,t}, {\psi}).
\end{equation}
Now, the equality of the first $\gamma$-factor follows from the above and another application of Lemma \ref{refinedlocalcoeff}  to the LHS of Equation \eqref{LSM0} (or Equation 3.11 of \cite{Sha90}).
\end{proof} 

We list some properties satisfied by the $\gamma$-factors. The proofs of all these claims can be found in \cite{Sha90} in characteristic 0 and Section 6 of \cite{Lom09} for classical groups in characteristic $p$, and the same ingredients would suffice to verify the properties below. New proofs of Properties (III), (V), and (VI) are obtained using Proposition \ref{gammafactorCLF}.\\
(I)\hspace{2pt}(\textit{Rationality}) The factors $\gamma(s, \sigma \times \tau, r_{i,t}, {\psi}), \; i=1,2,$ are rational functions of $q^{-s}$.  This follows from the inductive definition of $\gamma$-factors together with the rationality of the local coefficients. \\
(II)\hspace{2pt}(\textit{Compatibility with Artin factors}) Assume $\sigma \times \tau$ has an Iwahori fixed vector. Let $\phi: \WD_F \rightarrow \mathord{}^LM$ be the homomorphism of the Weil-Deligne group of $F$ parametrizing $\sigma \times \tau$. For each $i =1,2$, let $L(s, r_{i,t} \circ \phi)$ and $\epsilon(s, r_{i,t} \circ \phi, \psi)$ be the Artin $L$-function and root number attached to $r_{i,t} \circ \phi$. Then
\[\gamma(s, \sigma \times \tau, r_{i,t}, \psi) = \epsilon(s, r_{i,t} \circ \phi, \psi)\frac{ L(1-s, r_{i,t}^\vee \circ \phi)}{L(s, r_{i,t} \circ \phi)}.\]
This follows from the Proposition 3.4 of \cite{Sha90} (cf. Theorem 6.4 of \cite{Lom09} for classical groups).\\
(III)\hspace{2pt}(\textit{Multiplicativity of $\gamma$-factors}) The multiplicativity property of the local coefficients has been established in full generality, independent of characteristic, in \cite{Sha81}. To prove the mutiplicativity property of the $\gamma$-factors in positive characteristic, we combine Proposition \ref{gammafactorCLF}, the results of Section \ref{IndRep}, and the mutiplicativity property of $\gamma$-factors in characteristic 0 established in Theorem 3.5(3) of \cite{Sha90}. More precisely, let $\sigma$ be a representation of $\GSpin_5(F)$ such that
\[ \sigma \hookrightarrow \Ind_Q^{\GSpin_5(F)} \sigma_1\otimes 1,\]
where ${\bf Q}$ is a parabolic subgroup of $\GSpin_5$ (Cf. Chapter 2.1 of \cite{RS07} for a description of parabolic subgroups of $\GSpin_5 \cong \GSp_4$) and $\sigma_1$ is a representation of $M_Q$, with ${\bf M_Q}$ the Levi subgroup of ${\bf  Q}$. 
Similarly, let $\tau$ be a representation of $\GL_t(F)$ such that 
\[\tau \hookrightarrow \Ind_{R}^{\GL_t(F)}\tau_1 \otimes 1,\]
where ${\bf R}$ is a parabolic subgroup of $\GL_t$ with Levi ${\bf M_R}$ and $\tau_1$ is an irreducible, admissible representation of $M_R$. Let $I_m$ be the $m$-th Iwahori congruence subgroup of $\GSpin_{5+2t}(F)$. Choose $m$ large enough such that $\sigma^{I_m \cap \GSpin_5(F)} \neq 0$ and $\tau^{I_m \cap \GL_t(F)} \neq 0$. Then it is easy to see that $\sigma_1^{I_m \cap M_Q} \neq 0$ and $\tau_1^{I_m \cap M_R}\neq 0$.  Choose a local field $F'$ of characteristic 0 such that $F'$ is $l$-close to $F$, where $l$ is large enough so that Proposition \ref{gammafactorCLF} is valid and Theorem \ref{RSGF} holds for $n=2$. Let $\sigma' $ be the representation of $\GSpin_5(F')$ corresponding to $\sigma$, and similarly let $\tau'$ be the representation of $\GL_t(F')$ corresponding to $\tau$. Similarly we obtain representations $\sigma_1'$ and $\tau_1'$. By the results of Section \ref{IndRep}, we see that
\[ (\Ind_Q^{\GSpin_5(F)} \sigma_1\otimes 1)^{I_m \cap \GSpin_5(F)} \cong (\Ind_{Q'}^{\GSpin_5(F')} \sigma_1'\otimes 1)^{I_m' \cap \GSpin_5(F')}\]
and similarly,
\[(\Ind_{R}^{\GL_t(F)}\tau_1\otimes 1)^{I_m \cap \GL_t(F)} \cong (\Ind_{R'}^{\GL_t(F')}\tau_1'\otimes 1)^{I_m' \cap \GL_t(F')}\]
Consequently, we get an embedding
\[ \sigma'  \hookrightarrow \Ind_{Q'}^{\GSpin_5(F')} \sigma_1'\otimes 1 \hspace{0.2in}\text{ and } \hspace{0.2in}\tau' \hookrightarrow \Ind_{R'}^{\GL_t(F')}\tau_1'\otimes 1.\]
Combining Proposition \ref{gammafactorCLF}, Theorem \ref{RSGF}, and Observation (e) of Section \ref{maintheorem} with the multiplicativity property for $\sigma' \times \tau'$ in characteristic 0 (Theorem 3.5  (3) of \cite{Sha90}) proves this property for $\sigma \times \tau$ in characteristic $p$.\\
(IV)\hspace{2pt}(\textit{Global Functional Equation}) Let $\Sigma = \otimes' \Sigma_v$ be a globally generic cuspidal representation of $\GSpin_5(\A)$ and let $\fT = \otimes' \fT_v$  globally generic cuspidal representation of $\GL_t(\A)$. Let $\calx_t$ and $\calx_{0,t}$ be as in Equations \eqref{globalgenchar} and \eqref{globalgenchar1} respectively defined using $\Psi$. Let $S$ be a set of places of $k$ such that $\Sigma_v$, $\fT_v$ and $\Psi_{v}$ are all unramified outside $S$. Then for $i = 1,2$,

\begin{equation}
L^S(s, \Sigma \times \fT, r_{i,t}) = \displaystyle{\prod_{\nu \in S}}\gamma(s, \Sigma_v \times \fT_v, r_{i,t}, \Psi_{v}) L^S(1-s, \Sigma \times \fT, r_{i,t}^\vee).
\end{equation}
 As explained before, note that with $(\bfg_{0,t}, \bfm_{0,t})$ as before, $r_{2,t}$ simply denotes the adjoint action of $^LM_{0,t}$ on the Lie algebra of the unipotent radical $^L\fn_{0,t}$ of $^LN_{0,t}$. Now Equation \eqref{Crf.eq1} yields the above equation for the second $\gamma$-factor. Now, another application of Equation \eqref{Crf.eq} coupled with the functional equation for the second $\gamma$-factor yields the functional equation for the first $\gamma$-factor (cf. Theorem 6.5 of \cite{Lom09} for classical groups). \\
(V)\hspace{2pt}(\textit{Stability}) Let $\sigma_1$ and $\sigma_2$ be irreducible, admissible, generic representations of $\GSpin_5(F)$ with the same central character. Then for any sufficiently ramified character $\eta$ of $\GL_1(F)$, we have
\[\gamma(s, \sigma_1\times \eta,r_{1,t}, \psi) = \gamma(s, \sigma_2 \times \eta, r_{1,t}, \psi).\]
This is Theorem 4.1 of \cite{AS06} for local fields of characteristic 0. It follows that this holds for local fields of positive characteristic using Proposition \ref{gammafactorCLF}.

Property (V) for the first $\gamma$-factor, combined with the remaining properties above, can now be used to prove that the first $\gamma$-factor is indeed uniquely characterized by these properties.
In fact, the stability property can be used to prove that the first $\gamma$-factor for representations of $\GSpin_5(F) \times \GL_1(F)$ is uniquely determined, as in Section 3.2 of \cite{Lom12}. 
Now, we can use arguments similar to the proof of Theorem 6.15 of \cite{Lom09} (Also see Section 3.3 of \cite{Lom12}) to prove the uniqueness of the first $\gamma$-factor for representations of $\GSpin_5(F) \times \GL_t(F)$. We refer the reader to \cite{GL14} where the uniqueness of the second $\gamma$-factor is dealt with in more generality. We list one additional property of these $\gamma$-factors.\\
(VI)\hspace{2pt}(\textit{Local functional equation}) The $\gamma$-factors satisfy the following local functional equation:
\[\gamma(s, \sigma \times \tau, r_{i,t} , \psi)\gamma(1-s, \sigma \times \tau , r_{i,t}^\vee, \bar{\psi}) = 1.\]
This is Theorem 6.15 of \cite{Lom09} for the first $\gamma$-factor and  Corollary 4.4 of \cite{HL11} for the second $\gamma$-factor for classical groups. For our case, this property is true for $\gamma$-factors in characteristic 0 (cf. Equation 3.10 of \cite{Sha90}).   Note that $\gamma(1-s, \sigma \times \tau, r_{i,t}^\vee, \bar{\psi}) = \gamma(1-s, (\sigma \times \tau)^\vee, r_{i,t}, \bar{\psi})$, and if $\sigma \times \tau \leftrightarrow \sigma' \times \tau'$ over close local fields, then $(\sigma \times \tau)^\vee \leftrightarrow (\sigma'\times \tau')^\vee$. Now, use Proposition \ref{gammafactorCLF} to deduce this property in characteristic $p$. 
\subsection{$L$-factors and root numbers}
We define the first $L$-factor, following Section 7 of \cite{Sha90} and \cite{Lom09}, using the Langlands classification.  Let $\sigma \times \tau$ be a tempered representation. Let $P_{\sigma \times \tau, r_{1,t}}(T)$ be the polynomial satisfying $P_{\sigma, \times \tau, r_{1,t}}(0) = 1$ and such that $P_{\sigma \times \tau, r_{1,t}}(q^{-s})$ is the numerator of $\gamma(s, \sigma \times \tau, r_1, \psi)$. Set
\[L(s, \sigma \times \tau, r_{1,t}) = P_{\sigma \times \tau, r_{1,t}}(q^{-s})^{-1}\]
and
\[L(s, \sigma \times \tau, r_{1,t}^\vee) = P_{\sigma^\vee \times \tau^\vee, r_{1,t}}(q^{-s})^{-1}.\]
By Property (VI), we see that
\begin{equation*}
\gamma(s, \sigma \times \tau, r_{1,t}, \psi)\frac{ L(s, \sigma \times \tau, r_{1,t})}{L(1-s, \sigma \times \tau, r_{1,t}^\vee)}
\end{equation*}
is a monomial in $q^{-s}$ which we denote by $\epsilon(s, \sigma \times \tau, r_{1,t}, \psi)$, the root number attached to $\sigma \times \tau$ and $r_{1,t}$. Hence
\begin{equation}\label{epsilonfactor}
 \gamma(s, \sigma \times \tau, r_{1,t}, \psi)=\epsilon(s, \sigma \times \tau, r_{1,t}, \psi)\frac{L(1-s, \sigma \times \tau, r_{1,t}^\vee)}{L(s, \sigma \times \tau, r_{1,t})}
 \end{equation}
Note that the $L$-factor defined above is independent of $\psi$, as explained in Section 7 of \cite{Sha90}. 
The definition of the $L$-function is now extended to quasi-tempered representations using analytic continuation, and then defined for any irreducible, admissible, generic representation using the Langlands classification. Finally, the root number is defined as in Equation \eqref{epsilonfactor}.

\begin{proposition}\label{LEpsilonCLF}
 Let $F$ be a non-archimedean local field and let $\chi_t$, $\chi_{0,t}$ be defined using a nontrivial additive character $\psi$ as in Equations \eqref{gencharaddchar} and \eqref{gencharaddchar1}. Let  $\sigma \times \tau$ be an irreducible, admissible $\chi_t$-generic representation of $M_t$. Let $l$ be large enough so that Proposition \ref{gammafactorCLF} holds and Theorem \ref{RSGF} holds for $n=2$, and let $F'$ be another field that is $l$-close to $F$. Let $\sigma' \times \tau'$ be the $\chi_t'$-generic representation of $M'$ as in Section \ref{GR}. Then
\begin{align*}
L(s, \sigma \times \tau, r_{1,t}) &= L(s, \sigma' \times \tau', r_{1,t}),\\
\epsilon(s, \sigma \times \tau, r_{1,t}, \psi) &= \epsilon(s, \sigma' \times \tau', r_{1,t}, \psi').
\end{align*}
\end{proposition}
\begin{proof}

In Proposition \ref{gammafactorCLF}, we proved the equality of $\gamma$-factors over close local fields. To prove the equality of the $L$-factors, first assume that $\sigma \times \tau$ is tempered (resp. quasi-tempered). 
Then arguing as in Section \ref{Discreteseries}, we see that $\sigma' \times \tau'$ is also tempered (resp. quasi-tempered). Now the equality of $L$-factors is clear in this case using Proposition \ref{gammafactorCLF}. Suppose $\sigma \times \tau$ is any irreducible $\chi_t$-generic representation. Let
 $(Q, \nu, \rho)$ is the Langlands' data for $\sigma \times \tau$. Since induced representations from parabolic subgroups correspond over close local fields, we see that $(Q', \nu, \rho')$ is the Langlands data for $\sigma' \times \tau'$ where $\rho \leftrightarrow \rho'$ under the Hecke algebra isomorphism $\fH(Q, Q \cap I_l) \cong \fH(Q', Q' \cap I_l')$. The equality of the $L$-factors now follows.

The equality of the $\epsilon$-factors is clear from the above and Equation \eqref{epsilonfactor}.
\end{proof} 
\section{The local Langlands correspondence for $\GSp_4(F')$ with $\Char(F') = p >2$}\label{LLCGSP4POSITIVE}
In this section, we prove that the local Langlands correspondence for $\GSp_4(F)$ in characteristic 0 preserves depth, re-characterize the map in characteristic 0, and establish the correspondence for $\GSp_4(F')$ where $F'$ is a local field of odd positive characteristic.
\subsection{Depth preservation}\label{DepthPreservation}
Let us now recall the main theorem of \cite{GT11}.
 Let $ \Pi(\GSp_4(F))$ be the set of equivalence classes of irreducible, admissible representations of $\GSp_4(F)$ and 
 $\Phi(\GSp_4(F))$ be the set of admissible homomorphisms from $\WD_F \rightarrow \GSp_4(\mathbb{C})$, where the homomorphisms are taken upto $\GSp_4(\mathbb{C})$-conjugacy. 
\begin{theorem}\text{\emph{(Theorem 10.1 of \cite{GT11})}}\label{char0} 
There is a unique surjective finite-to-one map \[L: \Pi(\GSp_4(F)) \longrightarrow \Phi(\GSp_4(F)) \] satisfying:
\begin{enumerate}[(a)]
\item the central character $\omega_\sigma$ of $\sigma$ corresponds to the similitude character $\operatorname{sim}(\phi_\sigma)$ of $\phi_\sigma: = L(\sigma)$ via local class field theory.
\item $\sigma$ is a discrete series representation iff $\phi_\sigma$ does not factor through any proper Levi subgroup of $\GSp_4(\mathbb{C})$.
\item if $\sigma$ is generic or non-supercuspidal, then for any irreducible representation $\tau$ of $\GL_t(F)$ with $t \leq 2$, 
\[ \begin{cases}
L(s, \sigma \times \tau) &= L(s, \phi_\sigma \otimes \phi_\tau) \\
\gamma(s, \sigma \times \tau, \psi) &= \gamma(s, \phi_\sigma \otimes\phi_ \tau, \psi)
\end{cases} \]
where the factors on the right are of Artin type, and the factors on the left are obtained using the Langlands-Shahidi method.
\item if $\sigma$ is non-generic supercuspidal, then for any supercuspidal representation $\tau$ of $\GL_t(F)$ with $t \leq 2$, the Plancherel measure $\mu(s, \sigma \times \tau)$ is equal to 
\begin{align*}
 \gamma(s, \phi_\sigma^\vee \otimes \phi_\tau, \psi) \gamma(-s, \phi_\sigma \otimes \phi_\tau^\vee, \bar{\psi})\gamma(2s, \Sym^2\phi_\tau \otimes& \Sim\phi_\sigma^{-1}, \psi)\\
& \gamma(-2s, \Sym^2\phi_\tau^{\vee} \otimes \Sim\phi_\sigma, \bar{\psi}) 
\end{align*}
\end{enumerate}
\end{theorem}
In part $(c)$ above, we have followed the notation of \cite{GT11}, that is,
\[L(s, \sigma \times \tau) := L(s, \sigma^\vee \times \tau, r_{1,t}); \hspace*{0.2in} \gamma(s, \sigma\times \tau, \psi) := \gamma(s, \sigma^\vee \times \tau, r_{1,t}, \psi);\; t = 1,2.\]
Note that the definition above is extended to non-generic, non-supercuspidal representations as explained in Page 11 of \cite{GT11}.

By Theorem 2.3.6.4 of \cite{Yu09} (recalled as Theorem \ref{depth} in this article), we know that the Local Langlands correspondence for $\GL_n$ preserves depth. Since the LLC for $\GSp_4$ was established using similitude theta correspondence and the LLC for $\GL_n$, we combine the results of \cite{Pan02} and \cite{Yu09} to show that the LLC for $\GSp_4$ preserves depth in odd residue charateristic (The results of \cite{Pan02} hold only for odd residue characteristic). 

\begin{proposition}Let $F$ be a non-archimedean local field of characteristic 0 with residue characteristic $p >2$. Let $\sigma \in \Pi(\GSp_4(F))$. Let $\phi_\sigma = L(\sigma)$ as above. Then
\[\depth(\sigma) = \depth(\phi_\sigma).\]
\end{proposition}
\begin{proof}
This proposition was also observed in the proof of Proposition 1 of \cite{Sor10}, but the details of the proof were  omitted. The author thanks Professor Wee Teck Gan for explaining some of the steps in this proof. \\
\textbullet \hspace{2pt} If $\tau$ is an irreducible, admissible representation of $\GL_n(F)$, and $\phi_\tau$ is the Langlands parameter attached to $\tau$, then $\depth(\tau) = \depth(\phi_\tau)$ by Theorem \ref{depth}. Let $D$ be a division algebra over $F$ of dimension $n^2$. Let $G' = D^\times$ and $G =\GL_n(F)$. Let $\tau^D$ be an essentially square integrable representation of $D^\times$ that corresponds to an essentially square integrable $\tau$ via the Jacquet-Langlands correspondence.  Then by the main theorem of \cite{LR03}, $\depth(\tau^D) = \depth(\tau)$. \\
\textbullet \hspace{2pt} If $\tau_i$ is an irreducible, admissible representation of $\GL_{n_i}(F), i=1,2$, then \[\depth(\tau_1 \times \tau_2) = \max\{\depth(\tau_1), \depth(\tau_2)\}.\] This follows from the simple observation that $(\tau_i)^{G_{{x_i}, r+}} \neq 0$, then $(\tau_i)^{G_{{x_i}, s+}} \neq 0$ for each $s \geq r$. \\
\textbullet \hspace{2pt} Now let us recall the definition of the $L$-map from Section 7 of \cite{GT11}. The similitude orthogonal groups that appear in relation to the theta correspondence for an irreducible representation of $\GSp_4(F)$ are
\begin{align}\label{Orthogonal}
\GSO_{2,2} &\cong (\GL_2 \times \GL_2)/\{(z, z^{-1}):z \in \GL_1\},\nonumber\\
\GSO_{4,0} &\cong (D^\times\times D^\times)/\{(z, z^{-1}):z \in \GL_1\},\\
\GSO_{3,3} &\cong (\GL_4 \times \GL_1)/\{(z, z^{-2}):z \in \GL_1\};\nonumber
\end{align}
where $D$ is a quaternion division algebra over $F$. Via these isomorphisms, an irreducible, admissible representation of $\GSO_{2,2}(F)$ is of the form $\tau_1 \otimes \tau_2$ where $\tau_1$ and $\tau_2$ are irreducible representations of $\GL_2(F)$  with the same central character, and a representation of $\GSO_{4,0}(F)$ is of the form $\tau_1^D \otimes \tau_2^D$ where $\tau_i^D,\;i=1,2,$ are representations of $D^\times$ with the same central character. Similarly, an irreducible representation of $\GSO_{3,3}(F)$ is of the form $\tau \otimes \mu$ where $\tau$ is a representation of $\GL_4(F)$ and  $\omega_\tau = \mu^2$. 
If $\sigma$ participates in theta correspondence with $\GSO_{2,2}(F)$ or $\GSO_{4,0}(F)$, then
\[\sigma = \theta(\tau_1 \otimes \tau_2) \;\text{ or }\; \sigma = \theta(\tau_1^D \otimes \tau_2^D)\]and $\phi_\sigma = \phi_{\tau_1} \oplus \phi_{\tau_2}$, obtained by combining the Jacquet-Langlands (for the $\GSO_{4,0}$ case) and the Local Langlands for $\GL_2$.  If $\sigma$ participates in theta correspondence with $\GSO_{3,3}$, then $\theta(\sigma) = \tau \otimes \mu$ and $\phi_\sigma = \phi_\tau \times \mu$. \\
\textbullet \hspace{2pt} Let $\bft$ denote the standard maximal torus in $\GSp_4$, $\bft^{\text{der}}$ the maximal torus in $\Sp_4$ and $\bfz \cong \bG_m$ the center of $\GSp_4$.  If $\sigma$ is an irreducible representation of $\GSp_4(F)$ and $\theta(\sigma)$ is the representation of an appropriate $\GSO(V)$ that corresponds to $\sigma$ via theta correspondence, then $\depth(\sigma) = \depth(\theta(\sigma))$. To prove this, note that if
\[\sigma|_{\Sp_4(F)} = \oplus \sigma_i\]
then
\[\theta(\sigma)|_{\SO(V)} = \oplus \theta(\sigma_i)\]
where $\sigma_i$ corresponds to $\theta(\sigma_i)$ via the isometry theta correspondence by Lemma 2.2 of \cite{GT-011}. Now, $\depth(\sigma_i) = \depth(\theta(\sigma_i))$ by \cite{Pan02} (The results of \cite{Pan02} hold only when the residue characteristic is $> 2$). 
Therefore, 
\[\depth(\sigma|_{\Sp_4(F)}) = \min\{ \depth(\sigma_i)\} = \min\{ \depth(\theta(\sigma_i)\} = \depth(\theta(\sigma)|_{SO(V)}) = A\text{  (say).}\]
Furthermore, $\omega_{\sigma} = \omega_{\theta(\sigma)}$. Let $A_1 = \max\{\depth(\omega_\sigma), A\}$. We claim that
\[\depth(\sigma) = A_1. \]
Note that the kernel of the similitude character $\operatorname{sim}:\GSp_4(F) \rightarrow F^{\times}$ is $\Sp_4(F)$ and  $\operatorname{sim}(Z) = (F^{\times})^2$. Since the residue characteristic $p >2$, we have by Hensel's lemma that
\begin{align*}
T_{\calp^{1+\lfloor r\rfloor }} & = Z_{\calp^{1+\lfloor r\rfloor }}T^{\text{der}}_{\calp^{1+\lfloor r\rfloor }},  \text{ for each $r \geq 0$}.
\end{align*}
Therefore, using Equations \eqref{MPpresentation} and \eqref{MPpresentation1}, we see that
\[G_{x,r+} = Z_{\calp^{1+ \lfloor r\rfloor }}G^{\text{der}}_{x,r+} \text{ for each }r \geq 0.\] 
Now, since $\depth(\sigma|_{\Sp_4(F)}) =A$, there exists $x \in \calb$ such that $(\sigma|_{\Sp_4(F)}) ^{G^{\text{der}}_{x,A+}} \neq 0$. In particular, $(\sigma|_{\Sp_4(F)}) ^{G^{\text{der}}_{x,A_1+}} \neq 0$.  Also $\omega_\sigma|_{Z_{\calp^{1+ \lfloor A_1 \rfloor}} }=1$. Then $\sigma^{G_{x,A_1+}} \neq 0$. Hence $\depth(\sigma) \leq A_1$. On the other hand, if there exists $x \in \calb$ such that $\sigma^{G_{x,r+}} \neq 0$, then $\omega_\sigma|_{ Z_{\calp^{1+\lfloor r\rfloor }}} = 1$ and $(\sigma|_{\Sp_4(F)}) ^{G^{\text{der}}_{x,r+}} \neq 0$. Therefore $A_1 \leq r$. Hence $A_1 = \depth(\sigma)$. Similarly, we can show that
\[\depth(\theta(\sigma)) = \max\{\depth(\omega_{\theta(\sigma)}),\depth( \theta(\sigma)|_{\SO(V)})\},\] where $\SO(V)$ is one of the groups in Equation \eqref{Orthogonal} above. Hence
\[\depth(\sigma) = \depth(\theta(\sigma)).\] 

Now the proposition follows by combining the observations above.
\end{proof}
\subsection{A re-characterization of the LLC for $\GSp_4$ over local fields of characteristic 0}

Let us introduce the following notation. For $m \geq1$, Let $ \Pi(\GSp_4(F))_m$ be the set of equivalence classes of irreducible, admissible representations of $\GSp_4(F)$ of depth $\leq m$ and similarly, let
 $\Phi(\GSp_4(F))_m$ be the set of admissible homomorphisms from $\WD_F \rightarrow \GSp_4(\mathbb{C})$ of depth $\leq m$, where the homomorphisms are taken up to $\GSp_4(\mathbb{C})$-conjugacy.

From the results of Section \ref{DepthPreservation}, we know that \[\depth(\sigma) = \depth(\phi_\sigma).\] Hence the above map restricts to give a well-defined map 
\[L_m:\Pi(\GSp_4(F))_m \rightarrow \Phi(\GSp_4(F))_m.\]
In this subsection, we re-characterize Theorem \ref{char0} into a corresponding statement about the representations of depth $\leq  m$ (cf. Theorem \ref{charac} for an analogous statement for $\GL_n$). 

Before proceeding, let us recall some notation from \cite{GT11} about representations of $\WD_F$. Let $\G_F(n), \G_F^0(n)$ be as in Section \ref{LLCGLNCLF} and let and $\G_F^2(n)$ denote the set irreducible (indecomposable) representations of $\WD_F$ of dimension $n$, respectively. An irreducible representation of $\WD_F$ is of the form $\rho \otimes S_r$ where $\rho$ is an irreducible representation of $W_F$ and $S_r$ is the irreducible $r$-dimensional representation of $\SL_2(\C)$. If $\phi$ is a semi-simple representation of $\WD_F$, let $m_\phi(\rho, r)$ denote the multiplicity of the representation $\rho \otimes S_r$ in $\phi$. Also, let $v_\phi(\rho, r)$ denote the order of the pole at $s=0$ of the local $L$-factor $L(s, \phi \otimes (\rho \otimes S_r)^\vee)$. We then have
\[L(s, \rho \otimes S_r) = L(s + \frac{r-1}{2}, \rho),\]
and with $t = \min(r,k)$,
\[L(s, (\rho_1 \otimes S_r)\otimes(\rho_2 \otimes S_k)^\vee) = \displaystyle{\prod_{j=0}^{t-1} L(s +\frac{r+k-2-2j}{2}, \rho_1 \otimes \rho_2^\vee}).\]
Also, for any semi-simple representation of $\WD_F$, 
\begin{equation}\label{polesmult}
v_\phi(\rho,r) = \displaystyle{\sum_{k \geq 1} \sum_{j=1}^{\min\{r,k\}}m_\phi(\rho|\;-\;|^{-(\frac{r+k}{2}-j)},k)}.
\end{equation}

Furthermore, note that if $\phi$ is a semi-simple representation of $\WD_F$ with $\phi = \oplus \phi_i$, then
\begin{equation}\label{DepthAdditiveMax}
\depth(\phi) = \displaystyle{\max\{\depth(\phi_i)\}}.
\end{equation}
Let $\G_F(n)_m$ be the isomorphism classes of $n$-dimensional representations of $\WD_F$ of depth $\leq m$.
We now prove the following theorem.
\begin{lemma}\label{conversetheoremrefined}
Let $\phi_1, \phi_2 \in \G_F(4)$ such that $\phi_1$ and $\phi_2$ do not have an irreducible 3-dimensional constituent. Assume $\det(\phi_1) = \det(\phi_2)$. Let $m \geq 1$ be such that $\max\{\depth{\phi_1}, \depth{\phi_2}\} \leq m$. Suppose for  each $\tau \in \G_F(r)_{2m}, r=1,2$,
\begin{align}\label{Lepsilon}
L(s, \phi_1 \otimes \tau) &= L(s, \phi_2 \otimes \tau),\\ \nonumber
\epsilon(s, \phi_1 \otimes \tau, \psi) &= \epsilon(s, \phi_2 \otimes \tau, \psi).
\end{align}
Then, $\phi_1 \cong \phi_2$. 
\end{lemma}
\begin{proof}
This lemma would  simply be Case 1 of  Theorem 10.1, \cite{GT11},  if there had been no restriction on the depth of $\tau$. The proof here is obtained by combining Theorem \ref{stability} with the ingredients in Theorem 10.1 of \cite{GT11} and Equation \eqref{DepthAdditiveMax}.  We write $m_i(\rho,r)$ and $\nu_i(\rho,r)$ instead of $m_{\phi_i}(\rho,r)$ and $\nu_{\phi_i}(\rho,r)$. We observe the following.

\noindent \textbf{Observation}:  Let $\phi_1, \phi_2 \in \G_F(2)$. Assume $\det(\phi_1) = \det(\phi_2)$ and suppose for each $\chi \in \G_F(1)_{2m}$, 
\begin{align*}
L(s, \phi_1 \otimes \chi) &= L(s, \phi_2 \otimes \chi),\\
\epsilon(s, \phi_1 \otimes \chi, \psi) &= \epsilon(s, \phi_2 \otimes \chi, \psi).
\end{align*}
Then, $\phi_1 \cong \phi_2$. To see this, note that if $\phi_1, \phi_2 \in \G_F^0(2)$, then we are done by Theorem \ref{charac}.  Suppose $\phi_1 \in \G_F(2) \backslash \G_F^0(2)$.  Applying formula \eqref{polesmult} to $\phi_1, \phi_2$, we have
\[v_i(\chi, 1) = m_i(\chi,1) +m_i(\chi|\;-\;|^{-1/2},2) \]
and 
\[v_i(\chi|\;-\;|^{1/2}, 1) = m_i(\chi|\;-\;|^{1/2},1) +m_i(\chi,2)\]
 for all $\chi \in \G_F(1)_{2m}$. Then there are two possibilities. The first possibility is that $\phi_1 = \chi \otimes S_2$. Note that  $\depth(\chi) \leq m$. Now, suppose $\phi_2 \neq \chi \otimes S_2$, then $m_2(\chi|\;-\;|^{1/2},1) =1$. Therefore
\[\phi_2 = (\chi|\;-\;|^{1/2} \otimes S_1) \oplus (\eta \otimes S_1).\]
Then $\depth(\eta) \leq m$ and $L(s, \phi_2 \otimes (\eta \otimes S_1)^\vee)$ has a pole at $s=0$, but $L(s, \phi_1 \otimes (\eta \otimes S_1)^\vee)$ is holomorphic at $s=0$, a contradiction. Hence $\phi_1 \cong \phi_2$. The other possibility is $\phi_1 = (\chi \otimes S_1) \oplus (\eta \otimes S_1)$. It is evident that this case can be handled as above to obtain $\phi_1 \cong \phi_2$. 

The proof of the lemma is just going to use arguments similar to the ones in the observation above and Theorem \ref{charac}. We provide the details for completeness. Suppose $\phi_1, \phi_2 \in \G_F^0(4)$. Then $\phi_1 \cong \phi_2$ by Theorem \ref{charac}. Without loss of generality, assume $\phi_1 \notin \G_F^0(4)$. We have the following cases.\\
\textbf{Case} 1: Suppose $\phi_1$ has an irreducible 1-dimensional constituent. As observed in Theorem 10.1 of \cite{GT11}, we have
\[ m_i(\chi,1) = v_i(\chi,1) + v_i(\chi|\;-\;|^{-1},1) - v_i(\chi|\;-\;|^{-1/2},2).\]
Consequently, $m_1(\chi,1) = m_2(\chi,1) \;\forall \;\chi \in \G_F(1)_{2m}$. In particular, since any irreducible constituent of $\phi_1$ also has depth $\leq m$, we see that if $\phi_1 = \phi_{1}^\# \oplus (\chi \otimes S_1),$ then $\phi_2 = \phi_{2}^\# \oplus (\chi \otimes S_1).$ 
Now $\det(\phi_1) = \det(\phi_2)$ and Equation \eqref{Lepsilon} holds for $\phi_1^\#, \phi_2^\# \in \G_F(3)$. Notice that $\phi_1^\#, \phi_2^\# \in \G_F(3) \backslash \G_F^2(3)$, since they do not contain an irreducible 3-dimensional constituent. In particular, $\phi_1$ contains an irreducible 1-dimensional constituent say $\chi \otimes S_1$.  Let $v_i^\#(\rho,r)$ and $m_i^\#(\rho,r)$ denote the corresponding objects for $\phi_i^\#$. Then formula \eqref{polesmult} gives
\[ m_i^\#(\mu, 1) = v_i^\#(\mu,1) +v_i^\#(\mu|\;-\;|^{-1},1) - v_i^\#(\mu|\;-\;|^{-1/2},2).\]
This implies that
\[\phi_1^\# = \phi_1^\circ \oplus (\mu \otimes S_1) \text{ and } \phi_2^\# = \phi_2^\circ \oplus (\mu \otimes S_1).\]
Now, $\det(\phi_1^\circ) = \det(\phi_2^\circ)$ and Equation \eqref{Lepsilon} is satisfied by $\phi_1^\circ$ and $\phi_2^\circ$. Also $\phi_1^\circ, \phi_2^\circ \in \G_F(2)_m$. We are now done by the observation above. \\
\textbf{Case} 2: Assume $\phi_1$ is the sum of two (possibly equivalent) irreducible representations. This case would follow by combining Theorem 10.1, Case 1, of \cite{GT11} and the observation above.\\
\textbf{Case} 3: We are finally reduced to the case when $\phi_1 \in \G_F^2(4)\backslash \G_F^0(4)$. Then again there are two possibilities. The first possibility is  $\phi_1 = \chi \otimes S_4$. Then formula \ref{polesmult} yields
\[ v_i(\chi|\;-\;|^{3/2},1) = m_i(\chi|\;-\;|,2) + m_i(\chi, 4).\]
Now if $\phi_2 \neq \chi \otimes S_4$, then $\chi|\;-\;| \otimes S_2$ occurs as a direct summand of $\phi_2$ and $m_i(\chi|\;-\;|,2) =1$. Consequently,
\[\phi_2 =( \chi|\;-\;|\otimes S_2) + \rho_0\]
with $\rho_0 \neq ( \chi|\;-\;|\otimes S_2)$. Now, $L(s, \phi_1 \otimes \rho_0^\vee)$ is holomorphic at $s=0$ while $L(s, \phi_2 \otimes \rho_0^\vee)$ has a pole at $s=0$, a contradiction. Hence $\phi_1 \cong \phi_2$. The second, and final,  possibility is $\phi_1 = \rho \otimes S_2$ for $\rho \in \G_F^0(2)$. This case follows by arguing as above. \qedhere
\end{proof}

Before stating the re-characterization, let us quickly observe the relation between conductor and depth for semisimple representations of $\WD_F$. Recall that for an irreducible $n$-dimensional representation $\rho$ of $\WD_F$, 
\[\depth(\rho) = \displaystyle{\max\left\{0, \frac{\cond(\rho) - n}{n}\right\}}\]
using Theorem \ref{depth} and Equation \eqref{JosLanDepth}.
Also, if $\rho = \oplus \rho_i$ where $\rho_i$ are irreducible representations of $\WD_F$,
then $\depth(\rho) = \max\{\depth(\rho_i)\}$ and $\cond(\rho) = \sum \cond(\rho_i)$. Hence
\begin{equation}\label{DepthConductorRelation}
\depth(\rho) \leq \cond(\rho) \leq n^2\cdot\depth(\rho) +n^2.
\end{equation}
where $\dim(\rho) = n$.  
We now state the re-characterization of Theorem \ref{char0}. 

\begin{theorem} 
Let $l(m) = 16m+16$. There is a unique surjective finite-to-one map \[L: \Pi(\GSp_4(F)) \longrightarrow \Phi(\GSp_4(F)) \] satisfying:
\begin{enumerate}[(a)]
\item the central character $\omega_\sigma$ of $\sigma$ corresponds to the similitude character $\operatorname{sim}(\phi_\sigma)$ of $\phi_\sigma: = L(\sigma)$ via local class field theory.
\item $\sigma$ is a discrete series representation iff $\phi_\sigma$ does not factor through any proper Levi subgroup of $\GSp_4(\mathbb{C})$.
\item if $\sigma$ is generic or non-supercuspidal and $\depth(\sigma) \leq m$, then for any irreducible representation $\tau$ of $\GL_t(F)$ with $t \leq 2$ and $\depth(\tau) \leq 2l(m)$, 
\[ \begin{cases}
L(s, \sigma \times \tau) &= L(s, \phi_\sigma \otimes \phi_\tau) \\
\gamma(s, \sigma \times \tau, \psi) &= \gamma(s, \phi_\sigma \otimes\phi_ \tau, \psi)
\end{cases} \]
where the factors on the right are of Artin type, and the factors on the left are obtained using the Langlands-Shahidi method.
\item if $\sigma$ is non-generic supercuspidal and $\depth(\sigma) \leq m$, then for any supercuspidal representation $\tau$ of $\GL_t(F)$ with $t \leq 2$ and $\depth(\tau) \leq m$, the Plancherel measure $\mu(s, \sigma \times \tau)$ is equal to 
\begin{align*}
 \gamma(s, \phi_\sigma^\vee \otimes \phi_\tau, \psi) \gamma(-s, \phi_\sigma \otimes \phi_\tau^\vee, \bar{\psi})\gamma(2s, \Sym^2\phi_\tau \otimes& \Sim\phi_\sigma^{-1}, \psi)\\
& \gamma(-2s, \Sym^2\phi_\tau^{\vee} \otimes \Sim\phi_\sigma, \bar{\psi}) 
\end{align*}
\end{enumerate}
\end{theorem}
\begin{proof}
 Let $L_1$ be the Langlands map constructed by Gan-Takeda in \cite{GT11}. Then $L_1$ satisfies all the conditions above. We need to prove uniqueness (because we have added a restriction on the $\depth(\tau)$ in Parts $(c)$ and $(d)$). Let $L_2$ be any other map that satisfies properties $(a)$ - $(d)$ above. We need to prove that $L_1 = L_2$ . Let $\sigma \in \Pi(\GSp_4(F))$ with $\depth(\sigma) \leq m$. Let $\phi_i = L_i(\sigma)$. We need to prove that $\phi_1 \cong \phi_2$. As in \cite{GT11},  it suffices to prove that $\phi_{1}$ and $\phi_{2}$ are equivalent as 4-dimensional representations of $\WD_F$. We have the following cases.\\
\textbf{Case} 1: Suppose $\pi$ is generic or non-supercuspidal. Requirement $(c)$ implies that
\begin{align*}
L(s, \phi_{1} \otimes \phi_{\tau}) &= L(s, \phi_{2} \otimes \phi_\tau),\\\nonumber
\epsilon(s, \phi_{1} \otimes \phi_\tau, \psi) & = \epsilon(s, \phi_{2} \otimes \phi_\tau, \psi).
\end{align*}
Taking $\tau$ to be the trivial representation, we see that $\cond(\phi_1) = \cond(\phi_2)$. Now, we know by Section \ref{DepthPreservation} that $\depth(\phi_1) \leq m$. Hence, by Equation \eqref{DepthConductorRelation}, $\depth(\phi_2) \leq l(m)$. 

Since $\phi_1$ and $\phi_2$ factor through $\GSp_4(\C)$, they do not admit an irreducible 3-dimensional constituent. Moreover, $\det(\phi_1) = \Sim(\phi_1)^2 = \Sim(\phi_2)^2 = \det(\phi_2)$. Hence applying Lemma \ref{conversetheoremrefined} for the integer $l(m)$,  we deduce that $\phi_1 \cong \phi_2$. \\
\textbf{Case} 2: Suppose $\sigma$ is non-generic supercuspidal. In this case, we have $\sigma = \theta(\tau_1^D \otimes \tau_2^D)$ for a representation $\tau_1^D \otimes \tau_2^D$ of $\GSO(D)$ (see Case 2, Theorem 10.1 of \cite{GT11}). If $\tau_i$ denotes the Jacquet-Langlands lift of $\tau_i^D$, then $\phi_1 = \phi_{\tau_1} \oplus \phi_{\tau_2}.$
Furthermore, $\depth(\phi_1) \leq m$. We need to prove that $\phi_1 \cong \phi_2$. Note that in this case, we don't yet know that $\depth(\phi_1) = \depth(\phi_2)$. But, as it is clear from the proof of Theorem 10.1, Case 2  in \cite{GT11}, we only need condition (d) to be satisfied for those $\phi_\tau$ with $\depth(\phi_\tau) \leq m$.\qedhere
\end{proof}
\begin{Remark} Note that the above theorem also implies that the map \[L_m: \Pi(\GSp_4(F)_m \rightarrow \Phi(\GSp_4(F))_m\] also satisfies Properties $(a)$ - $(d)$ above and is uniquely characterized by these properties.

\end{Remark}

\subsection{The main theorem}\label{LLCGSP4CLF}
Let ${F' = \F_{p^n}((t))}$ be a non-archimedean local field of positive characteristic with $p >2$. 

\begin{theorem}\label{LLCGSP4} Let $l(m) = 16m+16$. 
There exists a unique surjective finite-to-one map
\[L_m': \Pi(\GSp_4({F'}))_m \rightarrow \Phi(\GSp_4({F'}))_m\]
satisfying the following properties:
\begin{enumerate}[(a)]
\item the central character $\omega_{\sigma'}$ of ${\sigma'}$ corresponds to the similitude character $\operatorname{sim}(\phi_{\sigma'})$ of $\phi_{\sigma'}: = L_m'({\sigma'})$ via local class field theory.
\item ${\sigma'}$ is a discrete series representation iff $\phi_{\sigma'}$ does not factor through any proper Levi subgroup of $\GSp_4(\mathbb{C})$.
\item if ${\sigma'}$ is generic or non-supercuspidal, then for any irreducible representation ${\tau'}$ of $\GL_t({F'})$ with $t \leq 2$ and $\depth({\tau'}) \leq 2l(m)$, 
\[ \begin{cases}
L(s, {\sigma'} \times {\tau'}) &= L(s, \phi_{\sigma'} \otimes \phi_{\tau'}) \\
\gamma(s, {\sigma'} \times {\tau'}, \psi') &= \gamma(s, \phi_{\sigma'} \otimes\phi_ {\tau'}, \psi')
\end{cases} \]
where the factors on the right are of Artin type, and the factors on the left are obtained using the Langlands-Shahidi method.
\item if ${\sigma'}$ is non-generic supercuspidal, then for any supercuspidal representation ${\tau'}$ of $\GL_t({F'})$ with $t \leq 2$ and $\depth({\tau'}) \leq m$, the Plancherel measure $\mu(s, {\sigma'} \times {\tau'})$ is equal to 
\begin{align*}
 \gamma(s, \phi_{\sigma'}^\vee \otimes \phi_{\tau'}, \psi') \gamma(-s, \phi_{\sigma'} \otimes \phi_{\tau'}^\vee, \bar{\psi'})\gamma(2s, \Sym^2\phi_{\tau'} \otimes& \Sim\phi_{\sigma'}^{-1}, \psi')\\
& \gamma(-2s, \Sym^2\phi_{\tau'}^{\vee} \otimes \Sim\phi_{\sigma'}, \bar{\psi'}) 
\end{align*}
\end{enumerate}
\end{theorem}
\begin{proof} The map $L_m'$ is to be defined using the Kazhdan isomorphism, the Langlands map $L_m$ in characteristic 0, and the Deligne isomorphism. Let $l$ be a large enough integer such that Theorems  \ref{gammafactorCLF}, \ref{LEpsilonCLF} and  \ref{PlancherelCLF} hold for $\sigma \in \Pi(\GSp_4(F'))_m$ and $\tau \in \G_{F'}(r)_{2l(m)}, r = 1,2$, and Theorem \ref{GLNDK} holds for $\tau \in \G_{F'}(r)_{2l(m)}, r=1,2$. Note that $l$ depends only on $m$. Let $F$ be a local field of characteristic 0 such that $F'$ is $(l+1)$-close to $F$. 

 In Section \ref{DepthPreservation}, we proved that the map $L_m$ preserves depth. Let $\kappa_l^1$ be the bijection
 \begin{align}
 \{\text{Iso. classes of irr. representations } &(\sigma, V)  \text{ of  $\GSp_4(F)$ with $\sigma^{I_l} \neq 0$\}}\nonumber \\
  & \updownarrow{\kappa_l^1}\\
\{\text{Iso. classes of irr. representations }  &(\sigma', V')   \text{ of $\GSp_4({F'})$ with $\sigma'^{I_l'} \neq 0$\}}.\nonumber
\end{align}
 Suppose $\depth(\sigma) = s$. Then there is $x \in \bar{\calc}$ (see Section \ref{LLCGLNCLF}) such that $V^{G_{x,s+}} \neq 0$. By Lemma \ref{MPUC}, we have $I_{\lceil s\rceil+1} \subset G_{x, s+} \subset I$.  Recall that $I/I_{\lceil s\rceil+1} \cong I'/I_{\lceil s\rceil+1}'$ if $F$ and $F'$ are $(\lceil s\rceil+1)$-close by Section \ref{IGS}. Now it follows that $\kappa_{l+1}^1:\Pi(\GSp_4(F))_l \rightarrow \Pi(\GSp_4({F'}))_l$ is a well-defined  bijection since $F$ and $F'$ are $(l+1)$-close.
 
Next, let $\Del_l$ denote the bijection
 \begin{align}
\{ \text{Homomorphisms } \phi:\WD_F \rightarrow& \GSp_4(\C) \text{ such that $I_F^l \subset \ker(\phi)$\}}\nonumber \\
  & \updownarrow{\Del_l}\\
\{\text{Homomorphisms } \phi':\WD_{F'} \rightarrow& \GSp_4(\C) \text{ such that $I_{F'}^l \subset \ker(\phi')$\}}.\nonumber
\end{align}
 
It is clear that $\Del_{l+1}: \Phi(\GSp_4(F))_l \rightarrow \Phi(\GSp_4({F'}))_l$ is a well-defined bijection when $F$ and $F'$ are $(l+1)$-close. 
Let $L_m: \Pi(\GSp_4(F))_m \rightarrow \Phi(\GSp_4(F))_m$ be the map in characteristic 0 established by Gan-Takeda as explained in the previous section.

We define $L_m'$ so that the following diagram is commutative:
\begin{displaymath}
    \xymatrix{
        \Pi(\GSp_4({F}))_m \ar[r]^{L_{m}} \ar[d]_{\kappa^1_{l+1}} & \Phi(\GSp_4({F}))_m \ar[d]^{\Del_{l+1}} \\
        \Pi(\GSp_4(F'))_m \ar[r]_{L_m'}      &\Phi(\GSp_4(F'))_m}
\end{displaymath}
It is clear that $L_m'$ satisfies Property $(a)$. By Theorem \ref{SCDS}, $L_m'$ has property $(b)$. To prove property (c), let $\kappa^1_{l+1}(\sigma) = \sigma'$ and $\kappa_{l+1}(\tau) = \tau'$ (with $\kappa_{l+1}$ as in  Section \ref{LLCGLNCLF}). Note that by Theorem \ref{GLNDK} and Corollary \ref{corGLNDK}, 
\begin{align}\label{LLCGLNew}
\phi_{\tau'} = \Del_{l+1}(\phi_\tau).
\end{align}
Now,
\begin{align*}
\gamma(s, \sigma' \times \tau', \psi')& = \gamma(s, \sigma \times \tau, \psi)\; \text{( By Proposition \ref{gammafactorCLF})}\\
& = \gamma(s, L_m(\sigma) \otimes \phi_\tau, \psi)\; \text{ (By Theorem \ref{char0}) }\\
& = \gamma(s, (\Del_{l+1} \circ L_m(\sigma)) \otimes \Del_{l+1}(\phi_\tau), \psi')\; \text{ (By (iv) of Section \ref{Delignetheory})}\\
& = \gamma(s, L_m'(\sigma') \otimes \phi_{\tau'}, \psi')\; \text{(By definition of $L_m'$ and Equation \eqref{LLCGLNew})}.
\end{align*}
Similarly, combine Proposition \ref{LEpsilonCLF}, Theorem \ref{char0}, Property (iv) of Section  \ref{Delignetheory}, and Equation \eqref{LLCGLNew} to get the equality of $L$-factors.  Property $(d)$ is a consequence of Theorem \ref{PlancherelCLF}, Property (iv) of Section \ref{Delignetheory}, and Equation \eqref{LLCGLNew}. 

Note that the map $L_m'$ does not depend on the choice of this field of characteristic 0 for the following reason. Suppose $F_1$ is another local field of characteristic 0 that is $(l+1)$-close to $F$. Consider the following diagram.
\begin{displaymath}
    \xymatrix{
        \Pi(\GSp_4({F_1}))_m \ar[r]^{L_{m,1}} \ar[d]_{\kappa^1_{l+1}} & \Phi(\GSp_4({F_1}))_m \ar[d]^{\Del_{l+1}} \\
        \Pi(\GSp_4(F))_m \ar[r]_{L_m}      &\Phi(\GSp_4(F))_m}
\end{displaymath}
Here $L_{m,1}$ and $L_m$ are the maps defined by Gan-Takeda in characteristic 0. Note that this diagram is commutative, since $\Del_{l+1}^{-1} \circ L_m \circ \kappa^1_{l+1}: \Pi(\GSp_4({F_1}))_m \rightarrow \Phi(\GSp_4({F_1}))_m$ is another surjective finite-to-one map that satisfies the same list of properties as $L_{m,1}$. Since $L_{m,1}$ is uniquely characterized by these properties, we deduce that the above diagram is commutative. 

The uniqueness of the map $L_m'$ is clear because the corresponding map $L_m$ in characteristic 0 is the unique map satisfying Properties $(a)$  - $(d)$. 
\end{proof}

\begin{Corollary}
There exists a unique surjective finite-to-one map
\[L': \Pi(\GSp_4({F'})) \rightarrow \Phi(\GSp_4({F'}))\]
satisfying the following properties:
\begin{enumerate}[(a)]
\item the central character $\omega_{\sigma'}$ of ${\sigma'}$ corresponds to the similitude character $\operatorname{sim}(\phi_{\sigma'})$ of $\phi_{\sigma'}: = L'({\sigma'})$ via local class field theory.
\item ${\sigma'}$ is a discrete series representation iff $\phi_{\sigma'}$ does not factor through any proper Levi subgroup of $\GSp_4(\mathbb{C})$.
\item if ${\sigma'}$ is generic or non-supercuspidal and $\depth(\sigma')\leq m$, then for any irreducible representation ${\tau'}$ of $\GL_t({F'})$ with $t \leq 2$ and $\depth({\tau'}) \leq 2l(m)$, 
\[ \begin{cases}
L(s, {\sigma'} \times {\tau'}) &= L(s, \phi_{\sigma'} \otimes \phi_{\tau'}) \\
\gamma(s, {\sigma'} \times {\tau'}, \psi') &= \gamma(s, \phi_{\sigma'} \otimes\phi_ {\tau'}, \psi')
\end{cases} \]
where the factors on the right are of Artin type, and the factors on the left are obtained using the Langlands-Shahidi method.
\item if ${\sigma'}$ is non-generic supercuspidal with $\depth(\sigma')\leq m$, then for any supercuspidal representation ${\tau'}$ of $\GL_t({F'})$ with $t \leq 2$ and $\depth({\tau'}) \leq m$, the Plancherel measure $\mu(s, {\sigma'} \times {\tau'})$ is equal to 
\begin{align*}
 \gamma(s, \phi_{\sigma'}^\vee \otimes \phi_{\tau'}, \psi') \gamma(-s, \phi_{\sigma'} \otimes \phi_{\tau'}^\vee, \bar{\psi'})\gamma(2s, \Sym^2\phi_{\tau'} \otimes& \Sim\phi_{\sigma'}^{-1}, \psi')\\
& \gamma(-2s, \Sym^2\phi_{\tau'}^{\vee} \otimes \Sim\phi_{\sigma'}, \bar{\psi'}) 
\end{align*}
\end{enumerate}
\end{Corollary}
\begin{proof}
For $\sigma' \in \Pi(\GSp_4({F'}))$, suppose $\depth(\sigma') \leq m$. Define $L'(\sigma') = L_m'(\sigma')$. We need to ensure that the definition of $L'$ does not depend on $m$. So suppose $m \leq n$. Then $L_n'|_{{\Pi(\GSp_4(F'))}_m}$ satisfies Properties $(a)$  - $(d)$ of Theorem \ref{LLCGSP4}. Hence $L_n'|_{{\Pi(\GSp_4(F'))}_m} = L_m'$ and the corollary follows. 
\end{proof}
\bibliographystyle{plain}

\endgroup
\end{document}